\documentclass[11pt,oneside,english]{amsart}
\usepackage[latin9]{inputenc}
\usepackage{babel}
\usepackage{array}
\usepackage{float}
\usepackage{bm}
\usepackage{multirow}
\usepackage{amsbsy}
\usepackage{amstext}
\usepackage{amsthm}
\usepackage{amssymb}
\usepackage{stmaryrd}
\usepackage{graphicx}
\usepackage[unicode=true,pdfusetitle,
 bookmarks=true,bookmarksnumbered=false,bookmarksopen=false,
 breaklinks=false,pdfborder={0 0 1},backref=false,colorlinks=false]
 {hyperref}

\makeatletter

\providecommand{\tabularnewline}{\\}

\numberwithin{equation}{section}
\numberwithin{figure}{section}
\theoremstyle{plain}
\newtheorem{thm}{\protect\theoremname}[section]
\theoremstyle{definition}
\newtheorem{defn}[thm]{\protect\definitionname}
\theoremstyle{plain}
\newtheorem{prop}[thm]{\protect\propositionname}
\theoremstyle{remark}
\newtheorem{rem}[thm]{\protect\remarkname}
\theoremstyle{definition}
\newtheorem{example}[thm]{\protect\examplename}
\theoremstyle{plain}
\newtheorem{fact}[thm]{\protect\factname}
\theoremstyle{plain}
\newtheorem{lem}[thm]{\protect\lemmaname}
\theoremstyle{plain}
\newtheorem{cor}[thm]{\protect\corollaryname}

\usepackage{amsfonts}
\usepackage{amsthm}
\usepackage{cite}
\usepackage{url}
\usepackage{pb-diagram}

\usepackage{enumitem}

\makeatother

\providecommand{\corollaryname}{Corollary}
\providecommand{\definitionname}{Definition}
\providecommand{\examplename}{Example}
\providecommand{\factname}{Fact}
\providecommand{\lemmaname}{Lemma}
\providecommand{\propositionname}{Proposition}
\providecommand{\remarkname}{Remark}
\providecommand{\theoremname}{Theorem}

\begin{document}
\title{Twisted Moduli Spaces and Duistermaat-Heckman Measures}
\address{University of Arizona, Department of Mathematics,
617 N Santa Rita Ave, Tucson, AZ 85721, United States }
\email{ajzerouali@math.arizona.edu}
\author{Ahmed J. Zerouali}
\begin{abstract}
Following Boalch-Yamakawa and Meinrenken, we consider a certain class
of moduli spaces on bordered surfaces from a quasi-Hamiltonian perspective.
For a given Lie group $G$, these character varieties parametrize
flat $G$-connections on ``twisted'' local systems, in the sense
that the transition functions take values in $G\rtimes\mathrm{Aut}(G)$.
After reviewing the necessary tools to discuss twisted quasi-Hamiltonian
manifolds, we construct a Duistermaat-Heckman (DH) measure on $G$
that is invariant under the twisted conjugation action $g\mapsto hg\kappa(h^{-1})$
for $\kappa\in\mathrm{Aut}(G)$, and characterize it by giving a localization
formula for its Fourier coefficients. We then illustrate our results
by determining the DH measures of our twisted moduli spaces.
\end{abstract}

\maketitle
\setcounter{section}{-1}

\section{Introduction}

The theory of quasi-Hamiltonian manifolds was initiated by Alekseev,
Malkin and Meinrenken in \cite{[AMM98]}, and one of its fundamental
results is the bijective correspondence between such spaces and Banach
manifolds with Hamiltonian actions of loop groups. On the one hand,
this theory allows one to generalize several core aspects of symplectic
geometry to the loop group setting (e.g. convexity properties, Duistermaat-Heckman
distributions \cite{[AMW00], [AMW02], [Mein05]} and geometric quantization
\cite{[Krep08], [Mein12], [Sha06a]}), and on the other hand, it provides
a finite-dimensional approach to the study of various types of moduli
spaces of flat bundles over surfaces with boundary\cite{[MeinNoDaNotes], [Boa07], [LBS15a]}.

The present paper is concerned with Hamiltonian actions of \textit{twisted}
loop groups, for a twist given by a Lie group automorphism $\kappa\in\mathrm{Aut}(G)$,
and the corresponding twisted quasi-Hamiltonian $G$-spaces \cite{[LMS17], [Mein17], [BY15], [Kno16]}.
In the latter case, the moment map takes values in a Lie group $G$,
and is equivariant with respect to the $\kappa$-twisted conjugation
action of the group on itself:
\[
\mathrm{Ad}_{g}^{\kappa}(h)=gh\kappa(g^{-1}),\ \ \forall g,h\in G,
\]
Basic examples of twisted quasi-Hamiltonian $G$-spaces include orbits
of the action $\mathrm{Ad}_{G}^{\kappa}$, as well as twisted $G$-character
varieties \cite{[Mein17], [BY15]}.

The first objective of this work is to construct and compute the Duistermaat-Heckman
(DH) measure of a twisted quasi-Hamiltonian $G$-space $M$ (tq-Hamiltonian
in short). As in the standard theory \cite{[AMW00], [AMW02], [Mein05]},
the DH measure of $M$ encodes the volumes of its symplectic quotients.
The main feature of the setup at hand is that the Fourier expansion
is given in terms of twining characters \cite{[FSS96], [Hon17b], [Zer18a]},
the class functions on $G$ with respect to twisted conjugation, which
were introduced by Fuchs, Schellekens and Schweigert in \cite{[FSS96]}.
The main result of this work is a localization formula for the Fourier
coefficients of the DH measures.

The second objective of this article is to determine the DH measures
explicitly for twisted $G$-character varieties \cite{[Mein17]}.
Similar moduli spaces were studied from the algebro-geometric perspective
\cite{[HK18], [MW09]}, in view of obtaining a generalization of the
Verlinde formula that incorporates the Dynkin diagram automorphisms
of the group involved. Following the paradigm of Alekseev-Meinrenken-Woodward
\cite{[MW98], [Mein12], [AMW01]}, a symplectic approach to deriving
such a Verlinde formula would be to compute the quantization of the
reduced spaces of twisted moduli spaces, which requires a K-theoretic
analogue of the localization formula obtained here. Another motivation
for the determination of the DH measures of twisted moduli spaces
is the study of the cohomology of their symplectic quotients, by generalizing
Witten's formulas for intersection pairings to the present setting,
and following \cite{[Wit91], [Wit92], [Loi17], [Mein05]}.

The contents of this paper are organized as follows. After reviewing
the basics of tq-Hamiltonian geometry, section \ref{Sec_tq-Ham} gives
a detailed discussion of twisted moduli spaces. The Duistermaat-Heckman
measure of tq-Hamiltonian manifold and the localization theorem are
adressed in section \ref{Sec_DH_Meas}, which ends with the computation
of DH measures of twisted moduli spaces. The first set of prerequisites
are some properties of twisted conjugation, which are summarized at
the beginning of sections \ref{Sec_tq-Ham} and \ref{Sec_DH_Meas}.
The second set of prerequisites are the Dirac geometry techniques
used in Alekseev, Bursztyn and Meinrenken's \cite{[AMB09]}. Section
\ref{App_Dirac_Geo} gathers some key results of that work in the
context of twisted conjugation.

\subsection*{Acknowledgements}

I thank my thesis advisor, Prof. Meinrenken, for introducing me to
this subject, as well as for his guidance and advice. I also thank
Professors Gualtieri and Jeffrey for their feedback and suggestions,
and Dr. Loizides for discussion. This work was partially supported
by FRQNT, OGS and Queen Elizabeth II scholarships.

\tableofcontents{}

\section{Twisted quasi-Hamiltonian manifolds\label{Sec_tq-Ham}}

In this section, we establish the geometric setup of the present work.
We start by recalling some basics of the twisted conjugation action
of a Lie group on itself in \ref{SubSec_Twist_Conj}. Next, we review
the definition of twisted quasi-Hamiltonian manifolds in \ref{SubSec_tq-Ham_mflds},
the first examples thereof, and the fundamental operations on these
spaces. These preliminaries are then illustrated in section \ref{SubSec_Moduli},
where we discuss a class of moduli spaces of flat torsors on bordered
surfaces.

\subsection{Twisted conjugation\label{SubSec_Twist_Conj}}

\subsubsection{Notation}

For a Lie group $G$ with (fixed) maximal torus $T$, we denote the
corresponding Lie algebras by $\mathfrak{g}$ and $\mathfrak{t}$.
We denote the group of automorphisms of $G$ by $\mathrm{Aut}(G)$,
and its normal subgroup of inner automorphisms by $\mathrm{Inn}(G)$.
The group of outer automorphisms $\mathrm{Out}(G)$ of $G$ is defined
by the exact sequence: 
\begin{equation}
1\to\mathrm{Inn}(G)\to\mathrm{Aut}(G)\to\mathrm{Out}(G)\to1.\label{Eq_Aut_ExctSequ}
\end{equation}
We assume that the Lie algebra $\mathfrak{g}$ is endowed with an
$\mathrm{Aut}(G)$-invariant, symmetric, and non-degenerate bilinear
form $B\in S^{2}\mathfrak{g}^{\ast}$, and we often write $\xi\cdot\zeta=B(\xi,\zeta)$
for the inner product of elements $\xi,\zeta\in\mathfrak{g}$.

We denote the set of roots of $G$ by $\mathfrak{R}\subseteq\mathfrak{t}^{\ast}$,
and the Weyl group by $W=N_{G}(T)/T$. We use real roots below: for
$\xi\in\mathfrak{t}$ and $t=e^{\xi}\in T$, we write $t^{\alpha}=e^{2\pi\mathsf{i}\langle\alpha,\xi\rangle}$
for the corresponding character. We fix a choice of positive roots
$\mathfrak{R}_{+}\subseteq\mathfrak{R}$, and along with the identification
$\mathfrak{t}\cong\mathfrak{t}^{\ast}$ given by the invariant inner
product $B$ on $\mathfrak{g}$, we denote the closed fundamental
Weyl chamber by $\mathfrak{t}_{+}^{\ast}\subseteq\mathfrak{t}^{\ast}$.
For $G$ simply connected, we denote the weight lattice by $\Lambda^{\ast}=\mathrm{Hom}(\Lambda,\mathbb{Z})$,
where $\Lambda=\ker(\exp)\cap\mathfrak{t}$ is the integral lattice,
and we denote the dominant weights by $\Lambda_{+}^{\ast}=\Lambda^{\ast}\cap\mathfrak{t}_{+}^{\ast}$.

\subsubsection{The twisted conjugation action}
\begin{defn}
Let $\kappa\in\mathrm{Aut}(G)$ be an automorphism of $G$. The $\bm{\kappa}$\textbf{-twisted
conjugation action} of $G$ on itself is given by: 
\[
\mathrm{Ad}_{g}^{\kappa}(h)=gh\kappa(g^{-1}),\ \ \forall g,h\in G.
\]
We denote by $G\kappa$ the group $G$ endowed with the action $\mathrm{Ad}_{G}^{\kappa}$.
The orbits of this action are called the $\bm{\kappa}$\textbf{-twisted
conjugacy classes} of $G$. 
\end{defn}

In terms of the semi-direct product $G\rtimes\mathrm{Aut}(G)$, twisted
conjugation corresponds to the usual conjugation action of the subgroup
$G$ on the component $G\kappa\subseteq G\rtimes\mathrm{Aut}(G)$:
\[
(g,1)(h,\kappa)(g,1)^{-1}=\left(\mathrm{Ad}_{g}^{\kappa}(h),\kappa\right).
\]
For any automorphism $\tau=\mathrm{Ad}_{a}\circ\kappa$ with $a\in G$,
the $\kappa$-twisted and $\tau$-twisted conjugacy classes are related
by the right multiplication map $R_{a}:x\mapsto xa^{-1}$, since:
\[
R_{a^{-1}}\circ\mathrm{Ad}_{g}^{(\mathrm{Ad}_{a}\circ\kappa)}=\mathrm{Ad}_{g}^{\kappa}\circ R_{a^{-1}},\ \ \forall g\in G.
\]
Thus, to study the twisted conjugation action, it is sufficient to
consider the case where $\kappa\in\mathrm{Aut}(G)$ is induced by
a Dynkin diagram automorphism.

For the remainder of this subsection, we focus on the case of $G$
compact 1-connected and simple. This setup will be of particular importance
later (Proposition \ref{Prop_Isom_Twists} and section \ref{Sec_DH_Meas}),
and several facts are worth mentioning. For more details, see \cite[\S\S 2-3]{[Zer18a]}
and references therein.

Any automorphism $\kappa\in\mathrm{Aut}(\Pi)$ of the Dynkin diagram
of $G$ permutes the simple roots $\Pi$ of $\mathfrak{R}$ and preserves
its Cartan matrix. If $\{e_{\pm\alpha}\}_{\alpha\in\mathfrak{R}_{+}}$
denote the Chevalley generators of $\mathfrak{g}_{\mathbb{C}}$, there
then exists a unique corresponding $\kappa\in\mathrm{Aut}(\mathfrak{g}_{\mathbb{C}})$,
such that $\kappa(e_{\pm\alpha})=e_{\pm\kappa(\alpha)}$ for all $\alpha\in\Pi$.
This $\kappa\in\mathrm{Aut}(\mathfrak{g}_{\mathbb{C}})$ preserves
$\mathfrak{g}\subset\mathfrak{g}_{\mathbb{C}}$, $\mathfrak{t}\subset\mathfrak{t}_{\mathbb{C}}$
and $\mathfrak{t}_{+}^{\ast}\subset\mathfrak{t}^{\ast}$, and exponentiates
to an element $\kappa\in\mathrm{Aut}(G)$ preserving the maximal torus
$T$. We hence obtain a homomorphism $\mathrm{Aut}(\Pi)\to\mathrm{Aut}(G)$
descending to $\mathrm{Out}(G)$, and on top of the identification
$\mathrm{Out}(G)\equiv\mathrm{Aut}(\Pi)$, we obtain a splitting of
the exact sequence (\ref{Eq_Aut_ExctSequ}). As such, when $G$ is
compact 1-connected and simple, we will view $\mathrm{Out}(G)$ \textit{as
a subgroup of }$\mathrm{Aut}(G)$.

\begin{table}[H]
\begin{center}%
\begin{tabular}{|c|c|c|c|c|c|c|}
\hline 
$\mathfrak{\ensuremath{R}}$ & $A_{2}$ & $A_{2n}\mbox{, }n\ge2$ & $A_{2n-1},n\ge2$ & $D_{n+1}\mbox{, }n\ge4$ & $D_{4}$ & $E_{6}$\tabularnewline
\hline 
\hline 
$\mathrm{Out}(G)$ & $\mathbb{Z}_{2}$ & $\mathbb{Z}_{2}$ & $\mathbb{Z}_{2}$ & $\mathbb{Z}_{2}$ & $S_{3}$ & $\mathbb{Z}_{2}$\tabularnewline
\hline 
\multirow{2}{*}{$\mathfrak{R}_{(\kappa)}$} & \multirow{2}{*}{$A_{1}$} & \multirow{2}{*}{$C_{n}$} & \multirow{2}{*}{$B_{n}$} & \multirow{2}{*}{$B_{n}$} & $G_{2}$ if $|\kappa|=3$ & \multirow{2}{*}{$F_{4}$}\tabularnewline
\cline{6-6} 
 &  &  &  &  & $B_{3}$ if $|\kappa|=2$ & \tabularnewline
\hline 
\end{tabular}\end{center}\caption{Non-trivial $\mathrm{Out}(G)$ and orbit root systems}
\label{Table_OutG_Orbit}
\end{table}

We now introduce the objects that control the $\kappa$-twisted conjugacy
classes of $G$. Let $G^{\kappa}$ and its maximal torus $T^{\kappa}$
denote the $\kappa$-fixed subgroups of $G$ and $T$, and let $\mathfrak{g}^{\kappa}$
and $\mathfrak{t}^{\kappa}$ denote the corresponding Lie algebras.
Let $T_{\kappa}\subseteq T$ denote the image of of the homomorphism
$T\to T$, $t\mapsto t\kappa(t^{-1})$, and let $\mathfrak{t}_{\kappa}=\mathrm{Lie}(T_{\kappa})$.
As for usual conjugation, we define the \textit{outer Weyl group}
$W^{(\kappa)}=N_{G}^{\kappa}(T^{\kappa})/T^{\kappa}$, where $N_{G}^{\kappa}(T^{\kappa})=\{g\in G\ |\ \mathrm{Ad}_{g}^{\kappa}(T^{\kappa})=T^{\kappa}\}$.
We denote by $W^{\kappa}\subseteq W$ the subgroup of elements commuting
with $\kappa$.

For non-trivial $\mathrm{Out}(G)$ (Table \ref{Table_OutG_Orbit}),
one associates the \textit{orbit root system} $\mathfrak{R}_{(\kappa)}$
to the data $(\mathfrak{R},\kappa)$. Let $p:\mathfrak{t}^{\ast}\to(\mathfrak{t}^{\ast})^{\kappa}$
denote the orthogonal projection. For $\mathfrak{R}\ne A_{2n}$, the
projection $p(\mathfrak{R})$ is an indecomposable root system, and
its dual is equivalent to $\mathfrak{R}_{(\kappa)}$. For $\mathfrak{R}=A_{2n}$
with $n>1$, the orbit root system $\mathfrak{R}_{(\kappa)}$ is equivalent
to the dual of the $B_{n}$ subsystem of $p(\mathfrak{R})=BC_{n}$.
For $\mathfrak{\mathfrak{R}}=A_{2}$, $\mathfrak{R}_{(\kappa)}=A_{1}$
is given by twice the highest root of $\mathfrak{R}$. We call the
compact simply connected form of $\mathfrak{R}_{(\kappa)}$ the \textit{orbit
Lie group}, and denote it by $G_{(\kappa)}$. Note that the Weyl group
of $\mathfrak{R}_{(\kappa)}$ coincides with $W^{\kappa}$.

With the preliminaries of the last two paragraphs, we have:
\begin{prop}
Let $G$ be compact 1-connected and simple with nontrivial $\kappa\in\mathrm{Out}(G)$.\begin{enumerate}

\item One has that $T=T^{\kappa}\cdot T_{\kappa}$, that $\mathfrak{t}=\mathfrak{t}^{\kappa}\oplus\mathfrak{t}_{\kappa}$,
and that:
\[
T^{\kappa}\cap T_{\kappa}\simeq\begin{cases}
(\mathbb{Z}_{2})^{\dim\mathfrak{t}_{\kappa}}, & |\kappa|=2;\\
\mathbb{Z}_{3}, & |\kappa|=3.
\end{cases}
\]

\item One has that $W^{(\kappa)}\simeq(T^{\kappa}\cap T_{\kappa})\rtimes W^{\kappa}$,
and that any $\kappa$-twisted conjugacy class in $G$ intersects
$T^{\kappa}$ in an orbit of $W^{(\kappa)}$.

\item The orbit Lie group is related to $G$ by the isomorphism $^{L}G_{(\kappa)}\simeq(^{L}G)_{0}^{\kappa}$,
and its maximal torus is isomorphic to $T^{\kappa}/(T^{\kappa}\cap T_{\kappa})$
($\phantom{.}^{L}K$ denotes the Langlands dual of $K$, and $K_{0}\subseteq K$
the identity component).

\item The weight lattice of $G_{(\kappa)}$ coincides with $(\Lambda^{\ast})^{\kappa}=\Lambda^{\ast}\cap(\mathfrak{t^{\ast}})^{\kappa}$,
the lattice of $\kappa$-fixed weights of $G$. The integral lattice
of $G_{(\kappa)}$ coincides with the lattice $\exp_{\mathfrak{t}^{\kappa}}^{-1}(T^{\kappa}\cap T_{\kappa})\subset\mathfrak{t^{\kappa}}$.

\item The $\kappa$-twisted conjugacy classes of $G$ are parametrized
by the fundamental alcove of the orbit group $G_{(\kappa)}$.

\end{enumerate}
\end{prop}

In this proposition, statement (3) is discussed for $G$ complex in
\cite{[KLP09]}. Parts (1), (2) and (4) are contained in Lemma 3.5,
Theorem 3.6, Proposition 3.9 and Proposition 2.4 in \cite{[Zer18a]}
(see references therein for the original authors of these results).
For statement (5), we note that the affine Weyl group $\Lambda\rtimes W$
is replaced by $\exp_{\mathfrak{t}^{\kappa}}^{-1}(T^{\kappa}\cap T_{\kappa})\rtimes W^{\kappa}$
when dealing with twisted conjugation. The orbit space $G/\mathrm{Ad}_{G}^{\kappa}$
identifies with $T^{\kappa}/W^{(\kappa)}\simeq\mathfrak{t}^{\kappa}/\left(\exp_{\mathfrak{t}^{\kappa}}^{-1}(T^{\kappa}\cap T_{\kappa})\rtimes W^{\kappa}\right)$,
which is the fundamental alcove of $G_{(\kappa)}$ by statements (2)-(4).
An explicit description of this alcove is given in \cite[\S 3]{[MW04]}
(c.f. \cite[Prop.3.10]{[Zer18a]}).

\subsection{Quasi-Hamiltonian geometry\label{SubSec_tq-Ham_mflds}}

\subsubsection{Definition and first examples}

Let $\theta^{L},\theta^{R}\in\Omega^{1}(G,\mathfrak{g})$ denote the
left- and right-invariant Maurer-Cartan 1-forms respectively, and
let $\eta=\tfrac{1}{12}[\theta^{L},\theta^{L}]\cdot\theta^{L}=\tfrac{1}{12}[\theta^{R},\theta^{R}]\cdot\theta^{R}$
be the Cartan 3-form. We have the following definition: 
\begin{defn}
\label{Def_tq-Ham_Mfld}Let $G$ be a Lie group, and let $\kappa\in\mathrm{Aut}(G)$
be an automorphism. A \textbf{twisted quasi-Hamiltonian $\boldsymbol{G}$-space}
\cite{[BY15],[Mein17],[Kno16]} (tq-Hamiltonian in brief) is a triple
$(M,\omega,\Phi)$ where $M$ is a $G$-manifold, $\omega\in\Omega^{2}(M)^{G}$
an invariant 2-form, and $\Phi\colon M\to G\kappa$ is the \textbf{group-valued
moment map} satisfying the following conditions: 
\begin{itemize}
\item[(i)] \textit{Equivariance}: $\Phi(g\cdot x)=\mathrm{Ad}_{g}^{\kappa}\left(\Phi(x)\right)$,
for all $g\in G$ and $x\in M$;
\item[(ii)] \textit{Differential}: $d\omega=\Phi^{\ast}\eta$;
\item[(iii)] \textit{Moment map condition}: $\iota_{\xi_{M}}\omega=\frac{1}{2}\Phi^{\ast}\left(\theta^{L}\cdot\kappa(\xi)+\theta^{R}\cdot\xi\right)$,
for all $\xi\in\mathfrak{g}$;
\item[(iv)] \textit{Minimal degeneracy}: $\ker\omega_{x}\cap\ker(\Phi_{\ast}|_{x})=\{0\}$
for all $x\in M$.
\end{itemize}
\end{defn}

Before addressing the basic examples, it will be useful to highlight
a few observations. 
\begin{rem}
\label{Rks_Def_tq-Ham_mflds}\noindent\begin{enumerate}

\item Taking $\kappa=1$ above, we recover the original definition
\cite{[AMM98]} of a quasi-Hamiltonian manifold (up to a sign convention).
Any tq-Hamiltonian $G$-space with moment map $\Phi\colon M\to G\kappa$
can be viewed as an untwisted q-Hamiltonian space with $G\rtimes\langle\kappa\rangle$-valued
moment map, and we only require that the subgroup $G\subseteq G\rtimes\langle\kappa\rangle$
acts on $M$. Due to this observation, many results of q-Hamiltonian
geometry carry over directly to the twisted case.

\item To be specific about the target of the moment map of $(M,\omega,\Phi)$,
we will often refer to $(M,\omega,\Phi)$ as a $G\kappa$\textit{-valued
tq-Hamiltonian space}.

\item Definition \ref{Def_tq-Ham_Mfld} can be reformulated succintly
in the language of Dirac geometry. The triple $(M,\omega,\Phi)$ is
a tq-Hamiltonian $G$-space if $(\Phi,\omega):(M,TM,0)\dashrightarrow(G\kappa,E_{G}^{\kappa},\eta)$
is a strong $G$-equivariant Dirac morphism \cite[Def.2.4]{[AMB09]},
where $E_{G}^{\kappa}\subseteq\mathbb{T}G$ denotes the twisted Cartan-Dirac
structure \cite[Rk.3.5]{[Mein17]}. We review this viewpoint in section
\ref{Subsec_Dirac_Geo}.

\item There exists an equivalence of categories between $G\kappa$-valued
tq-Hamiltonian manifolds and weakly symplectic Banach manifolds equipped
with a Hamiltonian action of the $\kappa$-twisted loop group $L^{(\kappa)}G=\{\gamma:\mathbb{R}\to G\ |\ \gamma(t+1)=\kappa\left(\gamma(t)\right)\}$.
This is an extension of the equivalence theorem \cite[Thm.8.3]{[AMM98]}
to the twisted setup, and is explained in \cite[Thm.2.4]{[Kno16]}.
This equivalence of categories can be interepreted as a certain Morita
equivalence of pre-symplectic groupoids, as explained in \cite[Rk.3.3, \S7]{[LMS17]}.

\end{enumerate}
\end{rem}

Our first example is the quasi-Hamiltonian analogue of coadjoint orbits
in symplectic geometry. 
\begin{example}
\textbf{\label{Ex_TConj_Cl}(Twisted conjugacy classes) }Let $a\in G$
be a given element, and denote its $\kappa$-twisted conjugacy class
by $\mathcal{C}=\mathrm{Ad}_{G}^{\kappa}(a)$. This space has a natural
structure of a $G\kappa$-valued tq-Hamiltonian space. The group $G$
acts on $\mathcal{C}$ by $\kappa$-twisted conjugation, the moment
map $\Phi_{\mathcal{C}}:\mathcal{C}\to G\kappa$ is given by the inclusion
$\mathcal{C}\hookrightarrow G\kappa$, and the invariant 2-form is
uniquely determined by the moment map condition: 
\begin{equation}
\omega_{\mathcal{C}}(\xi_{\mathcal{C}},\zeta_{\mathcal{C}})_{x}=-\frac{1}{2}\left((\mathrm{Ad}_{x}\circ\kappa)-(\mathrm{Ad}_{x}\circ\kappa)^{-1}\right)\xi\cdot\zeta,\ \ \forall\xi,\zeta\in\mathfrak{g}.\label{Eq_Form_TConj_Class}
\end{equation}
The identity $d\omega_{\mathcal{C}}=\Phi_{\mathcal{C}}^{\ast}\eta$
and the minimal degeneracy condition are verified by elementary computations. 
\end{example}

For future reference, we give a special case of the example above
that allows us to view the group $G\kappa$ itself as a tq-Hamiltonian
$G\times G$-space. This example generalizes the discussion of \cite[\S 2.6]{[AMW02]}. 
\begin{example}
\textbf{\label{Ex_Gkappa}($\bm{G\kappa}$ as a twisted conjugacy
class) }For a given Lie group $G$ and a fixed automorphism $\kappa\in\mathrm{Aut}(G)$,
let $K=G\rtimes\langle\kappa\rangle$, and let $\nu\in\mathrm{Aut}(K\times K)$
be the involution: 
\[
\nu:K\times K\longrightarrow K\times K\mbox{, }(a,b)\longmapsto(b,a).
\]
Letting $G\times G$ act on $K\times K$ via $\nu$-twisted conjugation,
consider the orbit through $(\kappa,\kappa^{-1})$:
\[
\mathcal{C}=\mathrm{Ad}_{G\times G}^{\nu}(\kappa,\kappa^{-1}).
\]
We have a $G\times G$-equivariant diffeomorphism: 
\[
G\kappa\longrightarrow\mathcal{C},\ g\longmapsto\left(g\kappa,(g\kappa)^{-1}\right),
\]
where the $G\times G$ action on $G\kappa$ is given by: 
\begin{equation}
(g_{1},g_{2})\cdot g=g_{1}g\kappa(g_{2}^{-1}),\ \ \forall g,g_{1},g_{2}\in G.\label{Eq_Gkappa_Action}
\end{equation}
We can thus identify $G\kappa$ with the $\nu$-twisted conjugacy
class $\mathcal{C}$, with moment map: 
\begin{equation}
\Phi:G\kappa\longrightarrow(K\times K)\nu,\ \Phi(g)=\left(g\kappa,(g\kappa)^{-1}\right),\label{Eq_Gkappa_Moment}
\end{equation}
and q-Hamiltonian form $\omega_{G\kappa}\equiv0$ by equation (\ref{Eq_Form_TConj_Class}),
since for all $g\in G$: 
\[
\left(\mathrm{Ad}_{(g\kappa,(g\kappa)^{-1})}\nu\right)^{-1}=\mathrm{Ad}_{(g\kappa,(g\kappa)^{-1})}\nu.
\]
In conclusion, $(G\kappa,0,\Phi)$ is a $(K\times K)\nu$-valued tq-Hamiltonian
$G\times G$-space.
\end{example}

As previously mentioned, the twisted moduli spaces of section \ref{SubSec_Moduli}
provide a large class of examples of tq-Hamiltonian manifolds. The
construction of these spaces and their properties rely on fusion and
reduction, which we now review.

\subsubsection{Fusion}

We start with \textit{internal fusion}. Let $G$ and $H$ denote two
Lie groups, and consider an untwisted q-Hamiltonian $G\times G\times H$-space
$(M,\omega,\Phi)$ with moment map $\Phi:M\to G\times G\times H$.
Internal fusion produces an alternative q-Hamiltonian structure $(M,\omega_{\mathrm{fus}},\Phi_{\mathrm{fus}})$
on the same manifold $M$, with $\Phi_{\mathrm{fus}}$ now taking
values in $G\times H$. In the twisted setup, the precise definition
is as follows: 
\begin{prop}
\label{Prop_Def_Internal_Fusion}\textbf{\textup{(Internal fusion)}}
Let $G$ and $H$ be Lie groups, and $\kappa_{1},\kappa_{2}\in\mathrm{Aut}(G)$,
$\tau\in\mathrm{Aut}(H)$. Let $(M,\omega,\Phi)$ be a tq-Hamiltonian
$G\times G\times H$-manifold with moment map: 
\[
\Phi:M\to G\kappa_{1}\times G\kappa_{2}\times H\tau,\ \ \Phi(x)=\left(\Phi_{1}(x),\Phi_{2}(x),\Phi_{3}(x)\right).
\]
The triple $(M,\omega_{\mathrm{fus}},\Phi_{\mathrm{fus}})$, where
$M$ is equipped with the $G\times H$ action such that $G$ acts
diagonally, and where: 
\begin{eqnarray*}
\omega_{\mathrm{fus}} & = & \omega-\tfrac{1}{2}\Phi_{1}^{\ast}\theta^{L}\cdot\Phi_{2}^{\ast}\kappa_{1}^{\ast}\theta^{R},\\
\Phi_{\mathrm{fus}}(x) & = & \left(\Phi_{1}(x)\cdot\kappa_{1}\left(\Phi_{2}(x)\right),\Phi_{3}(x)\right),\ \forall x\in M,
\end{eqnarray*}
is a $G\kappa_{1}\kappa_{2}\times H\tau$-valued tq-Hamiltonian $G\times H$-space,
called the \textbf{internal fusion} of $(M,\omega,\Phi)$ with respect
to the first two components of $\Phi$. 
\end{prop}

Our next operation is a special case of the previous one. Given two
tq-Hamiltonian $G$-spaces $(M_{i},\omega_{i},\Phi_{i})$ with $i=1,2$,
it is easily seen from Definition \ref{Def_tq-Ham_Mfld} that $(M_{1}\times M_{2},\omega_{1}+\omega_{2},\Phi_{1}\times\Phi_{2})$
is a tq-Hamiltonian $G\times G$-space. Applying internal fusion to
the direct product $M_{1}\times M_{2}$, we obtain the so-called \textit{fusion
product} of the $(M_{i},\omega_{i},\Phi_{i})$. 
\begin{prop}
\textbf{\textup{\label{Prop_Def_Fusion}(Fusion product)}} Let $G$
be a Lie group with automorphisms $\kappa_{1},\kappa_{2}\in\mathrm{Aut}(G)$.
For $i=1,2$, suppose that $(M_{i},\omega_{i},\Phi_{i})$ are $G\kappa_{i}$-valued
tq-Hamiltonian $G$-spaces. Let $G$ act diagonally on $M_{1}\circledast M_{2}:=M_{1}\times M_{2}$,
and define: 
\begin{eqnarray*}
\omega_{\mathrm{fus}} & = & \omega_{1}+\omega_{2}-\tfrac{1}{2}\Phi_{1}^{\ast}\theta^{L}\cdot\Phi_{2}^{\ast}\kappa_{1}^{\ast}\theta^{R},\\
\Phi_{\mathrm{fus}}(x_{1},x_{2}) & = & \Phi_{1}(x_{1})\cdot\kappa_{1}\left(\Phi_{2}(x_{2})\right),\ \ \forall x_{i}\in M_{i}.
\end{eqnarray*}
The triple $(M_{1}\circledast M_{2},\omega_{\mathrm{fus}},\Phi_{\mathrm{fus}})$
is a $G\kappa_{1}\kappa_{2}$-valued tq-Hamiltonian space called the
\textbf{fusion product} of $(M_{1},\omega_{1},\Phi_{1})$ and $(M_{2},\omega_{2},\Phi_{2})$. 
\end{prop}

\begin{rem}
We keep the notation of Propositions \ref{Prop_Def_Internal_Fusion}
and \ref{Prop_Def_Fusion}.
\begin{enumerate}
\item Fusion composes the twisting automorphisms assigned to the target
spaces of moment maps. This is the first notable difference with the
fusion of untwisted q-Hamiltonian manifolds.
\item From the standpoint of Dirac geometry, fusion is simply a composition
of Dirac morphisms. As such, Propositions \ref{Prop_Def_Internal_Fusion}
and \ref{Prop_Def_Fusion} are the twisted analogues of \cite[Thm.5.6]{[AMB09]},
and the main ingredient in the proof is the fact that the group multiplication
map: 
\[
\mathrm{Mult}\circ(1\times\kappa_{1}):G\kappa_{1}\times G\kappa_{2}\longrightarrow G\kappa_{1}\kappa_{2},\ (a,b)\longmapsto a\kappa_{1}(b),
\]
extends to a $G$-equivariant Dirac morphism \cite[Thm.3.9]{[AMB09]}.
See Proposition \ref{Prop_Twisted_Mult_Inv} and Remark \ref{Rk_tq-Ham_Fusion_Inversion}
for more details.
\item Continuing on the previous remark, composing $\mathrm{Mult}:G\times G\to G$
by elements of $\mathrm{Aut}(G)$ allows to define several fusion
products on the category of tq-Hamiltonian manifolds. This is a second
major difference in comparison with the untwisted theory.
\end{enumerate}
\end{rem}

\subsubsection{Reduction}

We will need the following generalization of \cite[Prop.4.4]{[AMM98]}
(see Remark \ref{Rk_tq-Ham_Fusion_Inversion}): 
\begin{prop}
\textbf{\textup{\label{Prop_Inversion}(Inversion)}} Let $G$ be a
Lie group and $\kappa\in\mathrm{Aut}(G)$ an automorphism. If $(M,\omega,\Phi)$
is a $G\kappa$-valued tq-Hamiltonian $G$-space, its \textbf{inverse}
$(M^{-},-\omega,\Phi^{-})$ is a $G\kappa^{-1}$-valued tq-Hamiltonian
manifold, where $M^{-}=M$ as $G$-manifolds, and where the moment
map is given by: 
\[
\Phi^{-}:M\longrightarrow G\kappa^{-1},\ \ \Phi^{-}(x)=\kappa^{-1}\left(\Phi(x)^{-1}\right).
\]
\end{prop}

Reduction of tq-Hamiltonian $G$-spaces is defined in terms of the
shifting trick \cite[Rk.6.2]{[AMM98]}, which brings the operation
back to standard q-Hamiltonian reduction: 
\begin{defn}
\label{Def_Reduction}Given a $G\kappa$-valued tq-Hamiltonian space
$(M,\omega,\Phi)$, let $\mathcal{C}=\mathrm{Ad}_{G}^{\kappa}(a)$
be the twisted conjugacy class of $a\in G$. The \textbf{q-Hamiltonian}
\textbf{reduction} of $M$ at $a\in G$ is defined as the quotient:
\[
M_{a}=(M\circledast\mathcal{C}^{-})/\!/G.
\]
\end{defn}

By Proposition \ref{Prop_Def_Fusion}, $M\circledast\mathcal{C}^{-}$
is an untwisted q-Hamiltonian $G$-space, and by the discussion in
\cite[\S\S 5-6]{[AMM98]} we can state: 
\begin{prop}
\textbf{\textup{\label{Prop_Reduction}(Reduction)}} With the notation
above, one has that: 
\begin{enumerate}
\item A point $a\in G$ is a regular value of $\Phi:M\rightarrow G$ if
and only if $e\in G$ is a regular value of the moment map $\Phi_{a}:M\circledast\mathcal{C}^{-}\rightarrow G$.
\item If $Z_{a}^{\kappa}\subseteq G$ denotes the stabilizer of $a\in G$
under $\mathrm{Ad}_{G}^{\kappa}$, one has that: 
\[
M_{a}=\Phi_{a}^{-1}(e)/G\simeq\Phi^{-1}(a)/Z_{a}^{\kappa}.
\]
\item The 2-form on $\Phi_{a}^{-1}(e)\subset M\circledast\mathcal{C}^{-}$descends
to a symplectic form $\omega_{a}\in\Omega^{2}(M_{a})$. The space
$M_{a}$ is then a (singular) symplectic space in the sense of Sjamaar-Lerman
\cite{[SL91]}.
\end{enumerate}
\end{prop}

\subsection{Moduli spaces\label{SubSec_Moduli}}

This section is divided into two parts. The first one sets up the
notation, and reminds of several general constructions and facts pertaining
to character varieties associated to bordered surfaces. Its purpose
is to formulate a precise definition of our twisted moduli spaces,
and to specify the type of flat connections they parametrize.

The second part concerns the tq-Hamiltonian geometry of our character
varieties. After explaining the construction of the moment map and
the invariant 2-form associated to a twisted moduli space, we illustrate
the discussion of the previous section with concrete examples.

\subsubsection{Character varieties}

Let $X$ be a locally path-connected and locally simply connected
topological space, and let $Y\subseteq X$ be a closed subspace. We
employ the following conventions for the fundamental groupoid $\Pi_{1}(X,Y)\rightrightarrows Y$.
The source and target maps $\mathsf{s},\mathsf{t}:\Pi_{1}(X,Y)\to Y$
are defined as $\mathsf{s}[\gamma]=\gamma(0)$ and $\mathsf{t}[\gamma]=\gamma(1)$
for $[\gamma]\in\Pi_{1}(X,Y)$. The groupoid product is given by $[\beta][\gamma]=[\beta\ast\gamma]$
if $[\beta],[\gamma]\in\Pi_{1}(X,Y)$ satisfy $\mathsf{t}[\gamma]=\mathsf{s}[\beta]$,
where for representatives $\beta,\gamma:[0,1]\to X$:
\[
\left(\beta\ast\gamma\right)(t)=\begin{cases}
\gamma(2t), & t\in[0,\tfrac{1}{2}];\\
\beta(2t-1), & t\in[\tfrac{1}{2},1].
\end{cases}
\]
For a pair $(X,Y)$ as above and a group $K$, the associated $\bm{K}$\textbf{-character
variety} is the space of groupoid homomorphisms: 
\[
\mathrm{Hom}\left(\Pi_{1}(X,Y),K\right),
\]
where $K$ is viewed as a groupoid with one object. We will use the
following facts: 
\begin{fact}
With the notation above:
\begin{itemize}
\item[(a)] The $K$-character variety has a natural action of the \textbf{gauge
group} $\mathrm{Map}(Y,K)$, such that for all $\phi\in\mathrm{Map}(Y,K)$
and $\rho\in\mathrm{Hom}\left(\Pi_{1}(X,Y),K\right)$: 
\[
(\phi\cdot\rho)_{\alpha}=\phi_{\mathsf{t}(\alpha)}\rho_{\alpha}\phi_{\mathsf{s}(\alpha)}^{-1}\in K,\ \ \forall\alpha\in\Pi_{1}(X,Y).
\]
\item[(b)] Any morphism of pairs of topological spaces $f:(X,Y)\to(X',Y')$
induces a map: 
\[
f^{\ast}:\mathrm{Hom}\left(\Pi_{1}(X',Y'),K\right)\longrightarrow\mathrm{Hom}\left(\Pi_{1}(X,Y),K\right),
\]
which intertwines the actions of the gauge groups $\mathrm{Map}(Y',K)$
and $\mathrm{Map}(Y,K)$.
\item[(c)] Any morphism of groups $\varphi:K\to H$ induces a map: 
\[
\varphi_{\ast}:\mathrm{Hom}\left(\Pi_{1}(X,Y),K\right)\longrightarrow\mathrm{Hom}\left(\Pi_{1}(X,Y),K\right),
\]
which intertwines the actions of $\mathrm{Map}(Y,K)$ and $\mathrm{Map}(Y,H)$.
\end{itemize}
\end{fact}

For the remainder of this section, $\Sigma$ denotes a compact oriented
surface such that each connected component has a non-empty boundary,
each boundary circle in $\partial\Sigma$ has precisely one basepoint
$p_{j}$, and $S=\{p_{j}\}$ denotes the resulting finite collection
of basepoints. When there is no risk of confusion, we will simply
denote the fundamental groupoid of $\Sigma$ based at $S$ by $\Pi:=\Pi_{1}(\Sigma,S)$.
In view of the previous paragraph, we have the following definition: 
\begin{defn}
\label{Def_TMS}Let $G$ be a Lie group, and let $\mathsf{p}:G\rtimes\mathrm{Aut}(G)\to\mathrm{Aut}(G)$
denote the projection $(g,\kappa)\mapsto\kappa$. For an element $\sigma\in\mathrm{Hom}\left(\Pi,\mathrm{Aut}(G)\right)$,
called the \textbf{twist}, the $\bm{\sigma}$\textbf{-twisted moduli
space} associated to $(\Sigma,S,G)$ is the preimage of $\sigma$
under the induced map:
\[
\mathsf{p}_{\ast}:\mathrm{Hom}\left(\Pi,G\rtimes\mathrm{Aut}(G)\right)\to\mathrm{Hom}\left(\Pi,\mathrm{Aut}(G)\right).
\]
We denote this space by: 
\[
M_{\sigma}(\Sigma,G):=\mathrm{Hom}_{\sigma}\left(\Pi,G\right).
\]
\end{defn}

\begin{rem}
Let $\Pi$ and $\sigma\in\mathrm{Hom}\left(\Pi,\mathrm{Aut}(G)\right)$
be as in the definition.
\begin{enumerate}
\item The notation $\mathrm{Hom}_{\sigma}(\Pi,G)$ emphasizes the fact that
we view its elements as $\sigma$-twisted groupoid morphisms $\Pi\to G$.
That is, a map $\rho:\Pi\to G$ lies in $\mathrm{Hom}_{\sigma}(\Pi,G)$
if and only if for all composable $\alpha,\beta\in\Pi$: 
\[
\rho_{\alpha\beta}=\rho_{\alpha}\sigma_{\alpha}(\rho_{\beta}),\ \ \rho_{\alpha^{-1}}=\sigma_{\alpha}^{-1}(\rho_{\alpha}^{-1}).
\]
\item With the notation of the definition, the subspace: 
\[
(\mathsf{p}_{\ast})^{-1}(\sigma)\subseteq\mathrm{Hom}\left(\Pi,G\rtimes\mathrm{Aut}(G)\right)
\]
is not invariant under the action of the full gauge group $\mathrm{Map}\left(S,G\rtimes\mathrm{Aut}(G)\right)$.
It is however invariant under the subgroup $\mathrm{Map}\left(S,G\right)\simeq G^{|S|}$.
The latter acts as follows on $\rho\in\mathrm{Hom}_{\sigma}(\Pi,G)$:
\[
(\phi\cdot\rho)_{\alpha}=\phi_{\mathsf{t}(\alpha)}\rho_{\alpha}\sigma_{\alpha}\left(\phi_{\mathsf{t}(\alpha)}^{-1}\right)\in G,\ \ \forall\alpha\in\Pi,\ \phi\in\mathrm{Map}(S,G).
\]
\end{enumerate}
We now turn to the type of objects parametrized by the spaces $M_{\sigma}(\Sigma,G)$.
We recall the following (see also \cite{[BY15],[MoerStackNotes],[BryBk]}): 
\end{rem}

\begin{defn}
\label{Def_Loc_Systems_Torsors_Framings}Let $M$ be a manifold, let
$K$ be a Lie group and let $\mathcal{G}\to M$ be a Lie group bundle
of typical fibre $K$. 
\begin{itemize}
\item[(a)] The bundle $\mathcal{G}\to M$ is called a \textbf{local system}
(of groups) on $M$ if it is equipped with a flat Ehresmann connection
such that parallel transport between fibres is given by group isomorphisms.
\item[(b)] A bundle $\mathcal{P}\to M$ is called a $\mathcal{G}$\textbf{-torsor}
if it is equipped with a fibre-wise free and transitive action of
the bundle $\mathcal{G}\to M$.
\item[(c)] Let $\mathcal{P}\to M$ be a torsor. A \textbf{framing} of $\mathcal{P}$
at a point $x\in M$ is a choice of $\mathcal{G}_{x}$-equivariant
isomorphism $\psi_{x}:\mathcal{G}_{x}\to\mathcal{P}_{x}$.
\end{itemize}
In the case where $\mathcal{G}\to M$ is a local system of groups
and $\mathcal{P}\to M$ is a local system of simply transitive homogeneous
spaces for $\mathcal{G}$, we will say that $\mathcal{P}$ is a \textbf{flat
$\mathcal{G}$-torsor}. 
\end{defn}

\begin{rem}
For the sake of clarity, let us mention some facts that will be implicitly
used later. More details can be found in \cite[\S 2]{[BY15]}.
\begin{enumerate}
\item For any local system of groups $\mathcal{G}\to M$ with typical fibre
$K$, there exists a trivializing open cover $\{\mathcal{U}_{i}\}_{i\in I}$
of $M$ such that the transition functions $\mathcal{U}_{i}\cap\mathcal{U}_{j}\to\mathrm{Aut}(K)$
are locally constant on $\mathcal{U}_{i}\cap\mathcal{U}_{j}\ne\emptyset$.
\item If $\mathcal{P}\to M$ is a $\mathcal{G}$-torsor, its typical fibre
is the manifold $K$, without its group structure. Furthermore, a
framing $\psi_{x}:\mathcal{G}_{x}\to\mathcal{P}_{x}$ of $\mathcal{P}$
at a point $x\in M$ is equivalent to choosing a distinguished point
$p=\psi_{x}(e)\in\mathcal{P}_{x}$.
\item If we take $\mathcal{G}=M\times K$ to be the trivial $K$-bundle
over $M$, a $\mathcal{G}$-torsor $\mathcal{P}\to M$ is the same
as a principal $K$-bundle over $M$. As such, $\mathcal{G}$-torsors
give a natural generalization of principal bundles.
\end{enumerate}
\end{rem}

Let $\Sigma$ be an oriented, compact and connected bordered surface.
As above, let $S\subseteq\partial\Sigma$ be the boundary basepoints,
and let $\Pi=\Pi_{1}(\Sigma,S)$ be the fundamental groupoid. Consider
the universal cover of $\Sigma$ based at $S$: 
\[
\widetilde{\Sigma}:=\left\{ \gamma:[0,1]\to\Sigma\ \big|\ \gamma(0)\in S\right\} /\left\{ \mbox{homotopy rel. }\{0,1\}\right\} ,
\]
with projection $\pi:\widetilde{\Sigma}\to\Sigma$, $[\gamma]\mapsto\mathsf{t}[\gamma]$.
This projection coincides with the quotient map with respect to the
action of $\Pi$ on $\widetilde{\Sigma}$ by concatenation from the
right: 
\[
\Pi\times\widetilde{\Sigma}\longrightarrow\widetilde{\Sigma},\ (\alpha,x)\longmapsto x\cdot\alpha^{-1},
\]
for $(\alpha,x)\in\Pi\times\widetilde{\Sigma}$ such that $\mathsf{s}(x)=\mathsf{s}(\alpha)$.

For a fixed Lie group $G$ with automorphism group $\mathrm{Aut}(G)$,
any twist $\sigma\in\mathrm{Hom}\left(\Pi,\mathrm{Aut}(G)\right)$
gives rise to a local system of groups: 
\begin{equation}
\mathcal{E}_{\sigma}:=\widetilde{\Sigma}\times_{(\sigma,\Pi)}\mathrm{Aut}(G)=\left(\widetilde{\Sigma}\times\mathrm{Aut}(G)\right)/\sim_{(\sigma,\Pi)},\label{Eq_Loc_Syst_Twist}
\end{equation}
where the equivalence relation $(\sigma,\Pi)$ is given by: 
\[
(x,\kappa)\sim_{(\sigma,\Pi)}(y,\tau)\Leftrightarrow\exists\alpha\in\Pi:\ y=x\cdot\alpha^{-1},\ \tau=\sigma_{\alpha}\kappa,
\]
for $(x,\kappa),(y,\tau)\in\widetilde{\Sigma}\times\mathrm{Aut}(G)$.
We then have the following local system of groups: 
\begin{equation}
\mathcal{G}_{\sigma}:=\mathcal{E}_{\sigma}\times_{\mathrm{Aut}(G)}G,\label{Eq_Acting_Sheaf}
\end{equation}
and for any $\sigma$-twisted homomorphism $\rho\in\mathrm{Hom}_{\sigma}(\Pi,G)$,
we have a flat $\mathcal{G}_{\sigma}$-torsor: 
\begin{equation}
\mathcal{P}_{\rho}:=\widetilde{\Sigma}\times_{(\rho,\Pi)}G,\label{Eq_Flat_Torsor}
\end{equation}
where the equivalence relation $\sim_{(\rho,\Pi)}$ on $\widetilde{\Sigma}\times G$
is given by: 
\[
(x,g)\sim_{(\rho,\Pi)}(y,h)\Longleftrightarrow\exists\alpha\in\Pi:\ y=x\cdot\alpha^{-1},\ h=\rho_{\alpha}\sigma_{\alpha}(g).
\]
Thus, the bundle $\mathcal{P}_{\rho}\to\Sigma$ is naturally equipped
with a flat Ehresmann connection, for which the holonomy along $\alpha\in\Pi$
is given by $\rho_{\alpha^{-1}}\in G$. Furthermore, a framing of
$\mathcal{P}_{\rho}$ at any basepoint $p_{j}\in S$ induces a trivialization
of $\mathcal{P}_{\rho}|_{S}$.

The discussion above associates a flat $\mathcal{G}_{\sigma}$-torsor
framed at $S\subseteq\partial\Sigma$ to any $\rho\in M_{\sigma}(\Sigma,G)$,
and by modifying the proof of the Riemann-Hilbert correspondence \cite[Cor.1.4]{[Del70Bk]}
to account for boundary framings, we have: 
\begin{prop}
\label{Prop_Riemann-Hilbert}Let $\Sigma$ be a bordered surface with
boundary basepoints $S\subseteq\partial\Sigma$, and let $\mathcal{G}_{\sigma}\to\Sigma$
be the local system of groups (\ref{Eq_Acting_Sheaf}) obtained from
a twist $\sigma\in\mathrm{Hom}\left(\Pi,\mathrm{Aut}(G)\right)$.
There is a bijective correspondence between elements of $M_{\sigma}(\Sigma,G)$
and flat $\mathcal{G}_{\sigma}$-torsors on $\Sigma$ framed at the
boundary basepoints $S\subseteq\partial\Sigma$. 
\end{prop}

Suppose now that $\Sigma=\Sigma_{h}^{b}$ is a connected surface of
genus $h\ge0$ with $b\ge1$ boundary circles. Recall that the fundamental
groupoid $\Pi=\Pi_{1}(\Sigma,S)$ admits a \textit{system of free
generators} $\mathcal{F}$. Such generators are obtained by considering
a finite set of non-intersecting paths $\{P_{i}\}_{i\in I}$ in $\Sigma$
with endpoints in $S$, such that cutting $\Sigma$ along the $P_{i}$
results in a polygon. The system $\mathcal{F}$ is then given by the
set $\{[P_{i}]\}_{i\in I}\subseteq\Pi$ and the homotopy classes of
$(b-1)$ boundary circles. Since any twisted homomorphism $\rho\in M_{\sigma}(\Sigma,G)$
is completely determined by its values on the elements of $\mathcal{F}$,
we can state: 
\begin{prop}
\label{Prop_TMS_as_Mflds}Let $\Sigma=\Sigma_{h}^{b}$ be a connected
bordered surface with boundary basepoints $S$, and consider a twist
$\sigma\in\mathrm{Hom}\left(\Pi,\mathrm{Aut}(G)\right)$. Any choice
of a system of free generators of $\Pi$ gives rise to a diffeomorphism:
\[
M_{\sigma}(\Sigma,G)\cong G^{2(h+b-1)}.
\]
\end{prop}

\begin{rem}
Let $\Sigma_{h}^{b}$ be as above, and let $S=\{p_{j}\}_{j=1}^{b}$
denote the boundary basepoints. 
\begin{enumerate}
\item Given a system of free generators $\mathcal{F}\subseteq\Pi$, the
twist $\sigma\in\mathrm{Hom}\left(\Pi,\mathrm{Aut}(G)\right)$ is
constructed in practice by the assigning values $\sigma_{\alpha}\in\mathrm{Aut}(G)$
for $\alpha\in\mathcal{F}$, and then extending by the homomorphism
property.
\item A concrete example of a system of free generators of $\Pi$ is described
in detail in \cite[\S 9.2]{[AMM98]}. There, the authors take the
set: 
\[
\mathcal{F}=\left\{ A_{i},\ B_{i},\ V_{j},\ U_{j}\ |\ 1\le i\le h,\ 2\le j\le b\right\} ,
\]
where the $\{A_{i},B_{i}\}_{i=1}^{h}$ correspond to the handles of
$\Sigma_{h}^{b}$, the class $V_{j}$ is that of the boundary circle
based at $p_{j}\in S$, and $U_{j}$ is the class of a path from $p_{j}$
to $p_{1}$. These generators are then subject to the relation: 
\[
V_{1}\prod_{j=2}^{b}U_{j}V_{j}U_{j}^{-1}\prod_{i=1}^{h}[A_{i},B_{i}]=1,
\]
which is the word formed by the boundary segments of the polygon obtained
by cutting $\Sigma_{h}^{b}$ along the $A_{i}$, $B_{i}$ and $U_{j}$.
\end{enumerate}
\end{rem}

To study twisted moduli spaces with $G$ compact, connected and simply
connected, it is sufficient to consider twists taking values in the
diagram automorphism group $\mathrm{Out}(G)$, viewed as a subgroup
of $\mathrm{Aut}(G)$. This is explained by: 
\begin{prop}
\label{Prop_Isom_Twists}Let $G$ be a compact 1-connected Lie group,
and let $\Sigma$ be a bordered surface with boundary basepoints $S\subseteq\partial\Sigma$.
If the twists $\sigma,\tau\in\mathrm{Hom}\left(\Pi,\mathrm{Aut}(G)\right)$
have the same image in $\mathrm{Hom}\left(\Pi,\mathrm{Out}(G)\right)$,
then:
\begin{itemize}
\item[(a)] The local systems $\mathcal{G}_{\sigma}$ and $\mathcal{G}_{\tau}$
over $\Sigma$ are isomorphic.
\item[(b)] The moduli spaces $M_{\sigma}(\Sigma,G)$ and $M_{\tau}(\Sigma,G)$
are $G^{|S|}$-equivariantly isomorphic.
\end{itemize}
\end{prop}

\noindent \emph{Outline of proof.} Statement (b) follows from (a),
equation (\ref{Eq_Flat_Torsor}), and Proposition \ref{Prop_Riemann-Hilbert}.
Statement (a) is a consequence of the fact that if $\sigma,\tau\in\mathrm{Hom}\left(\Pi,\mathrm{Aut}(G)\right)$
have the same image in $\mathrm{Hom}\left(\Pi,\mathrm{Out}(G)\right)$,
then the local systems $\mathcal{E}_{\sigma}$ and $\mathcal{E}_{\tau}$
defined by equation (\ref{Eq_Loc_Syst_Twist}) are isomorphic. Such
an isomorphism can be obtained by constructing a map $\psi:\widetilde{\Sigma}\to\mathrm{Aut}(G)$
such that: 
\[
\psi(x\cdot\alpha^{-1})=\tau_{\alpha}\circ\psi(x)\circ\sigma_{\alpha}^{-1},\ \ \forall x\in\widetilde{\Sigma},\ \alpha\in\Pi.
\]
and use it to get an automorphism of the trivial principal $\mathrm{Aut}(G)$-bundle
$\widetilde{\Sigma}\times\mathrm{Aut}(G)$ that maps the equivalence
classes of $\sim_{(\sigma,\Pi)}$ to those of $\sim_{(\tau,\Pi)}$
(see notation after eq. (\ref{Eq_Loc_Syst_Twist})). An explicit construction
of the map $\psi:\widetilde{\Sigma}\to\mathrm{Aut}(G)$ parallels
the one given in \cite[App.A]{[Sen02]}. \qed
\begin{rem}
Up to this point, we have only considered connected surfaces $\Sigma$,
but this assumption is not really restrictive. For a surface $\Sigma=\Sigma_{1}\sqcup\Sigma_{2}$
with two connected components $\Sigma_{i}=\Sigma_{h_{i}}^{b_{i}}$
($i=1,2$), the fundamental groupoid decomposes as a product $\Pi=\Pi_{1}(\Sigma_{1},S_{1})\times\Pi_{1}(\Sigma_{2},S_{2})$,
where $S_{i}\subseteq\partial\Sigma_{i}$ are the basepoints. The
twist $\sigma$ decomposes accordingly, and the corresponding twisted
moduli space is just a product: 
\[
M_{\sigma}\left(\Sigma_{1}\sqcup\Sigma_{2},G\right)=M_{\sigma_{1}}\left(\Sigma_{1},G\right)\times M_{\sigma_{2}}\left(\Sigma_{2},G\right).
\]
The above discussion thus extends directly to compact surfaces with
finitely many connected components. 
\end{rem}

\subsubsection{The tq-Hamiltonian structure}

The preceding discussion shows that the moduli spaces $M_{\sigma}(\Sigma,G)$
are naturally equipped with a group-valued moment map, which is given
by evaluation along the boundary $\partial\Sigma$, and equivariant
with respect to the natural gauge group action. To specify the q-Hamiltonian
form $\omega_{\sigma}\in\Omega^{2}\left(M_{\sigma}(\Sigma,G)\right)^{G^{b}}$,
we follow \v Severa's formulation in \cite[Thm.3.1]{[Sev12]}.

Let $M=M_{\sigma}(\Sigma,G)$ with $\sigma\in\mathrm{Hom}\left(\Pi,\mathrm{Aut}(G)\right)$
and $\Sigma=\Sigma_{h}^{b}$ connected. Let $\Gamma\subseteq\mathrm{Aut}(G)$
be the subgroup generated by the image of $\sigma$, let $K=G\rtimes\Gamma$,
and consider the group $K(M):=\mathcal{C}^{\infty}(M,K)$ with pointwise
multiplication. The central extension $K(M)\times\Omega^{2}(M)$ then
has multiplication and inversion given by:
\begin{eqnarray}
(q_{1},\kappa_{1},\omega_{1})\cdot(q_{2},\kappa_{2},\omega_{2}) & = & \left(q_{1}.\kappa_{1}q_{2},\kappa_{1}\kappa_{2},\omega_{1}+\omega_{2}-\tfrac{1}{2}q_{1}^{\ast}\theta^{L}\cdot\kappa_{1}(q_{2}^{\ast}\theta^{R})\right),\label{Eq_K(M)_Prod}\\
(q_{1},\kappa_{1},\omega_{1})^{-1} & = & (\kappa_{1}^{-1}q_{1}^{-1},\kappa_{1}^{-1},-\omega_{1}),\label{Eq_K(M)_Inv}
\end{eqnarray}
where $q_{i}:M\to G$, $\kappa_{i}\in\mathrm{Aut}(G)$ and $\omega_{i}\in\Omega^{2}(M)$
for $i=1,2$. We are mainly interested in the elements $(\mathrm{ev}_{\gamma},0)\in K(M)\times\Omega^{2}(M)$,
where for $\gamma\in\Pi$:
\begin{equation}
\mathrm{ev}_{\gamma}:M_{\sigma}(\Sigma,G)\to K,\ \ \rho\mapsto(\rho_{\gamma},\sigma_{\gamma}).\label{Eq_Ev_K}
\end{equation}
Next, suppose $\Delta_{\mathcal{F}}(\Sigma)$ is a polygon obtained
from a system of free generators $\mathcal{F}\subset\Pi$ (see parag.
after. Prop.\ref{Prop_Riemann-Hilbert}). If $\{E_{i}\}_{i=1}^{n_{E}}$
denote the edges of $\partial\Delta_{\mathcal{F}}(\Sigma)$ ($n_{E}=4h+3b-2$),
then their homotopy classes in $\Pi$ satisfy the relation:
\[
\prod_{i=1}^{n_{E}}[E_{i}]=[\partial\Delta_{\mathcal{F}}(\Sigma)]=1.
\]
Using these objects, we state:
\begin{thm}
\label{Thm_TMS_tq-Ham_str}Let $\Sigma=\Sigma_{h}^{b}$ be a bordered
surface with boundary basepoints $S$, and fix a twist $\sigma\in\mathrm{Hom}\left(\Pi,\mathrm{Aut}(G)\right)$.
The triple $\left(M_{\sigma}(\Sigma,G),\omega_{\sigma},\Phi_{\sigma}\right)$
is tq-Hamiltonian $G^{b}$-space, where:
\begin{itemize}
\item[(i)] The moment map is given by the holonomies along the boundary circles
$V_{i}$ of $\Sigma$: 
\begin{equation}
\begin{array}{ccccc}
\Phi_{\sigma} & : & M_{\sigma}(\Sigma,G) & \longrightarrow & G\sigma_{V_{1}^{-1}}\times\cdots\times G\sigma_{V_{b}^{-1}},\\
 &  & \rho & \longmapsto & \left(\rho_{V_{1}^{-1}},\cdots,\rho_{V_{b}^{-1}}\right).
\end{array}\label{Eq_TMS_Moment}
\end{equation}
\item[(ii)] For any polygon presentation $\Delta_{\mathcal{F}}(\Sigma)$ of $\Sigma$
with edges $\{E_{i}\}_{i=1}^{n_{E}}$, the invariant 2-form $\omega_{\sigma}\in\Omega^{2}\left(M_{\sigma}(\Sigma,G)\right)^{G^{b}}$
is given by Ševera's formula: 
\begin{equation}
(e,1,\omega_{\sigma})=\prod_{i=1}^{n_{E}}(\mathrm{ev}_{E_{i}},0),\label{Eq_TMS_Form_Severa}
\end{equation}
where the product is in the group $K\left(M_{\sigma}(\Sigma,G)\right)\times\Omega^{2}\left(M_{\sigma}(\Sigma,G)\right)$.
\end{itemize}
\end{thm}

\begin{rem}
Equation (\ref{Eq_TMS_Form_Severa}) can be obtained by combining
the gauge theoretic construction in \cite[\S 9]{[AMM98]} with the
equivalence theorem for tq-Hamiltonian manifolds (Remark \ref{Rks_Def_tq-Ham_mflds}-(4)).
For a $\mathcal{G}_{\sigma}$-torsor $\mathcal{P}\to\Sigma$, let
$\mathcal{M}$ denote the moduli space of flat connections on $\mathcal{P}$,
modulo gauge transformations that are trivial along $\partial\Sigma$.
Then with $\kappa_{i}=\sigma_{V_{i}^{-1}}$, $\mathcal{M}$ is the
Hamiltonian $L^{(\kappa_{1})}G\times\cdots L^{(\kappa_{b})}G$-space
equivalent to $M_{\sigma}(\Sigma,G)$. Modifying the proof of \cite[Thm.9.3]{[AMM98]}
to account for the twist $\sigma$, one obtains the invariant symplectic
structure on $\mathcal{M}$, which by \cite[\S 7]{[LMS17]} gives
rise to $\omega_{\sigma}$. The key observation with \v Severa's
formulation is that the multiplication in the group $K\left(M_{\sigma}(\Sigma,G)\right)\times\Omega^{2}\left(M_{\sigma}(\Sigma,G)\right)$
reproduces the equations of the fusion product of tq-Hamiltonian manifolds.
\end{rem}

\begin{example}
\label{Ex_Moduli_Double}\textbf{(The double)} We generalize the example
$D(G)\simeq G\times G$ studied in \cite[\S 3.2]{[AMM98]}, by realizing
it a twisted moduli space associated to an annulus (or a cylinder)
$\Sigma_{0}^{2}$.

As generators of $\Pi=\Pi_{1}(\Sigma_{0}^{2},\{p_{1},p_{2}\})$, we
take the paths $X$ and $Y$ depicted in Figure \ref{Fig_Annulus},
and we define the twist $\sigma\in\mathrm{Hom}\left(\Pi,\mathrm{Aut}(G)\right)$
by setting: 
\[
\sigma_{X}=\tau,\ \sigma_{Y}=\kappa\in\mathrm{Aut}(G).
\]
We use the parametrization $(x,y)=(\rho_{X},\rho_{Y})$ for the elements
$\rho\in D_{\sigma}(G)=M_{\sigma}(\Sigma_{0}^{2},G)$. The gauge action
of an element $\phi=(g_{1},g_{2})$ of $\mathrm{Map}(S,G)=G\times G$
is then expressed as: 
\begin{eqnarray*}
(g_{1},g_{2})\cdot(x,y) & = & \left(g_{1}x\tau(g_{2}^{-1}),g_{2}x\kappa(g_{1}^{-1})\right).
\end{eqnarray*}

\begin{figure}[H]
\begin{centering}
\includegraphics[scale=0.5]{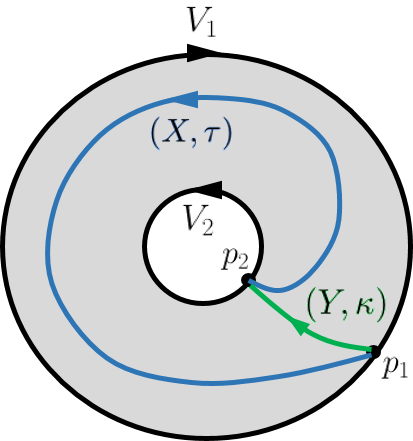}
\par\end{centering}
\caption{Generators of $\Pi_{1}(\Sigma_{0}^{2},\{p_{1},p_{2}\})$ and twist}
\label{Fig_Annulus}
\end{figure}

Since the boundary circles of $\Sigma_{0}^{2}$ are then given by
$V_{1}=Y^{-1}X^{-1}$ and $V_{2}=YX$, the components of the moment
map $\Phi_{D_{\sigma}}:D_{\sigma}(G)\to G\tau\kappa\times G\tau^{-1}\kappa^{-1}$
are given by:
\begin{equation}
\begin{cases}
\Phi_{1}(x,y) & =x\tau(y),\\
\Phi_{2}(x,y) & =\tau^{-1}\left(x^{-1}\kappa^{-1}(y^{-1})\right),
\end{cases}\label{Eq_Components_Moment_Dble}
\end{equation}
and satisfy: 
\[
\Phi_{D_{\sigma}}\left((g_{1},g_{2})\cdot(x,y)\right)=\left(\mathrm{Ad}_{g_{1}}^{\tau\kappa}\Phi_{1}(x,y),\ \mathrm{Ad}_{g_{2}}^{\tau^{-1}\kappa^{-1}}(x,y)\right).
\]

\begin{figure}
\begin{centering}
\includegraphics[scale=0.5]{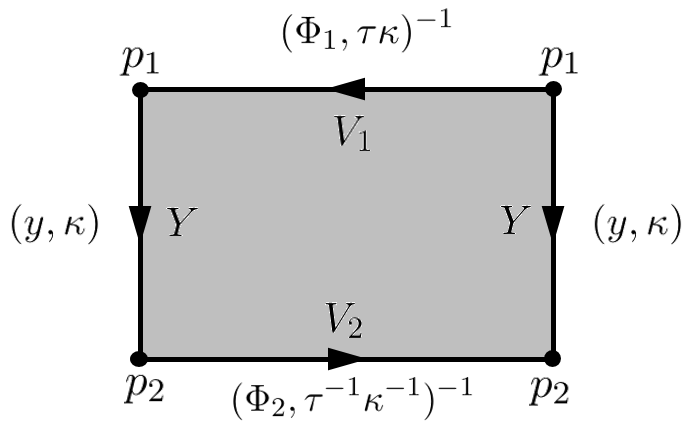}
\par\end{centering}
\caption{Construction of $\omega_{D_{\sigma}}$}

\label{Fig_2-Form_Double}
\end{figure}

To construct the invariant 2-form $\omega_{D_{\sigma}}$, we cut the
annulus of Figure \ref{Fig_Annulus} along the path $Y$, and assigning
the functions 
\[
(\mathrm{ev}_{V_{j}},0),\ (\mathrm{ev}_{Y},0)\in\mathcal{C}^{\infty}\left(D_{\sigma}(G),K\right)\times\Omega^{2}\left(D_{\sigma}(G)\right)
\]
to the edges of the obtained square, we have the situation depicted
in Figure \ref{Fig_2-Form_Double}. By Ševera's formula (\ref{Eq_TMS_Form_Severa}):
\[
(e,1,\omega_{D_{\sigma}})=\left(e,1,-\tfrac{1}{2}(\tau^{-1}x^{\ast}\theta^{L}\cdot y^{\ast}\theta^{R}+x^{\ast}\theta^{R}\cdot\kappa^{-1}y^{\ast}\theta^{L})\right),
\]
so that:
\begin{equation}
\omega_{D_{\sigma}}=-\tfrac{1}{2}\left(\tau^{-1}(x^{\ast}\theta^{L})\cdot y^{\ast}\theta^{R}+x^{\ast}\theta^{R}\cdot\kappa^{-1}(y^{\ast}\theta^{L})\right).\label{Eq_2Form_Double}
\end{equation}
The $\mathrm{Map}(S,G)$-invariance of $\omega_{\sigma}$ easily follows
from this equation. Modifying the identities used in the proof of
\cite[Prop.3.2]{[AMM98]} to incorporate the automorphisms $\tau,\kappa\in\mathrm{Aut}(G)$,
one checks the moment map condition, the equation $d\omega_{D_{\sigma}}=\Phi_{D_{\sigma}}^{\ast}\eta_{G\times G}$,
as well as the minimal degeneracy condition. 
\end{example}

\subsubsection{Fusion and reduction}

Let $\Sigma$ be a possibly disconnected bordered surface with boundary
basepoints $S=\{p_{j}\}_{j=1}^{b}$, and fix a twist $\sigma\in\mathrm{Hom}\left(\Pi,\mathrm{Aut}(G)\right)$.
For simplicity, we denote the moment map of the associated moduli
space by: 
\[
\Phi=(\Phi_{1},\cdots,\Phi_{b}):\ M_{\sigma}(\Sigma,G)\longrightarrow G\kappa_{1}\times\cdots\times G\kappa_{b},
\]
where $\kappa_{j}=\sigma_{V_{j}^{-1}}\in\mathrm{Aut}(G)$ for $j=1,\cdots,b$.

For integers $1\le i<j\le b$, the \textbf{internal fusion} $\left(M_{\sigma}(\Sigma,G)\right)_{ij}$
corresponds to the moduli space of the surface $\Sigma_{(ij)}=\Sigma\cup_{V_{i},V_{j}}\Sigma_{0}^{3}$,
obtained by gluing a pair of pants $\Sigma_{0}^{3}$ to the boundary
components $V_{i}$ and $V_{j}$ of $\Sigma$, as depicted in Figure
\ref{Fig_Gluing_Internal_Fusion}. Proposition \ref{Prop_Def_Internal_Fusion}
gives the tq-Hamiltonian structure on $\left(M_{\sigma}(\Sigma,G)\right)_{ij}$,
for which the moment map is given by: 
\begin{eqnarray*}
\Phi_{ij}:\left(M_{\sigma}(\Sigma,G)\right)_{ij} & \longrightarrow & G\kappa_{i}\kappa_{j}\times G\kappa_{1}\times\cdots\widehat{G\kappa_{i}}\times\cdots\widehat{G\kappa_{j}}\times\cdots\times G\kappa_{b},\\
\rho & \longmapsto & \left(\Phi_{i}\cdot(\kappa_{i}\circ\Phi_{j}),\ \Phi_{1},\cdots,\widehat{\Phi_{i}},\cdots,\widehat{\Phi_{j}},\cdots,\Phi_{b}\right)(\rho).
\end{eqnarray*}
We illustrate this with the generalization of \cite[Ex.6.1]{[AMM98]}
to our setup, the fused double $\mathbb{D}_{\varphi}(G)$.

\begin{figure}[t]
\begin{centering}
\includegraphics[scale=0.4]{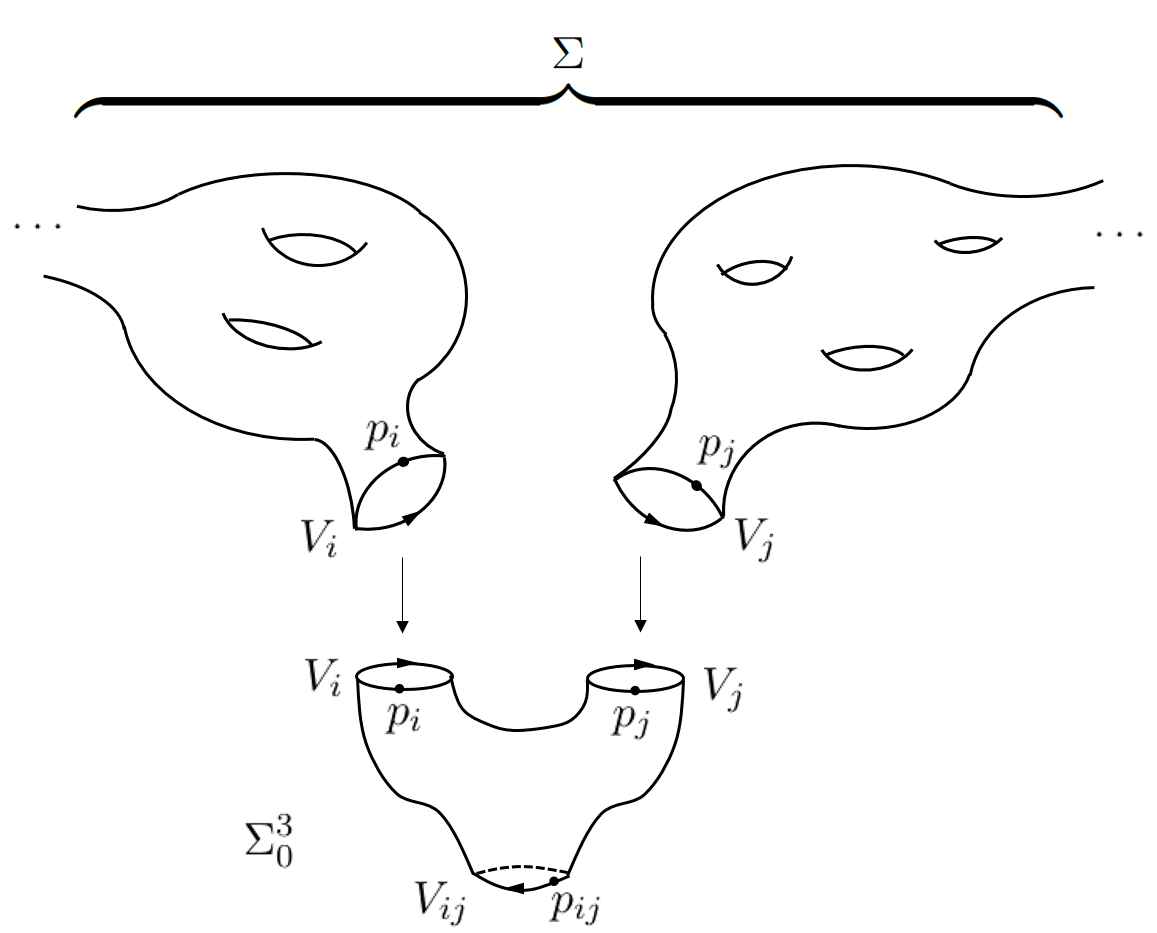}
\par\end{centering}
\caption{Gluing of surfaces and internal fusion}
\label{Fig_Gluing_Internal_Fusion}
\end{figure}

\begin{example}
\label{Ex_Moduli_Fused_Double} \textbf{(The fused double)} Keeping
the notation and objects of Example \ref{Ex_Moduli_Double}, consider
a pair of pants $\Sigma_{0}^{3}$ with paths $W_{1}$ and $W_{2}$
joining the basepoints as depicted in Figure \ref{Fig_Fused_Double_Illustration},
which we glue to the cylinder $\Sigma_{0}^{2}$ to obtain the surface
$\Sigma_{1}^{1}$.

\begin{figure}[H]
\begin{centering}
\includegraphics[scale=0.5]{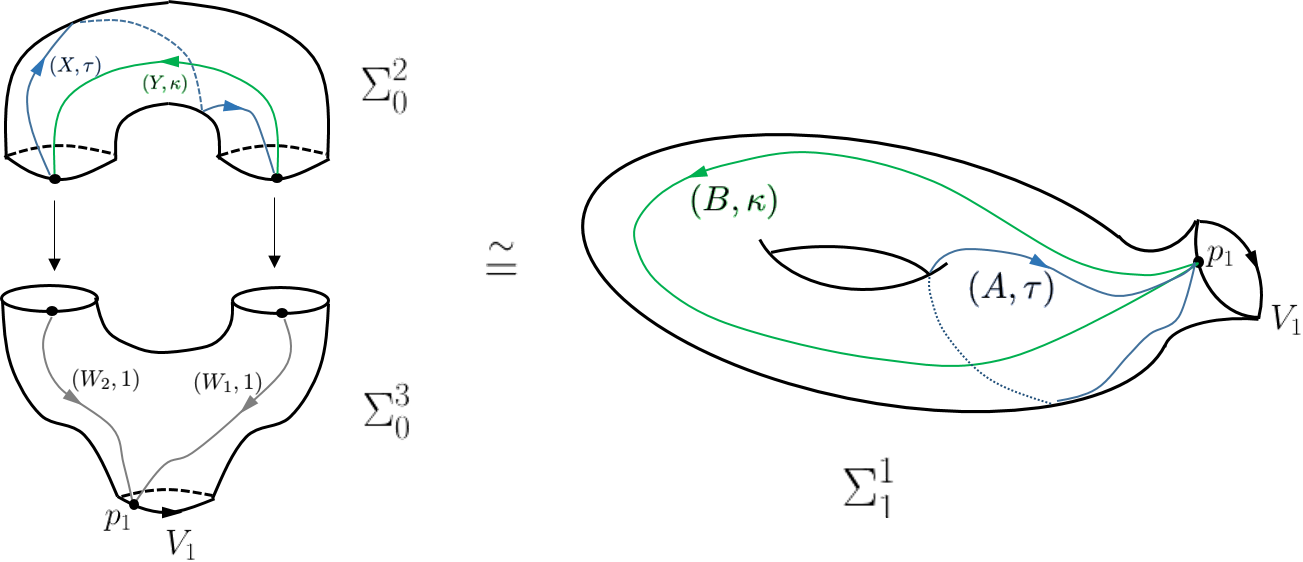}
\par\end{centering}
\caption{Gluing, generators and twist for $\Pi_{1}(\Sigma_{1}^{1},\{p_{1}\})$}
\label{Fig_Fused_Double_Illustration}
\end{figure}

The paths $A=W_{1}XW_{2}^{-1}$ and $B=W_{2}YW_{1}^{-1}$ give a system
of free generators for $\Pi=\Pi_{1}(\Sigma_{1}^{1},p_{1})$, and we
obtain the twist $\varphi\in\mathrm{Hom}\left(\Pi,\mathrm{Aut}(G)\right)$
satisfying: 
\[
\varphi_{A}=\tau\mbox{, }\varphi_{B}=\kappa.
\]
Parametrizing $\rho\in\mathbb{D}_{\varphi}(G)=M_{\varphi}(\Sigma_{1}^{1},G)$
using $(a,b)=(\rho_{A},\rho_{B})$, the action of the gauge group
$\mathrm{Map}(\{p_{1}\},G)=G$ is expressed as: 
\[
g\cdot(a,b)=\left(\mathrm{Ad}_{g}^{\tau}(a),\mathrm{Ad}_{g}^{\kappa}(b)\right),\ \ \forall g\in G,
\]
while the moment map $\Phi_{\mathbb{D}_{\varphi}}:\mathbb{D}_{\varphi}(G)\to G[\tau,\kappa]$
is given by: 
\[
\Phi_{\mathbb{D}_{\varphi}}(a,b)=\Phi_{1}(a,b)\tau\kappa\left(\Phi_{2}(a,b)\right),
\]
with $\Phi_{i}$ as in equation (\ref{Eq_Components_Moment_Dble})
and $[\tau,\kappa]=\tau\kappa\tau^{-1}\kappa^{-1}$. The equivariance
property in this case reads: 
\[
\Phi_{\mathbb{D}_{\varphi}}\left(g\cdot(a,b)\right)=\mathrm{Ad}_{g}^{[\tau,\kappa]}\left(\Phi_{\mathbb{D}_{\varphi}}(a,b)\right),\ \ \forall g\in G.
\]
Next, cutting $\Sigma_{1}^{1}$ along the paths $A$ and $B$, and
assigning the appropriate elements of $K\left(\mathbb{D}_{\varphi}(G)\right)$
to the edges of the resulting pentagon, equation (\ref{Eq_TMS_Form_Severa})
yields: 
\begin{eqnarray*}
(e,1,\omega_{\mathbb{D}_{\varphi}}) & = & \left(\Phi_{\mathbb{D}_{\varphi}},[\tau,\kappa],0\right)^{-1}\left(a,\tau,0\right)\left(b,\kappa,0\right)\left(a,\tau,0\right)^{-1}\left(b,\kappa,0\right)^{-1},\\
 & = & \left(\Phi_{\mathbb{D}_{\varphi}},[\tau,\kappa],0\right)^{-1}\left(\Phi_{1},\tau\kappa,-\tfrac{1}{2}\tau^{-1}a^{\ast}\theta^{L}\cdot b^{\ast}\theta^{R}\right)\left(\Phi_{2},\tau^{-1}\kappa^{-1},-\tfrac{1}{2}a^{\ast}\theta^{R}\cdot\kappa^{-1}b^{\ast}\theta^{L}\right),
\end{eqnarray*}
so that the resulting 2-form $\omega_{\mathbb{D}_{\varphi}}\in\Omega^{2}\left(\mathbb{D}_{\varphi}(G)\right)^{G}$
coincides with the one given by Proposition \ref{Prop_Def_Internal_Fusion}:
\[
\omega_{\mathbb{D}_{\varphi}}=\omega_{D_{\sigma}}-\tfrac{1}{2}\Phi_{1}^{\ast}\theta^{L}\cdot\Phi_{2}^{\ast}(\tau\kappa)^{\ast}\theta^{R},
\]
where as in Example \ref{Ex_Moduli_Double}: 
\[
\omega_{D_{\sigma}}=-\tfrac{1}{2}\left(\tau^{-1}(a^{\ast}\theta^{L})\cdot b^{\ast}\theta^{R}+a^{\ast}\theta^{R}\cdot\kappa^{-1}(b^{\ast}\theta^{L})\right).
\]
Thus, $\left(\mathbb{D}_{\varphi}(G),\omega_{\mathbb{D}_{\varphi}},\Phi_{\mathbb{D}_{\varphi}}\right)$
is a tq-Hamiltonian $G$-space with $G[\tau,\kappa]$-valued moment
map. 
\end{example}

More generally for $b,h\ge1$, the connected bordered surface $\Sigma_{h}^{b}$
is obtained by ``fusing'' $h$ one-holed tori $\Sigma_{1}^{1}$
and $(b-1)$ cylinders $\Sigma_{0}^{2}$, and we can hence identify
the corresponding moduli space with a product: 
\begin{equation}
M_{\sigma}(\Sigma_{h}^{b},G)=\underbrace{\left(\mathbb{D}_{\varphi_{1}}(G)\circledast\cdots\circledast\mathbb{D}_{\varphi_{h}}(G)\right)}_{h}\circledast\underbrace{\left(D_{\sigma_{1}}(G)\circledast\cdots\circledast D_{\sigma_{b-1}}(G)\right)}_{(b-1)}\label{Eq_TMS_GenSurface_Fusion}
\end{equation}
for appropriate twists $\sigma_{i},\varphi_{j}$.

For the symplectic reduction of the spaces $M_{\sigma}(\Sigma_{h}^{b},G)$,
let $\mathcal{C}_{i}$ be a $\kappa_{i}$-twisted conjugacy class
in $G$, and let: 
\[
\vec{\mathcal{C}}=(\mathcal{C}_{1},\cdots,\mathcal{C}_{b}).
\]
By Proposition \ref{Prop_Reduction}, the reduced space: 
\[
\mathcal{M}_{\sigma}(\Sigma,\vec{\mathcal{C}}):=\left(M_{\sigma}(\Sigma,G)\circledast\mathcal{C}_{1}^{-}\circledast\cdots\circledast\mathcal{C}_{b}^{-}\right)/\!/G^{b}.
\]
is a singular symplectic space, and two of its main properties are
that:
\begin{itemize}
\item[(i)] As a moduli space of flat connections, $\mathcal{M}_{\sigma}(\Sigma,\vec{\mathcal{C}})$
parametrizes the connections for which the $i$th boundary holonomy
takes values in the conjugacy class $\mathcal{C}_{i}\subseteq G\kappa_{i}$.
\item[(ii)] If $\overline{\Sigma}$ is the surface obtained by capping-off the
boundary components of $\Sigma$, the representation variety: 
\[
\mathrm{Hom}_{\sigma}\left(\pi_{1}(\overline{\Sigma}),G\right)/G\cong M_{\sigma}(\Sigma,G)/\!/\mathrm{Map}(S,G)
\]
carries a natural Poisson structure \cite{[Boa14],[BY15]}, for which
the symplectic leaves are precisely the quotients $\mathcal{M}_{\sigma}(\Sigma,\vec{\mathcal{C}})$.
\end{itemize}
\begin{rem}
We discussed tq-Hamiltonian manifolds $(M,\omega,\Phi)$ for which
the moment map takes values in $G\rtimes\mathrm{Aut}(G)$, but where
only the identity component $G$ acts on $M$. A natural question
that arises is whether it is possible to develop a theory where \textit{the
disconnected group $G\rtimes\mathrm{Aut}(G)$ acts on }$M$. Although
this is an open question we hope to address in other work, we note
that it is possible to construct such examples from character varieties.
For instance, consider the surface $\Sigma=\Sigma_{2}^{2}$ and $\kappa\in\mathrm{Aut}(G)$
with $|\kappa|=2$. One can implement an action of $\langle\kappa\rangle$
on $\Pi_{1}(\Sigma,\{p_{1},p_{2}\})$, by letting $\kappa$ permute
the boundary circles and the generators $\{A_{i},B_{i}\}$ of the
handles, which gives an action of $(G\rtimes\langle\kappa\rangle)^{2}$
on $\mathrm{Hom}(\Pi,G\rtimes\langle\kappa\rangle)$.
\end{rem}

\section{Duistermaat-Heckman Measures and Localization\label{Sec_DH_Meas}}

This section studies the Duistermaat-Heckman measure associated to
a twisted q-Hamiltonian manifold, by extending some of the main results
of \cite{[AMW00],[AMW02]} to our setup. To put things into context,
let $(M,\omega,\Phi)$ be a Hamiltonian $G$-manifold with Liouville
form $\Lambda_{M}$ and DH measure $\mathrm{DH}_{\Phi}=\Phi_{\ast}|\Lambda_{M}|\in\mathcal{D}'(\mathfrak{g}^{\ast})^{G}$.
The Duistermaat-Heckman localization theorem states that at $\xi\in\mathfrak{g}$,
the Fourier transform of $\mathrm{DH}_{\Phi}$ localizes to integrals
over the connected components of the vanishing set $(\xi_{M})^{-1}(0)$.
In equation form \cite{[GGKBk],[ParHDR],[GLSBk],[BGV]}:
\begin{equation}
\int_{\mathfrak{g}^{\ast}}e^{2\pi\mathsf{i}\langle\mu,\xi\rangle}d\mathrm{DH}_{\Phi}(\mu)=\sum_{Z\subset\xi_{M}^{-1}(0)}\int_{Z}\frac{e^{\iota_{Z}^{\ast}(\omega+\langle\Phi,2\pi\mathsf{i}\xi\rangle)}}{\mathrm{Eul}(\mathcal{N}_{Z},2\pi\mathsf{i}\xi)},\label{Eq_DH_Localization}
\end{equation}
where $\iota_{Z}:Z\hookrightarrow M$ is the inclusion and $\mathrm{Eul}(\mathcal{N}_{Z},\cdot)$
is the equivariant Euler form of the normal bundle $\mathcal{N}_{Z}=TM|_{Z}/TZ$.

The main result of this work, Theorem \ref{Thm_Twisted_DH}, generalizes
equation (\ref{Eq_DH_Localization}) to the case of a $G\kappa$-valued
tq-Hamiltonian manifold $(M,\omega,\Phi)$. Here, $\mathrm{DH}_{\Phi}$
is an $\mathrm{Ad}_{G}^{\kappa}$-invariant measure on $G$, and is
given by a Fourier series in the $\kappa$-twining characters $\{\tilde{\chi}_{\lambda}^{\kappa}\}\subset L^{2}(G\kappa)^{G}$.
The RHS of equation (\ref{Eq_DH_Localization}) is replaced by a Fourier
coefficient $\langle\mathrm{DH}_{\Phi},\tilde{\chi}_{\lambda}^{\kappa}\rangle$,
while the localized integrals of the LHS now involve subgroups of
the max torus $T\subseteq G$ defined by $\kappa\in\mathrm{Aut}(G)$.

The upcoming subsections are organized as follows. Section \ref{Subsec_Twining_Char}
reviews twining characters and their properties. Section \ref{Subsec_DH_Basics}
discusses the basics of DH measures of tq-Hamiltonian manifold. One
of the delicate points there is the construction of the corresponding
Liouville form $\Lambda_{M}$, which is carried out in section \ref{Subsec_Dirac_Str_Pure_Spinors}.
Section \ref{Subsec_DH_Localization} deals with the localization
theorem, and also builds upon results in section \ref{Subsec_Dirac_Str_Pure_Spinors}.
Finally, we return to twisted moduli spaces in section \ref{SubSec_Application_TMS},
where we compute their DH measures.

\subsection{Twining characters\label{Subsec_Twining_Char}}

In the entirety of this section, $G$ denotes a compact, connected,
simply connected and simple Lie group. Unless otherwise stated, $\kappa\in\mathrm{Out}(G)$
is induced by a Dynkin diagram automorphism, and as explained in section
\ref{SubSec_Twist_Conj}, we view $\mathrm{Out}(G)$ as a subgroup
of $\mathrm{Aut}(G)$. We continue with the notation introduced at
the beginning of section \ref{SubSec_Twist_Conj}.

\subsubsection{Notation}

For a dominant weight $\lambda\in\Lambda_{+}^{\ast}$, we denote the
corresponding irreducible representation by $(\rho_{\lambda},V_{\lambda})$,
and the corresponding irreducible character by $\chi_{\lambda}$.
Letting $v_{\lambda}\in V_{\lambda}$ denote the normalized highest
weight vector, we will denote by: 
\begin{equation}
\Delta_{\lambda}(g):=\langle v_{\lambda},\rho_{\lambda}(g)\cdot v_{\lambda}\rangle,\ \ \forall g\in G,\label{Eq_Spherical_Harmonic_lambda}
\end{equation}
the ``spherical harmonic'' function corresponding to $\lambda\in\Lambda_{+}^{\ast}$.
Letting $\rho=\frac{1}{2}\sum_{\alpha\in\mathfrak{R}_{+}}\alpha$
denote the half-sum of positive roots of $G$, we note that $\rho\in(\mathfrak{t}^{\ast})^{\kappa}$.
Lastly, we denote the $\kappa$-fixed dominant weights by $(\Lambda_{+}^{\ast})^{\kappa}=\Lambda_{+}^{\ast}\cap(\mathfrak{t}^{\ast})^{\kappa}$.

\subsubsection{Definitions}

Let $(\rho_{V},V)$ be a representation of $G$. We say that such
a representation is $\bm{\kappa}$\textbf{-admissible} if it admits
an implementation of $\kappa$ on $V$ that is compatible with the
action of $G$, that is, if there exists a unitary operator $\tilde{\kappa}_{V}\in\mathrm{Aut}(V)$
such that:
\[
\rho_{V}\left(\kappa(g)\right)=\tilde{\kappa}_{V}\circ\rho_{V}(g)\circ\tilde{\kappa}_{V}^{-1},\ \ \forall g\in G.
\]
These representations are closed under direct sums and tensor products.
In the particular case that $V=V_{\lambda}$ for $\lambda\in(\Lambda_{+}^{\ast})^{\kappa}$,
Schur's lemma gives the existence of a unique unitary automorphism
$\tilde{\kappa}_{\lambda}\in\mathrm{Aut}(V_{\lambda})$ satisfying
the equation above and preserving the highest weight vector $v_{\lambda}\in V_{\lambda}$.
\begin{defn}
\label{Def_Twining_Char}Let $(\rho_{V},V)$ be a $\kappa$-admissible
representation of $G$, and let $\tilde{\kappa}_{V}\in\mathrm{Aut}(V)$
be an implementation restricting to $\tilde{\kappa}_{\lambda}$ on
each irreducible summand $V_{\lambda}$ with $\lambda\in(\Lambda_{+}^{\ast})^{\kappa}$.
The $\bm{\kappa}$\textbf{-twining character} of $(\rho_{V},V)$ is
the function $G\rightarrow\mathbb{C}$ given for any $g\in G$ by:
\[
\tilde{\chi}_{V}^{\kappa}(g):=\mathrm{tr}_{_{V}}\left(\tilde{\kappa}_{V}\circ\rho_{V}(g)\right).
\]
For $V=V_{\lambda}$ irreducible, we denote by $\tilde{\chi}_{\lambda}^{\kappa}$
the corresponding twining character.
\end{defn}

For a $\kappa$-admissible representation $(\rho_{V},V)$, the non-vanishing
contributions to the function $\tilde{\chi}_{V}^{\kappa}$ only come
from the irreducible factors $V_{\lambda}\subseteq V$ such that $\lambda\in(\Lambda_{+}^{\ast})^{\kappa}$,
since the implementation $\tilde{\kappa}_{V}\in\mathrm{Aut}(V)$ necessarily
acts a permutation on the remaining irreducible factors \cite[\S 4.1]{[Zer18a]}.
For a second $\kappa$-admissible representation $(\rho_{V'},V')$,
one easily checks that:
\[
\tilde{\chi}_{V\oplus V'}^{\kappa}=\tilde{\chi}_{V}^{\kappa}+\tilde{\chi}_{V'}^{\kappa},\ \ \tilde{\chi}_{V\otimes V'}^{\kappa}=\tilde{\chi}_{V}^{\kappa}\cdot\tilde{\chi}_{V'}^{\kappa},
\]
as it is the case with usual characters.
\begin{rem}
\label{Rk_Twining_Char_Discon_Gp}An alternative way of introducing
$\kappa$-twining characters uses representations of the disconnected
group $G\rtimes\langle\kappa\rangle$. A detailed reference for this
material is Mohrdieck's \cite[\S 2.4]{[Mo00]}. By \cite[Prop.2.7]{[Mo00]},
an irrep $V$ of $G\rtimes\langle\kappa\rangle$ is either of the
form $V=\oplus_{i=1}^{|\kappa|}V_{\kappa^{i}(\lambda)}$ for $\lambda\in\Lambda_{+}^{\ast}$
such that $\kappa(\lambda)\ne\lambda$, or of the form $V=V_{\lambda}$
for $\lambda\in(\Lambda_{+}^{\ast})^{\kappa}$, with $|\kappa|$ inequivalent
homomorphisms $\rho_{\lambda,j}:G\rtimes\langle\kappa\rangle\to\mathrm{Aut}(V_{\lambda})$,
coming from the distinct implementations $\exp(\frac{2\pi\mathsf{i}}{|\kappa|}j)\tilde{\kappa}_{\lambda}$
of $\kappa$ on $V_{\lambda}$. 

Under the identification $G\kappa\equiv G$, the $\kappa$-admissible
representations of $G$ are the restrictions to the component $G\kappa$
of the representations of $G\rtimes\langle\kappa\rangle$. At the
level of irreducible characters, the restriction $\chi_{V}|_{G\kappa}$
vanishes for $\kappa(\lambda)\ne\lambda$ \cite[Prop.2.8]{[Mo00]},
while $\chi_{V}|_{G\kappa}\equiv\exp(\frac{2\pi\mathsf{i}}{|\kappa|}j)\tilde{\chi}_{\lambda}^{\kappa}$
for $\lambda\in(\Lambda_{+}^{\ast})^{\kappa}$, $j=1,\cdots,|\kappa|$. 

The reason for introducing twining characters is that we prefer to
think in terms of class functions for $\mathrm{Ad}_{G}^{\kappa}$
on the group $G$, as opposed to restrictions of functions defined
on $G\rtimes\langle\kappa\rangle$. This is also the motivation for
introducing $\kappa$-admissible representations.
\end{rem}

\subsubsection{Twining characters as $L^{2}$ functions}

Let $L^{2}(G)$ be the space of $\mathbb{C}$-valued $L^{2}$-functions
on $G$ with respect to the normalized Haar measure $dg$, and recall
the convolution product on continuous functions:
\[
\left(\psi\ast\varphi\right)(x)=\int_{G}\psi(xg^{-1})\varphi(g)dg,\ \ \psi,\varphi\in\mathcal{C}^{0}(G,\mathbb{C}).
\]
Let $L^{2}(G\kappa)^{G}\subset L^{2}(G)$ denote the subspace of $\mathrm{Ad}_{G}^{\kappa}$-invariant
functions. The irreducible twining characters $\{\tilde{\chi}_{\lambda}^{\kappa}\}_{\lambda\in(\Lambda_{+}^{\ast})^{\kappa}}$
extend several properties of the usual characters:
\begin{prop}
\label{Prop_Twining_Char_L2_Class_Fns}With the notations of this
section: 
\begin{enumerate}
\item The averaging map $\mathrm{Av}^{\kappa}:L^{2}(G)\rightarrow L^{2}(G\kappa)^{G}\mbox{, }f\mapsto\int_{G}((\mathrm{Ad}_{g}^{\kappa})^{\ast}f)dg$
is an orthogonal projection.
\item The twining characters $\{\tilde{\chi}_{\lambda}^{\kappa}\}_{\lambda\in(\Lambda_{+}^{\ast})^{\kappa}}$
generate a dense subspace of $L^{2}(G\kappa)^{G}$, and satisfy the
orthogonality relations: 
\[
\langle\tilde{\chi}_{\lambda}^{\kappa},\tilde{\chi}_{\mu}^{\kappa}\rangle_{L^{2}}=\delta_{\mu\lambda},\ \ \lambda,\mu\in(\Lambda_{+}^{\ast})^{\kappa}.
\]
\item For a second Dynkin diagram automorphism $\tau\in\mathrm{Out}(G)$,
one has the identities: 
\begin{equation}
\left(\tilde{\chi}_{\lambda}^{\kappa}\ast\tilde{\chi}_{\mu}^{\tau}\right)(x)=(\dim V_{\lambda})^{-1}\delta_{\lambda\mu}\tilde{\chi}_{\lambda}^{\tau\kappa}(x)=(\dim V_{\lambda})^{-1}\delta_{\lambda\mu}\tilde{\chi}_{\lambda}^{\kappa\tau}\left(\kappa(x)\right),\ \ \forall x\in G,\label{Eq_Twining_Convolution}
\end{equation}
\begin{equation}
\left(\tilde{\chi}_{\lambda}^{\kappa}\otimes\tilde{\chi}_{\lambda}^{\tau}\right)(x,y)=\dim V_{\lambda}\int_{G}\tilde{\chi}_{\lambda}^{(\kappa_{1}\kappa_{2})}\left(gx\kappa_{1}(g^{-1})\kappa_{1}(y)\right)dg,\ \ \forall(x,y)\in G\times G,\label{Eq_Prod_Twining}
\end{equation}
for all $\lambda,\mu\in(\Lambda_{+}^{\ast})^{\kappa}\cap(\Lambda_{+}^{\ast})^{\tau}$.
\item For $\lambda\in(\Lambda_{+}^{\ast})^{\kappa}$, the spherical harmonic
$\Delta_{\lambda}:G\rightarrow\mathbb{C}$ of eq. (\ref{Eq_Spherical_Harmonic_lambda})
satisfies the identity: 
\[
\int_{G}\Delta_{\lambda}(\mathrm{Ad}_{g}^{\kappa}x)dg=(\dim V_{\lambda})^{-1}\tilde{\chi}_{\lambda}^{\kappa}(x).
\]
\end{enumerate}
\end{prop}

We only prove equation (\ref{Eq_Prod_Twining}) here, since the remaining
statements of are examined in the proof of \cite[Prop.4.4]{[Zer18a]}.

\noindent\emph{Proof of eq.(\ref{Eq_Prod_Twining})} The function:
\[
f(x,y)=\dim V_{\lambda}\int_{G}\tilde{\chi}_{\lambda}^{(\kappa\tau)}\left(gx\kappa(g^{-1})\kappa(y)\right)dg,
\]
is an $\mathrm{Ad}_{G}^{\kappa}\times\mathrm{Ad}_{G}^{\tau}$-invariant
$L^{2}$ function on $G\times G$. It suffices to show that $\langle\tilde{\chi}_{\lambda}^{\kappa}\otimes\tilde{\chi}_{\lambda}^{\tau},f\rangle=1$,
which implies by the orthogonality and density of twining characters
that $f\equiv\tilde{\chi}_{\lambda}^{\kappa}\otimes\tilde{\chi}_{\lambda}^{\tau}$.
Using the invariance of the Haar measure under inversion and $\mathrm{Out}(G)$,
along with the convolution of twining characters and the identities:
\[
\overline{\tilde{\chi}_{\lambda}^{\tau}(y)}=\tilde{\chi}_{\lambda}^{(\tau^{-1})}(y^{-1})=\tilde{\chi}_{\lambda}^{(\tau^{-1})}\left(\tau^{-1}(y^{-1})\right),
\]
one computes that for any $h\in G$:
\[
\int_{G}\tilde{\chi}_{\lambda}^{(\kappa\tau)}\left(h\kappa(y)\right)\overline{\tilde{\chi}_{\lambda}^{\tau}(y)}dy=\left(\tilde{\chi}_{\lambda}^{(\kappa\tau)}\ast\tilde{\chi}_{\lambda}^{(\tau^{-1})}\right)\left((\kappa\tau)^{-1}(h)\right)=(\dim V_{\lambda})^{-1}\tilde{\chi}_{\lambda}^{\kappa}(h).
\]
Using this last equation with $h=gx\kappa(g^{-1})$, we now have:
\begin{eqnarray*}
\langle f,\tilde{\chi}_{\lambda}^{\kappa}\otimes\tilde{\chi}_{\lambda}^{\tau}\rangle & = & \dim V_{\lambda}\int_{G\times G}\left[\int_{G}\tilde{\chi}_{\lambda}^{(\kappa\tau)}\left(gx\kappa(g^{-1})\kappa(y)\right)\overline{\tilde{\chi}_{\lambda}^{\tau}(y)}dy\right]\overline{\tilde{\chi}_{\lambda}^{\kappa}(x)}dgdx\\
 & = & \int_{G\times G}\tilde{\chi}_{\lambda}^{\kappa}\left(gx\kappa(g^{-1})\right)\overline{\tilde{\chi}_{\lambda}^{\kappa}(x)}dgdx=\left(\int_{G}dg\right)\langle\tilde{\chi}_{\lambda}^{\kappa},\tilde{\chi}_{\lambda}^{\kappa}\rangle=1,
\end{eqnarray*}
where we used the $\mathrm{Ad}_{G}^{\kappa}$-invariance of $\tilde{\chi}_{\lambda}^{\kappa}$
in the last equality. \qed

\subsection{Definitions and basic properties\label{Subsec_DH_Basics}}

Let $(M,\omega,\Phi)$ denote a $G\kappa$-valued tq-Hamiltonian manifold.
The data of $G$ and $\kappa\in\mathrm{Out}(G)$ gives rise to a distinguished
$\mathrm{Ad}_{G}^{\kappa}$-invariant form $\psi_{G}^{\kappa}\in\Omega(G)$
(see eq.(\ref{Eq_Cartan-Dirac_Spinor_Psi_kappa}), and section \ref{Subsec_Dirac_Str_Pure_Spinors}
for the construction). The \textbf{Liouville form} associated to $(M,\omega,\Phi)$
is the $G$-invariant form defined by:
\[
\Lambda_{M}:=\left(e^{\omega}\Phi^{\ast}\psi_{G}^{\kappa}\right)_{[\mathrm{top}]}\in\Omega^{[\mathrm{top}]}(M)^{G}.
\]
In parallel to the symplectic context, we have:
\begin{defn}
The \textbf{Duistermaat-Heckman (DH) measure} associated to $(M,\omega,\Phi)$
is defined as the pushforward:
\[
\mathrm{DH}_{\Phi}:=\Phi_{\ast}|\Lambda_{M}|\in\mathcal{D}'(G\kappa)^{G},
\]
where $|\Lambda_{M}|\in\mathcal{D}'(M)^{G}$ is the induced Liouville
measure, and $\mathcal{D}'(G\kappa)^{G}$ is the space of $\mathrm{Ad}_{G}^{\kappa}$-invariant
distributions on $G$.
\end{defn}

By construction, the singular support of $\mathrm{DH}_{\Phi}$ coincides
with the singular values of the map $\Phi:M\to G\kappa$. Outside
of that set, we think of the DH measure as (locally) given by a Fourier
series:
\begin{equation}
\mathrm{DH}_{\Phi}=\tfrac{1}{\mathrm{vol}_{G}}\Big(\sum_{\lambda\in(\Lambda_{+}^{\ast})^{\kappa}}\langle\mathrm{DH}_{\Phi},\tilde{\chi}_{\lambda}^{\kappa}\rangle\overline{\tilde{\chi}_{\lambda}^{\kappa}}\Big)d\mathrm{vol}_{G},\label{Eq_DH_Fourier_series}
\end{equation}
where $d\mathrm{vol}_{G}$ is the Riemannian measure on $G$. (see
Remark \ref{Rk_Comments_DH_measure} for comments on this point).

For $i=1,2$, let $\kappa_{i}\in\mathrm{Out}(G)$, and consider $G\kappa_{i}$-valued
tq-Hamiltonian $G$-manifolds $(M_{i},\omega_{i},\Phi_{i})$ with
fusion product $(M_{1}\circledast M_{2},\omega_{\mathrm{fus}},\Phi_{\mathrm{fus}})$,
where:
\[
\Phi_{\mathrm{fus}}:M_{1}\circledast M_{2}\to G\kappa_{1}\kappa_{2},(x,y)\mapsto\mathrm{Mult}\left(\Phi_{1}(x)\times\kappa_{1}(\Phi_{2}(y))\right).
\]
 Under fusion, DH measures behave as follows:
\begin{prop}
\label{Prop_Convolution_DH_Meas_Fusion} With the notation above,
one has that:
\[
\mathrm{DH}_{\Phi_{\mathrm{fus}}}=\mathrm{DH}_{\Phi_{1}}\ast(\kappa_{1})_{\ast}\mathrm{DH}_{\Phi_{2}}\in\mathcal{D}'\left(G\kappa_{1}\kappa_{2}\right)^{G}.
\]
For any $\lambda\in(\Lambda_{+}^{\ast})^{\kappa_{1}}\cap(\Lambda_{+}^{\ast})^{\kappa_{2}}$,
the Fourier coefficient of $\mathrm{DH}_{\Phi_{\mathrm{fus}}}$ corresponding
to $\widetilde{\chi}_{\lambda}^{(\kappa_{1}\kappa_{2})}$ is given
by:
\[
\langle\mathrm{DH}_{\Phi_{\mathrm{fus}}},\tilde{\chi}_{\lambda}^{(\kappa_{1}\kappa_{2})}\rangle=(\dim V_{\lambda})^{-1}\langle\mathrm{DH}_{\Phi_{1}},\tilde{\chi}_{\lambda}^{\kappa_{1}}\rangle\langle\mathrm{DH}_{\Phi_{2}},\tilde{\chi}_{\lambda}^{\kappa_{2}}\rangle.
\]
\end{prop}

\begin{proof}
For the Liouville forms, we have by Corollary \ref{Cor_Fusion_prod_Liouville}
that:
\[
\Lambda_{M_{1}\circledast M_{2}}=\Lambda_{M_{1}}\otimes\Lambda_{M_{2}}\in\Omega^{[\mathrm{top}]}(M_{1}\times M_{2})^{G}.
\]
The definitions of $\Phi_{\mathrm{fus}}$ and $\mathrm{DH}_{\Phi_{\mathrm{fus}}}$
yield:
\[
\mathrm{DH}_{\Phi_{\mathrm{fus}}}=\mathrm{Mult}_{\ast}\left(\mathrm{DH}_{\Phi_{1}}\otimes(\kappa_{1})_{\ast}\mathrm{DH}_{\Phi_{2}}\right),
\]
and since convolution of measures on $G$ is given by pushforward
under the multiplication map $\mathrm{Mult}:G\times G\to G$, we obtain
the first equation of the statement. The second equation follows from
the computation:
\begin{eqnarray*}
\langle\mathrm{DH}_{\Phi_{\mathrm{fus}}},\tilde{\chi}_{\lambda}^{(\kappa_{1}\kappa_{2})}\rangle & = & \int_{G\times G\times G}\tilde{\chi}_{\lambda}^{(\kappa_{1}\kappa_{2})}\left(x\kappa_{1}(y)\right)dgd\mathrm{DH}_{\Phi_{1}}\left(g^{-1}x\kappa_{1}(g)\right)d\mathrm{DH}_{\Phi_{2}}(y)\\
 & = & \int_{G\times G}\left[\int_{G}\tilde{\chi}_{\lambda}^{(\kappa_{1}\kappa_{2})}\left(gx\kappa_{1}(g^{-1})\kappa_{1}(y)\right)dg\right]d\mathrm{DH}_{\Phi_{1}}(x)d\mathrm{DH}_{\Phi_{2}}(y)\\
 & = & (\dim V_{\lambda})^{-1}\langle\mathrm{DH}_{\Phi_{1}},\tilde{\chi}_{\lambda}^{\kappa_{1}}\rangle\langle\mathrm{DH}_{\Phi_{2}},\tilde{\chi}_{\lambda}^{\kappa_{2}}\rangle,
\end{eqnarray*}
where the third equality comes from equation (\ref{Eq_Prod_Twining}).
\end{proof}
\begin{rem}
Suppose that we have a fusion product $M_{1}\circledast\cdots\circledast M_{r}$,
where each $M_{i}$ is a $G\kappa_{i}$-valued tq-Hamiltonian manifold.
The above generalizes \cite[Eq.(30)]{[AMW02]}) to: 
\[
\langle\mathrm{DH}_{\Phi_{\mathrm{fus}}},\tilde{\chi}_{\lambda}^{(\kappa_{1}\cdots\kappa_{r})}\rangle=(\dim V_{\lambda})^{1-r}\prod_{j=1}^{r}\langle\mathrm{DH}_{\Phi_{i}},\tilde{\chi}_{\lambda}^{\kappa_{i}}\rangle,\ \ \lambda\in\bigcap_{i=1}^{r}(\Lambda_{+}^{\ast})^{\kappa_{i}}.
\]
\end{rem}

The next point we address is the relation between the DH measure and
the volumes of reduced spaces $M_{a}=\Phi^{-1}(a)/Z_{a}^{\kappa}$.
Recall that the conjugacy classes of stabilizer subgroups $H\subset G$
induce the the \textit{orbit type stratification} $M=\cup_{(H)}M_{(H)}$
\cite[\S A.1]{[MW99]}, and that there exists a unique open dense
$M_{\mathrm{prin}}\subset M$ called the \textit{principal stratum}.
The latter corresponds to the smallest conjugacy class of stabilizers
$(\Gamma)$, where $\Gamma$ is called a \textit{principal stabilizer}.
\begin{prop}
\label{Prop_Radon-Nikodym}Let $(M,\omega,\Phi)$ be a $G\kappa$-valued
tq-Hamiltonian manifold with finite principal stabilizer $\Gamma$.
Let $a\in G\kappa$ be such that $\Phi^{-1}(a)$ intersects $M_{\mathrm{prin}}$,
and let $\mathcal{C}=\mathrm{Ad}_{G}^{\kappa}(a)$. At $a\in G\kappa$,
the Radon-Nikodym derivative of $\mathrm{DH}_{\Phi}$ with respect
to the Riemannian measure on $G$ is given by:
\[
\frac{d\mathrm{DH}_{\Phi}}{d\mathrm{vol}_{G}}(a)=\frac{1}{|\Gamma|}\frac{\mathrm{vol}_{G}}{\mathrm{Vol}(\mathcal{C})}\mathrm{Vol}(M_{a}).
\]
\end{prop}

\begin{proof}
We have the untwisted q-Hamiltonian manifold $Y=M\circledast\mathcal{C}^{-}$
with moment map $\Psi(x,h)=\Phi(x)h^{-1}$. The symplectic quotient
$M_{a}=\Phi^{-1}(\mathcal{C})/G$ coincides with $Y/\!/G=\Psi^{-1}(e)/G$,
and the principal stabilizers of both $\Psi^{-1}(e)$ and $\Phi^{-1}(a)$
are conjugate to $\Gamma$. By \cite[Thm.4.1]{[AMW02]} then:
\[
\frac{d\mathrm{DH}_{\Psi}}{d\mathrm{vol}_{G}}(e)=\tfrac{1}{|\Gamma|}\mathrm{vol}_{G}\mathrm{Vol}(M_{a}).
\]
The RHS of this equation is determined from the convolution of DH
measures under fusion:
\[
\frac{d\mathrm{DH}_{\Psi}}{d\mathrm{vol}_{G}}(e)=\langle\mathrm{DH}_{\Phi}\ast\mathrm{Inv}_{\ast}\mathrm{DH}_{\mathcal{C}},\delta_{e}\rangle=\mathrm{Vol}(\mathcal{C})\langle\mathrm{DH}_{\Phi},\delta_{a}\rangle=\mathrm{Vol}(\mathcal{C})\frac{d\mathrm{DH}_{\Phi}}{d\mathrm{vol}_{G}}(a).\qedhere
\]
\end{proof}
\begin{rem}
\label{Rk_Comments_DH_measure}We conclude this subsection with some
comments.

\noindent\begin{itemize}

\item[(a)] The validity of the Fourier series in eq.(\ref{Eq_DH_Fourier_series})
for $\mathrm{DH}_{\Phi}$ follows from the discussion in the last
paragraphs of \cite[\S 4.1]{[AMW02]}, which is based on the approach
in \cite[\S 3]{[Liu96]}. Consider the Laplace-Beltrami operator $\Delta$
on the disconnected group $G\rtimes\langle\kappa\rangle$. For $t>0$,
apply the smoothing operator to $\mathrm{DH}_{\Phi}$ and set:
\[
\tfrac{d\mathrm{DH}_{\Phi}}{d\mathrm{vol}_{G}}\equiv\tfrac{1}{|\kappa|\mathrm{vol}_{G}}\lim_{t\to0^{+}}\sum_{[V]\in\mathrm{Irr}(G\rtimes\langle\kappa\rangle)}\langle\mathrm{DH}_{\Phi},e^{-t\Delta}\chi_{V}\rangle\overline{\chi_{V}}.
\]
Since $\mathrm{DH}_{\Phi}$ is supported on the component $G\kappa$,
we have by Remark \ref{Rk_Twining_Char_Discon_Gp} that $\langle\mathrm{DH}_{\Phi},\chi_{V}\rangle$
vanishes for $V$ corresponding to $\lambda\in\Lambda_{+}^{\ast}\smallsetminus(\Lambda_{+}^{\ast})^{\kappa}$
, and $\langle\mathrm{DH}_{\Phi},\chi_{V}\rangle=\exp(\frac{2\pi\mathsf{i}}{|\kappa|}j)\langle\mathrm{DH}_{\Phi},\tilde{\chi}_{\lambda}^{\kappa}\rangle$
for $V$ corresponding to $\lambda\in(\Lambda_{+}^{\ast})^{\kappa}$,
$j=1,\cdots,|\kappa|$. The previous equation then becomes:
\[
\tfrac{d\mathrm{DH}_{\Phi}}{d\mathrm{vol}_{G}}\equiv\tfrac{1}{\mathrm{vol}_{G}}\lim_{t\to0^{+}}\sum_{\lambda\in(\Lambda_{+}^{\ast})^{\kappa}}e^{-tp(\lambda)}\langle\mathrm{DH}_{\Phi},\tilde{\chi}_{\lambda}^{\kappa}\rangle\overline{\tilde{\chi}_{\lambda}^{\kappa}},
\]
where $p(\lambda)=|\!|\lambda+\rho|\!|^{2}-|\!|\rho|\!|^{2}$ is the
eigenvalue of $\Delta$ on the character $\chi_{\lambda}$. This regularization
gives a rigorous meaning to equation (\ref{Eq_DH_Fourier_series}),
and is necessary when that sum diverges.

\item[(b)] The conditions on the principal stratum and stabilizer
in Proposition \ref{Prop_Radon-Nikodym} are there to guarantee that
the formula holds for singular values of the moment map (see \cite[\S A.1]{[MW99]}).
If we are merely interested in a regular value $a\in G\kappa$ of
$\Phi$, then $\mathrm{Im}(\Phi^{\ast}\theta^{R})_{x}=\mathfrak{g}_{x}^{\perp}=\mathfrak{g}$
for all $x\in\Phi^{-1}(a)$ by \cite[Prop.3.9]{[Mein17]}, and the
stabilizers $G_{x}\subset Z_{a}^{\kappa}$ are thus finite. The number
$|\Gamma|$ in the formula of Proposition \ref{Prop_Radon-Nikodym}
is then replaced by the cardinality of the subgroup $\cap_{x\in\Phi^{-1}(a)}G_{x}\subset Z_{a}^{\kappa}$.

\end{itemize}
\end{rem}

\subsection{Localization and Fourier coefficients\label{Subsec_DH_Localization}}

The main result of this work is the following:
\begin{thm}
\label{Thm_Twisted_DH}Let $G$ be a compact 1-connected simple Lie
group with $\kappa\in\mathrm{Out}(G)$, and let $(M,\omega,\Phi)$
be a $G\kappa$-valued tq-Hamiltonian manifold. For a $\kappa$-fixed
dominant weight $\lambda\in(\Lambda_{+}^{\ast})^{\kappa}$, let $\xi_{\lambda}=2\pi iB^{\sharp}(\rho+\lambda)\in\mathfrak{t}^{\kappa}$.The
Fourier coefficient of $\mathrm{DH}_{\Phi}$ corresponding to $\lambda$
is given by the localization formula:
\begin{equation}
\frac{\langle\mathrm{DH}_{\Phi},\tilde{\chi}_{\lambda}^{\kappa}\rangle}{\dim V_{\lambda}}=\sum_{Z\subseteq(\xi_{\lambda})_{M}^{-1}(0)}\int_{Z}\frac{(\Phi_{Z})^{\lambda+\rho}e^{\omega_{Z}}\Phi_{Z}^{\ast}(\phi_{\kappa})}{\mathrm{Eul}(\mathcal{N}_{Z},\xi_{\lambda})},\label{Eq_Fourier_Coeff_Twisted_DH}
\end{equation}
where: \begin{itemize}

\item $Z$ is a connected component of the vanishing set $(\xi_{\lambda})_{M}^{-1}(0)\subseteq\Phi^{-1}(T)$;

\item $\omega_{Z}$ and $\Phi_{Z}$ are the pullbacks of $\omega$
and $\Phi$ to $Z\subset\Phi^{-1}(T)$;

\item $\mathrm{Eul}(\mathcal{N}_{Z},\cdot)$ is the $T$-equivariant
Euler form of the normal bundle $\mathcal{N}_{Z}=TM|_{Z}/TZ$;

\item The function $(\Phi_{Z})^{\lambda+\rho}$ designates $(\Phi_{Z}(z))^{\lambda+\rho}=e^{2\pi i\langle\lambda+\rho,\zeta\rangle}$,
for $z\in Z$ and $\zeta\in\mathfrak{t}$ such that $\Phi_{Z}(z)=e^{\zeta}$.

\item The differential form $\phi_{\kappa}\in\Omega(G)$ is given
by:
\[
\phi_{\kappa}=\begin{cases}
2^{-\frac{1}{2}\dim T_{\kappa}}\exp(\varpi)\wedge d\mathrm{vol}_{T_{\kappa}}, & |\kappa|=2;\\
2^{-1}\exp\left(\varpi+\tfrac{\sqrt{3}}{2}d\mathrm{vol}_{T_{\kappa}}\right), & |\kappa|=3;
\end{cases}
\]
with $d\mathrm{vol}_{T_{\kappa}}$ denoting the Riemannian volume
form of the subtorus $T_{\kappa}\subseteq T$, and $\varpi\in\Omega^{2}(G)_{\mathbb{C}}$
the form introduced in eq. (\ref{Eq_varpi_Gauss_cell}).

\end{itemize}
\end{thm}

The proof below is heavily based on the material of section \ref{Subsec_Dirac_Str_Pure_Spinors},
and can therefore be skipped on a first reading.
\begin{proof}
\textbf{Step 1:} The connected components $Z\subseteq(\xi_{\lambda})_{M}^{-1}(0)$
are tq-Hamiltonian $T$-spaces. For any $z\in Z$, $\xi_{\lambda}\in(\mathfrak{t}^{\kappa})_{\mathbb{C}}$
implies that:
\[
\left(\kappa(\xi_{\lambda})^{L}-\xi_{\lambda}^{R}\right)_{\Phi(z)}=\left((\mathrm{Ad}_{\Phi(z)}-1)\xi_{\lambda}\right)_{\Phi(z)}^{R}=0,
\]
and since $\xi_{\lambda}$ is a regular element of $\mathfrak{t}_{\mathbb{C}}$
under $\mathrm{Ad}_{G}$, we have that $\Phi(z)\in T$. Thus, $Z\subseteq\Phi^{-1}(T)$,
and the restriction of the $G$ action to the action of $T$ preserves
these components.

\noindent \textbf{Step 2:} Consider the equivariant form:
\[
\widehat{\Lambda}_{M}^{\lambda}:=e^{\omega}\Phi^{\ast}(\Delta_{\lambda}\widehat{\psi}_{G}^{\kappa})\in\Omega_{G}(M),
\]
where $\widehat{\psi}_{G}^{\kappa}\in\Omega(G)$ is the restriction
of the Gauss-Dirac spinor of eq. (\ref{Eq_Gauss-Dirac_Spinor}) to
$G\subset G_{\mathbb{C}}$, and $\Delta_{\lambda}$ is the spherical
harmonic corresponding to $\lambda\in(\Lambda_{+}^{\ast})^{\kappa}$.
The form $\hat{\Lambda}_{M}^{\lambda}$ is equivariantly closed at
$\xi_{\lambda}\in(\mathfrak{t}^{\kappa})_{\mathbb{C}}$:
\[
\left(d-\iota_{(\xi_{\lambda})_{M}}\right)\hat{\Lambda}_{M}^{\lambda}=0.
\]
Proposition \ref{Prop_TTG_pure_spinors_Gauss}-(3) states that $\widehat{\psi}_{G}^{\kappa}$
satisfies the differential equation:
\[
(d+\eta)\left(\Delta_{\lambda}\hat{\psi}_{G}^{\kappa}\right)=\varrho\left(\mathsf{e}^{\kappa}(\xi_{\lambda})\right)\cdot\left(\Delta_{\lambda}\hat{\psi}_{G}^{\kappa}\right),
\]
and since $(\Phi,\omega)$ is a Dirac morphism, taking the pullback
under $\Phi:M\to G\kappa$ of this equation and then multiplying by
$e^{\omega}$ yields $d\hat{\Lambda}_{M}^{\lambda}=\iota_{(\xi_{\lambda})_{M}}\hat{\Lambda}_{M}^{\lambda}$.

Now, by the Berline-Vergne localization formula \cite[Thm.7.13]{[BGV]}:
\begin{equation}
\int_{M}(\hat{\Lambda}_{M}^{\lambda})_{[\mathrm{top}]}=\sum_{Z\subset(\xi_{\lambda})_{M}^{-1}(0)}\int_{Z}\frac{\iota_{Z}^{\ast}\hat{\Lambda}_{M}^{\lambda}(\xi_{\lambda})}{\mathrm{Eul}(\mathcal{N}_{Z},\xi_{\lambda})}.\label{Eq_Localization_Eqn}
\end{equation}

\noindent \textbf{Step 3:} The LHS of equation (\ref{Eq_Localization_Eqn})
is given by:
\[
\int_{M}(\hat{\Lambda}_{M}^{\lambda})_{[\mathrm{top}]}=(\dim V_{\lambda})^{-1}\langle\mathrm{DH}_{\Phi},\tilde{\chi}_{\lambda}^{\kappa}\rangle.
\]
We have $\hat{\Lambda}_{M}^{\lambda}=\Phi^{\ast}(\Delta_{\lambda})e^{-\iota_{(\xi_{\lambda})_{M}}}(e^{\omega}\Phi^{\ast}\psi_{G}^{\kappa})$
by \ref{Prop_TTG_pure_spinors_Gauss}-(2), and since $\exp(-\iota_{(\xi_{\lambda})_{M}})$
does not modify top degree parts, the integrand in the RHS is fact
$\Phi^{\ast}(\Delta_{\lambda})\Lambda_{M}$. Next, by Proposition
\ref{Prop_Twining_Char_L2_Class_Fns}-(4):
\begin{eqnarray*}
\int_{M}(\hat{\Lambda}_{M}^{\lambda})_{[\mathrm{top}]} & = & \int_{M}\Phi^{\ast}(\Delta_{\lambda})\Lambda_{M}=\int_{G}\Delta_{\lambda}(x)d\mathrm{DH}_{\Phi}(x)\\
 & = & \int_{G}\left(\int_{G}\Delta_{\lambda}\left(gx\kappa_{1}(g^{-1})\right)dg\right)d\mathrm{DH}_{\Phi}(x)\\
 & = & (\dim V_{\lambda})^{-1}\int_{G}\tilde{\chi}_{\lambda}^{\kappa}(x)d\mathrm{DH}_{\Phi}(x)=(\dim V_{\lambda})^{-1}\langle\mathrm{DH}_{\Phi},\tilde{\chi}_{\lambda}^{\kappa}\rangle.
\end{eqnarray*}

\noindent \textbf{Step 4:} The integrals in the RHS of equation (\ref{Eq_Localization_Eqn})
reduce to:
\[
\int_{Z}\frac{\iota_{Z}^{\ast}\hat{\Lambda}_{M}^{\lambda}(\xi_{\lambda})}{\mathrm{Eul}(\mathcal{N}_{Z},\xi_{\lambda})}=\int_{Z}\frac{(\Phi_{Z})^{\lambda+\rho}e^{\omega}\Phi_{Z}^{\ast}(\phi_{\kappa})}{\mathrm{Eul}(\mathcal{N}_{Z},\xi_{\lambda})}.
\]
By proposition \ref{Prop_Expressions_Gauss-Dirac_Spinor}, we have
that $(\widehat{\psi}_{G}^{\kappa})_{t}=t^{\rho}(\phi_{\kappa})_{t}$
for all $t\in T$, where $\phi_{\kappa}$ is the form in the statement.
Since $\Delta_{\lambda}(t)=t^{\lambda}$ for all $t\in T$, and since
$\Phi_{Z}$ is $T$-valued, we have: 
\[
\iota_{Z}^{\ast}\hat{\Lambda}_{M}^{\lambda}(\xi_{\lambda})=e^{\omega_{Z}}\Phi_{Z}^{\ast}(\Delta_{\lambda}\widehat{\psi}_{G}^{\kappa})=(\Phi_{Z})^{\lambda+\rho}e^{\omega}\Phi_{Z}^{\ast}(\phi_{\kappa}),
\]
which completes the proof of the theorem.
\end{proof}
In comparison with the Alekseev-Meinrenken-Woodward localization formula
\cite[Thm.5.2]{[AMW00]}, the main difference is the presence of the
form $\phi_{\kappa}\in\Omega(G)$ in equation (\ref{Eq_Fourier_Coeff_Twisted_DH}),
which equals $1$ for $\kappa=1$. A second notable difference is
the structure of the connected components of the vanishing set $(\xi_{\lambda})_{M}^{-1}(0)$,
as the next example shows.
\begin{example}
For $\kappa\in\mathrm{Out}(G)$, let $a\in T^{\kappa}$ be an element
with stabilizer $T^{\kappa}$ under $\mathrm{Ad}_{G}^{\kappa}$, and
consider the twisted conjugacy class $\mathcal{C}=\mathrm{Ad}_{G}^{\kappa}(a)$
with its tq-Hamiltonian structure of Example \ref{Eq_Form_TConj_Class}.
Applying the localization formula (\ref{Eq_Fourier_Coeff_Twisted_DH})
to $\lambda=0$ gives the Liouville (or DH) volume of the twisted
conjugacy class $\mathcal{C}$. Here, the vanishing set of $\xi_{0}=2\pi\mathsf{i}B^{\sharp}(\rho)$
is precisely $\mathcal{C}\cap T$. 

In the case where $\kappa=1$, we have $\mathcal{C}\cap T=\{w\cdot a\}_{w\in W}$.
For nontrivial $\kappa\in\mathrm{Out}(G)$ with $|\kappa|=2$, each
connected component of $\mathcal{C}\cap T$ is a $T_{\kappa}$-orbit
$Z_{w}=T_{\kappa}\cdot(w\cdot a)$ for some $w\in W^{\kappa}$, with
trivial normal bundle $\mathcal{N}_{Z_{w}}\to Z_{w}$ such that \cite[\S 7.1]{[BGV]}:
\[
\mathrm{Eul}(\mathcal{N}_{Z_{w}},2\pi\mathsf{i}\rho)^{-1}=(-1)^{\dim\mathfrak{t}^{\perp}}(-1)^{|w|}\frac{\mathrm{vol}_{G}}{\mathrm{vol}_{T}}.
\]
Using the properties of the orbit root system \cite[\S 2.3]{[Zer18a]}
associated to $(\mathfrak{R},\kappa)$ and \cite[Lem.3.7]{[Zer18a]}:
\begin{align*}
\mathrm{Vol}(\mathcal{C}) & =\sum_{w\in W^{\kappa}}\int_{Z_{w}}\frac{(\Phi_{Z_{w}})^{\rho}e^{\omega_{Z_{w}}}\Phi_{Z_{w}}^{\ast}\left(\phi_{\kappa}\right)}{\mathrm{Eul}\left(\mathcal{N}_{Z_{w}},2\pi\mathsf{i}\rho\right)}=\sum_{w\in W^{\kappa}}\frac{a^{w^{-1}\cdot\rho}}{\mathrm{Eul}\left(\mathcal{N}_{Z_{w}},2\pi\mathsf{i}\rho\right)}\int_{Z_{w}}e^{\omega_{Z_{w}}}\Phi_{Z_{w}}^{\ast}\left(\phi_{\kappa}\right)\\
 & =|T^{\kappa}\cap T_{\kappa}|^{\frac{1}{2}}\Big(\sum_{w\in W^{\kappa}}(-1)^{|w|}a^{w^{-1}\cdot\rho}\Big)\frac{\mathrm{vol}_{G}}{\mathrm{vol}_{T^{\kappa}}}=\big|\mathrm{det}_{(\mathfrak{t}^{\kappa})^{\perp}}(\mathrm{Ad}_{a}\kappa-1)\big|^{\frac{1}{2}}\frac{\mathrm{vol}_{G}}{\mathrm{vol}_{T^{\kappa}}},
\end{align*}
where the term $(\mathrm{vol}_{T^{\kappa}})^{-1}$ comes from $T=T^{\kappa}\cdot T_{\kappa}$
and the fact that $\phi_{\kappa}$ involves $d\mathrm{vol}_{T_{\kappa}}$.
Thus, the Liouville volume of $\mathcal{C}$ coincides with its Riemannian
volume \cite[Prop.4.2]{[AMW02]}.
\end{example}

\subsection{Application to moduli spaces\label{SubSec_Application_TMS}}

By decomposing a connected surface $\Sigma_{h}^{b}$ into $h$ 1-holed
tori and $(b-1)$ cylinders glued along pants $\Sigma_{0}^{3}$, the
moduli space corresponding to a twist $\sigma\in\mathrm{Hom}\left(\Pi,\mathrm{Out}(G)\right)$
can be realized as a fusion product:
\[
M_{\sigma}(\Sigma_{h}^{b},G)=\left(\mathbb{D}_{\varphi_{1}}(G)\circledast\cdots\circledast\mathbb{D}_{\varphi_{h}}(G)\right)\circledast\left(D_{\sigma_{1}}(G)\circledast\cdots\circledast D_{\sigma_{b-1}}(G)\right),
\]
where the doubles $D_{\sigma_{i}}(G)=M_{\sigma_{i}}(\Sigma_{0}^{2},G)$
and $\mathbb{D}_{\varphi_{i}}(G)=M_{\varphi_{i}}(\Sigma_{1}^{1},G)$
are as in Examples \ref{Ex_Moduli_Double} and \ref{Ex_Moduli_Fused_Double}
respectively, for appropriate twists $\sigma_{i}\in\mathrm{Hom}\left(\Pi_{1}(\Sigma_{0}^{2},\{p_{1},p_{2}\}),\mathrm{Out}(G)\right)$
and $\varphi_{i}\in\mathrm{Hom}\left(\Pi_{1}(\Sigma_{1}^{1},\{p\}),\mathrm{Out}(G)\right)$.
From Proposition \ref{Prop_Convolution_DH_Meas_Fusion}, the DH measure
$\mathrm{DH}_{\Phi_{\sigma}}$ associated to $M_{\sigma}(\Sigma_{h}^{b},G)$
is a convolution product of measures $\mathrm{DH}_{D_{\sigma_{i}}(G)}$
and $\mathrm{DH}_{\mathbb{D}_{\varphi_{i}}(G)}$, and it therefore
suffices to determine the Fourier coefficients of the latter to obtain
those of $\mathrm{DH}_{\Phi_{\sigma}}$.

We first establish the following twisted counterparts to \cite[Prop.4.4]{[AMW02]}
and \cite[\S 2.6]{[AMW02]}:
\begin{lem}
For $G$ compact 1-connected and simple, let $\tau,\kappa\in\mathrm{Out}(G)$.\begin{enumerate}

\item In the setup of Example \ref{Ex_Gkappa}, the Liouville measure
of $G\kappa$ coincides with its Riemannian measure as a symmetric
space, that is:
\[
\mathrm{Vol}(G\kappa)=\mathrm{vol}_{G}.
\]

\item The twisted double can be realized as the fusion product:
\[
D_{\sigma}(G)=G\tau\circledast G\kappa.
\]

\end{enumerate}
\end{lem}

\begin{proof}
Let $K=G\rtimes\langle\kappa,\tau\rangle$, and let $\nu$ denote
the involution $K\times K\to K\times K$, $(a,b)\mapsto(b,a)$. Recall
from Example \ref{Ex_Gkappa} that $(G\tau,0,\Phi)$ is the $\nu$-twisted
conjugacy class of $(\tau,\tau^{-1})\in K\times K$ with moment map
$\Phi(g)=(g\tau,(g\tau)^{-1})$.

\noindent (1) This follows from the proof of \cite[Prop.4.4]{[AMW02]},
with the next modifications. Here, $G\tau$ is identified with the
symmetric space $(G\times G)/Z_{(\tau,\tau^{-1})}^{\nu}$, where the
stabilizer of $(\tau,\tau^{-1})$ is given by:
\[
Z_{(\tau,\kappa^{-1})}^{\nu}=\left\{ (\tau(g),g)\in G\times G\ \big|\ g\in G\right\} .
\]
Letting $\mathfrak{c}=(\mathrm{Lie}Z_{(\tau,\tau^{-1})}^{\nu})^{\perp}=\ker(\nu+1)$,
we have by \cite[Prop.4.2]{[AMW02]}:
\[
\mathrm{Vol}(G\tau)=2^{\frac{1}{2}\dim\mathfrak{g}}\frac{\mathrm{vol}_{G}^{2}}{2^{\frac{1}{2}\dim\mathfrak{g}}\mathrm{vol}_{G}}=\mathrm{vol}_{G}=\big|\mathrm{det}_{\mathfrak{c}}\left(\mathrm{Ad}_{(\tau,\tau^{-1})}\nu-1\right)\big|^{\frac{1}{2}}\frac{\mathrm{vol}_{G\times G}}{\mathrm{vol}_{Z_{(\tau,\tau^{-1})}^{\nu}}},
\]
since the induced Riemannian metric on $Z_{(\tau,\tau^{-1})}^{\nu}\subseteq G\times G$
is twice that of $G$.

\noindent (2) For this part, we endow the space $G\kappa$ with the
tq-Hamiltonian structure $(G\kappa,0,\Psi)$, where $\Psi(g)=((g\kappa)^{-1},g\kappa)$,
and keep the structure above for $G\tau$. The moment map $\Phi_{\mathrm{fus}}$
of the fusion product $G\tau\circledast G\kappa$ takes values in
the component $G\tau\kappa\times G\tau^{-1}\kappa^{-1}$ of $K\times K$.
Under our usual identification $G\tau\kappa\times G\tau^{-1}\kappa^{-1}\equiv G\times G$,
the components of $\Phi_{\mathrm{fus}}$ are the same as those of
Example \ref{Ex_Moduli_Double}, and the resulting 2-form $\omega_{\mathrm{fus}}$
is that of $D_{\sigma}(G)$.
\end{proof}
\begin{prop}
Let $\mathrm{DH}_{D_{\sigma}(G)}$ and $\mathrm{DH}_{\mathbb{D}_{\varphi}(G)}$
denote the DH measures of the double $D_{\sigma}(G)$ and its internal
fusion $\mathbb{D}_{\varphi}(G)$ respectively. The Fourier coefficients
are respectively given by: 
\[
\langle\mathrm{DH}_{D_{\sigma}(G)},\tilde{\chi}_{\lambda}^{\tau\kappa}\otimes\tilde{\chi}_{\mu}^{\tau^{-1}\kappa^{-1}}\rangle=\delta_{\lambda\mu}\mathrm{vol}_{G}^{2},\ \ \forall\lambda,\mu\in(\Lambda_{+}^{\ast})^{\kappa}\cap(\Lambda_{+}^{\ast})^{\tau}.
\]
\[
\langle\mathrm{DH}_{\mathbb{D}_{\varphi}(G)},\tilde{\chi}_{\lambda}^{[\tau,\kappa]}\rangle=(\dim V_{\lambda})^{-1}\mathrm{vol}_{G}^{2},\ \ \forall\lambda\in(\Lambda_{+}^{\ast})^{\kappa}\cap(\Lambda_{+}^{\ast})^{\tau}.
\]
\end{prop}

\begin{proof}
For the first equation, we have by the previous lemma that:
\[
\langle\mathrm{DH}_{D_{\sigma}(G)},\tilde{\chi}_{\lambda}^{\tau\kappa}\otimes\tilde{\chi}_{\mu}^{\tau^{-1}\kappa^{-1}}\rangle=\int_{G\times G}\tilde{\chi}_{\lambda}^{\tau\kappa}\left(a\tau(b)\right)\tilde{\chi}_{\mu}^{\tau^{-1}\kappa^{-1}}\left(\tau^{-1}\left(a^{-1}\kappa^{-1}(b^{-1})\right)\right)d\mathrm{vol}_{G\times G}(a,b).
\]
Using the identities:
\[
\overline{\tilde{\chi}_{\mu}^{\kappa\tau}(g)}=\tilde{\chi}_{\mu}^{(\kappa\tau)^{-1}}\left((\kappa\tau)^{-1}(g^{-1})\right),\ \ \tilde{\chi}_{\mu}^{\kappa\tau}=\tau^{\ast}\tilde{\chi}_{\mu}^{\tau\kappa},\ \ (\mathrm{Ad}_{g}^{\tau\kappa})^{\ast}\tilde{\chi}_{\mu}^{\tau\kappa}=\tilde{\chi}_{\mu}^{\tau\kappa},
\]
we have that:
\[
\tilde{\chi}_{\mu}^{\tau^{-1}\kappa^{-1}}\left(\tau^{-1}\left(a^{-1}\kappa^{-1}(b^{-1})\right)\right)=\overline{\tilde{\chi}_{\mu}^{(\kappa\tau)}\left(a\tau(b)\right)}.
\]
Next, since $d\mathrm{vol}_{G}(g)=\mathrm{vol}_{G}dg$, combining
a change of variable, the invariance of $dg$ under translation, and
the orthogonality of twining characters yields:
\[
\langle\mathrm{DH}_{D_{\sigma}(G)},\tilde{\chi}_{\lambda}^{\tau\kappa}\otimes\tilde{\chi}_{\mu}^{\tau^{-1}\kappa^{-1}}\rangle=\mathrm{vol}_{G}^{2}\langle\tilde{\chi}_{\mu}^{(\kappa\tau)},\tilde{\chi}_{\lambda}^{(\kappa\tau)}\rangle=\delta_{\lambda\mu}\mathrm{vol}_{G}^{2}.
\]
Since $\mathbb{D}_{\varphi}(G)$ is the internal fusion of $D_{\sigma}(G)$,
we have by Proposition \ref{Prop_Twining_Char_L2_Class_Fns}-(3):
\begin{align*}
\langle\mathrm{DH}_{\mathbb{D}_{\varphi}(G)},\tilde{\chi}_{\lambda}^{[\tau,\kappa]}\rangle & =\tfrac{1}{\dim V_{\lambda}}\langle\mathrm{Mult}_{\ast}(1\times\tau\kappa)_{\ast}\mathrm{DH}_{D_{\sigma}(G)},\tilde{\chi}_{\lambda}^{\tau\kappa}\ast\tilde{\chi}_{\lambda}^{\tau^{-1}\kappa^{-1}}\rangle\\
 & =\tfrac{1}{\dim V_{\lambda}}\langle\mathrm{DH}_{D_{\sigma}(G)},\tilde{\chi}_{\lambda}^{\tau\kappa}\otimes\tilde{\chi}_{\lambda}^{\tau^{-1}\kappa^{-1}}\rangle=\tfrac{\mathrm{vol}_{G}^{2}}{\dim V_{\lambda}}.\qedhere
\end{align*}
\end{proof}
Combining the above with the remark following Proposition \ref{Prop_Convolution_DH_Meas_Fusion},
the argument of \cite[Prop.4.6]{[AMW02]} leads to the following corollary:
\begin{cor}
Consider the surface $\Sigma=\Sigma_{h}^{1}$ with fundamental groupoid
$\Pi$, let $\tilde{\varphi}_{i}\in\mathrm{Hom}(\Pi,\mathrm{Out}(G))$
be twists for $1\le i\le h$ such that the moment map of $\mathbb{D}_{\tilde{\varphi}_{i}}(G)$
is $G\tau_{i}$-valued, and let $\mathcal{C}_{j}\subseteq G$ be $\kappa_{j}$-twisted
conjugacy classes for $1\le j\le b$. Consider the space:
\[
M=\mathbb{D}_{\tilde{\varphi}_{1}}(G)\circledast\cdots\circledast\mathbb{D}_{\tilde{\varphi}_{h}}(G)\circledast\mathcal{C}_{1}\circledast\cdots\circledast\mathcal{C}_{b},
\]
with moment map:
\[
\Phi:M\to G\tau_{1}\times\cdots G\tau_{h}\times G\kappa_{1}\times\cdots G\kappa_{b}.
\]
For a dominant integral weight $\lambda\in\Lambda_{+}^{\ast}$ that
is invariant under the $\tau_{i}$'s and $\kappa_{j}$'s, the corresponding
Fourier coefficient of $\mathrm{DH}_{\Phi}$ is given by:
\[
\langle\tilde{\chi}_{\lambda}^{(\tau_{1}\cdots\tau_{h}\kappa_{1}\cdots\kappa_{b})},\mathrm{DH}_{\Phi}\rangle=\frac{(\mathrm{vol}_{G})^{2h}}{(\dim V_{\lambda})^{2h+b-1}}\prod_{j=1}^{b}\mathrm{Vol}(\mathcal{C}_{j})\tilde{\chi}_{\lambda}^{\kappa_{j}}(\mathcal{C}_{j}).
\]
\end{cor}

In conclusion, the results of \cite[\S 4.3]{[AMW02]} extend to the
case of twisted moduli spaces in the expected way.

\section{Dirac geometry and twisted conjugation\label{App_Dirac_Geo}}

This section adapts certain key results from \cite{[AMB09]} to the
setup of twisted conjugation, and gathers several technical facts
used in the previous sections. Section \ref{Subsec_Dirac_Geo} reviews
the fundamental concepts of Dirac geometry, namely Dirac structures,
Dirac morphisms and pure spinors. Section \ref{Subsec_Dirac_Str_Pure_Spinors}
summarizes the properties of certain Dirac structures on a Lie group,
and discusses the pure spinors involved in the definition of Liouville
forms and localization. Section \ref{Subsec_Dirac_Fusion} looks at
the Dirac geometry of fusion. We continue with the notation introduced
in sections \ref{SubSec_Twist_Conj} and \ref{Subsec_Twining_Char}.

\subsection{Reminders\label{Subsec_Dirac_Geo}}

\subsubsection{Dirac structures}

The starting point of Dirac geometry is the following concept \cite[\S 2.1]{[AMB09]}:
\begin{defn}
Let $M$ be a manifold with a closed 3-form $\chi\in\Omega^{3}(M)$,
and let $\mathbb{T}M=TM\oplus T^{\ast}M$ be equipped with the split
symmetric bilinear form:
\[
\langle X\oplus\alpha,Y\oplus\beta\rangle=\iota_{X}\beta+\iota_{Y}\alpha,\ \ \forall X\oplus\alpha,Y\oplus\beta\in\Gamma(\mathbb{T}M).
\]
A ($\chi$-twisted) \textbf{Dirac structure} on $M$ is a subbundle
$E\subset\mathbb{T}M$ which is Lagrangian (i.e. $E^{\perp}=E$),
and whose sections are closed under the following Courant bracket
on $\Gamma(\mathbb{T}M)$:
\[
[\![X\oplus\alpha,Y\oplus\beta]\!]_{\chi}=[X,Y]+\mathcal{L}_{X}\beta-d\iota_{Y}\alpha+\iota_{X}\iota_{Y}\chi,\ \ \forall X\oplus\alpha,Y\oplus\beta\in\Gamma(\mathbb{T}M).
\]
\end{defn}

An obvious example of a Dirac structure on a manifold $M$ is $E=TM$
with $\chi=0$. For $\chi\ne0$, non-trivial examples of Dirac structures
are given by graphs of 2-forms $\omega\in\Omega^{2}(M)$ satisfying
$d\omega=\chi$, and by graphs of $\chi$\textit{-twisted Poisson
structures}, i.e. bivectors $\pi\in\mathfrak{X}^{2}(M)$ satisfying
$\frac{1}{2}[\pi,\pi]_{\mathrm{Sch}}+\pi^{\sharp}\chi=0$. Below,
we focus on certain Dirac structures on a Lie group $G$ with $\chi\in\Omega^{3}(G)$
the Cartan 3-form $\eta=\frac{1}{12}[\theta^{L},\theta^{L}]\cdot\theta^{L}$.
These are nicely described in terms of a distinguished trivialization
of $\mathbb{T}G$ \cite[\S 3]{[AMB09]} which we now remind of.

Let $G$ be a Lie group, let $B$ be an $\mathrm{Aut}(G)$-invariant
symmetric non-degenerate bilinear form on $\mathfrak{g}=\mathrm{Lie}(G)$.
Equip $G$ with the action of $G\times G$ given by the map $\mathcal{A}:G\times G\to\mathrm{Diff}G$
given by:
\begin{equation}
\mathcal{A}(a,b)\cdot g=bga^{-1},\ \ \forall a,b,g\in G.\label{Eq_Action_G_Double}
\end{equation}
Letting $\overline{\mathfrak{g}}$ designate $\mathfrak{g}$ equipped
with the bilinear form $-B$, let $\mathfrak{d}:=\mathfrak{g}\oplus\overline{\mathfrak{g}}$
denote the Lie algebra of $G\times G$ with the sum of Lie brackets
and inner product $B_{\mathfrak{d}}=B\oplus-B$. The infinitesimal
action of $\mathfrak{d}$ on $G$ lifts to the trivialization $\mathsf{s}:G\times\mathfrak{d}\to\mathbb{T}G$
given by:
\begin{equation}
\mathsf{s}(\xi\oplus\zeta)|_{g}=(\xi_{g}^{L}-\zeta_{g}^{R})\oplus\tfrac{1}{2}(\theta_{g}^{L}\cdot\xi+\theta_{g}^{R}\cdot\zeta),\ \ \forall(\xi\oplus\zeta)\in\mathfrak{d},g\in G\label{Eq_Triv_map_TTG}
\end{equation}
which is a $G\times G$-equivariant map such that \cite[Prop.3.1]{[AMB09]}:
\[
B_{\mathfrak{d}}(x,y)=\langle\mathsf{s}(x),\mathsf{s}(y)\rangle,\ \ \mathsf{s}([x,y]_{\mathfrak{d}})=[\![\mathsf{s}(x),\mathsf{s}(y)]\!]_{\eta},
\]
for all $x,y\in\mathfrak{d}$. A central property of this map is that
\textit{it associates an $\eta$-twisted Dirac structure $E^{\mathfrak{s}}=\mathsf{s}(G\times\mathfrak{s})$
to any Lagrangian subalgebra} $\mathfrak{s}\subset\mathfrak{d}$ \cite[\S 3.2]{[AMB09]}.
The next examples look at two Dirac structures that are intimately
related to tq-Hamiltonian manifolds.
\begin{example}
\textbf{\label{Ex_Cartan-Dirac_F_kappa}(Cartan-Dirac structure)}
For a fixed automorphism $\kappa\in\mathrm{Aut}(G)$, consider the
$\kappa$-twisted diagonal subalgebra $\mathfrak{g}_{\Delta}^{\kappa}=\{\kappa(\xi)\oplus\xi\ |\ \xi\in\mathfrak{g}\}$
of $\mathfrak{d}$. Its image under the trivialization $\mathsf{s}$
is the $\boldsymbol{\kappa}$\textbf{-twisted Cartan-Dirac structure}
on $G$ \cite[Rk.3.5]{[Mein17]}:
\[
E_{G}^{\kappa}=\left\{ (\kappa(\xi)^{L}-\xi^{R})\oplus\tfrac{1}{2}(\kappa^{-1}\theta^{L}+\theta^{R})\cdot\xi\ \big|\ \xi\in\mathfrak{g}\right\} .
\]
This Dirac structure is the twisted counterpart to the Cartan-Dirac
structure studied in \cite[\S 3.3]{[AMB09]}. For the upcoming discussion,
we also introduce the following Lagrangian complement of $E_{G}^{\kappa}\subset\mathbb{T}G$:
\[
F_{G}^{\kappa}=\left\{ (\kappa(\xi)^{L}+\xi^{R})\oplus\tfrac{1}{2}(\kappa^{-1}\theta^{L}-\theta^{R})\cdot\xi\ \big|\ \xi\in\mathfrak{g}\right\} .
\]
We have that $F_{G}^{\kappa}=\mathsf{s}(G\times\mathfrak{g}_{\Delta^{-}}^{\kappa})$,
where $\mathfrak{g}_{\Delta^{-}}^{\kappa}=\{\kappa(\xi)\oplus-\xi\ |\ \xi\in\mathfrak{g}\}$
is the twisted anti-diagonal in $\mathfrak{d}$. Since $\mathfrak{g}_{\Delta^{-}}^{\kappa}$
is not a subalgebra, $F_{G}^{\kappa}$ is not Courant integrable.
\end{example}

\begin{example}
\textbf{\label{Ex_Gauss-Dirac}(Gauss-Dirac structure)} Suppose that
$G$ is a complex Lie group, and let $\mathfrak{g}=\mathfrak{n}_{+}\oplus\mathfrak{h}\oplus\mathfrak{n}_{-}$
be the triangular decomposition of its Lie algebra. The $\boldsymbol{\kappa}$\textbf{-twisted
Gauss-Dirac structure} $\widehat{F}_{G}^{\kappa}\subset\mathbb{T}G$
is the image under $\mathsf{s}:G\times\mathfrak{d}\to\Gamma(\mathbb{T}G)$
of the Lagrangian subalgebra:
\begin{equation}
\mathfrak{s}^{\kappa}=\left\{ (\xi_{+}+\kappa(\xi_{0}))\oplus(\xi_{-}-\xi_{0})\ \big|\ \xi_{\pm}\in\mathfrak{n}_{\pm},\xi_{0}\in\mathfrak{h}\right\} \subset\mathfrak{d}.\label{Eq_Gauss-Dirac_subalgebra}
\end{equation}
With this Dirac structure, we have an alternative Lagrangian splitting
$\mathbb{T}G=E_{G}^{\kappa}\oplus\widehat{F}_{G}^{\kappa}$. The link
between $F_{G}^{\kappa}$ and $\widehat{F}_{G}^{\kappa}$ is clarified
in the next section.
\end{example}

\subsubsection{Dirac morphisms}
\begin{defn}
For $i=1,2$, let $M_{i}$ be manifolds with closed forms $\chi_{i}\in\Omega^{3}(M_{i})$,
and let $E_{i}\to M_{i}$ be Dirac structures. A (strong) \textbf{Dirac
morphism} $(M_{1},E_{1},\chi_{1})\dashrightarrow(M_{2},E_{2},\chi_{2})$
is a pair $(\Phi,\omega)$, where $\Phi:M_{1}\to M_{2}$ and $\omega\in\Omega^{2}(M_{1})$
are such that:\begin{enumerate}

\item The forms $\chi_{i}$ and $\omega$ satisfy: $\Phi^{\ast}\chi_{2}=\chi_{1}+d\omega$;

\item For all $x\in M_{1}$ and $v_{2}\oplus\alpha_{2}\in E_{2}|_{\Phi(x)}$,
there exists a unique $v_{1}\oplus\alpha_{1}\in E_{1}|_{x}$ such
that:
\[
v_{2}=\Phi_{\ast}|_{x}v_{1},\ \ \Phi^{\ast}\alpha_{2}=\alpha_{1}+\iota_{v_{1}}\omega_{x}.
\]

\end{enumerate} Sections $\zeta_{i}\in\Gamma(E_{i})$ satisfying
the equations in (2) for all $x\in M_{1}$ are called $(\Phi,\omega)$-related,
and are denoted by $\zeta_{1}\sim_{(\Phi,\omega)}\zeta_{2}$.
\end{defn}

Regarding the Dirac morphisms that we will encounter below, it is
useful to recall the following points \cite[\S\S 1.6, 2.2]{[AMB09]}:
\begin{enumerate}

\item The composition of two Dirac morphisms $(\Phi_{1},\omega_{1})$
and $(\Phi_{2},\omega_{2})$ is given by:
\[
(\Phi_{2},\omega_{2})\circ(\Phi_{1},\omega_{1})=(\Phi_{2}\circ\Phi_{1},\omega_{1}+\Phi_{1}^{\ast}\omega_{2}).
\]

\item If the $M_{i}$ are $G$-manifolds and the Dirac structures
are $G$-equivariant, the Dirac morphism $(\Phi,\omega)$ is $G$-equivariant
when $\Phi:M_{1}\to M_{2}$ is $G$-equivariant and $\omega\in\Omega^{2}(M_{1})^{G}$.

\item For a (strong) Dirac morphism $(\Phi,\omega):(M_{1},E_{1},\chi_{1})\dashrightarrow(M_{2},E_{2},\chi_{2})$,
the uniqueness in condition (2) of the definition is equivalent to
the transversality condition that:
\[
\ker(\Phi,\omega)\cap E_{1}=0,
\]
 where $\ker(\Phi,\omega)\subset\mathbb{T}M_{1}$ is the subspace
of elements $X\oplus\alpha\sim_{(\Phi,\omega)}0$ \cite[\S 1.6]{[AMB09]}.

\end{enumerate}

Let $G$ be a Lie group with an automorphism $\kappa\in\mathrm{Aut}(G)$,
and let $E_{G}^{\kappa}\to G$ be the twisted Cartan-Dirac structure
introduced previously. Suppose that $M$ is a $G$-manifold, and let
$(\Phi,\omega):(M,TM,0)\dashrightarrow(G\kappa,E_{G}^{\kappa},\eta)$
be a $G$-equivariant Dirac morphism. Unwinding the definition, this
is equivalent to saying that the 2-form $\omega\in\Omega^{2}(M)^{G}$
and the equivariant map $\Phi:M\to G\kappa$ satisfy:\begin{itemize}

\item[(i)] $d\omega=\Phi^{\ast}\eta$;

\item[(ii)] $\iota_{\xi_{M}}\omega=\frac{1}{2}\Phi^{\ast}(\kappa^{-1}\theta^{L}+\theta^{R})\cdot\xi$,
for all $\xi\in\mathfrak{g}$;

\item[(iii)] $\ker\omega_{x}\cap\ker(\Phi_{\ast}|_{x})=0$, for all
$x\in M$;

\end{itemize} which are precisely the axioms making $(M,\omega,\Phi)$
a tq-Hamiltonian $G$-space. The case of untwisted q-Hamiltonian manifolds
($\kappa=1$) is addressed in \cite[\S 5.1]{[AMB09]}, and this formulation
in terms of Dirac morphisms was first discussed by Burzstyn and Crainic
in \cite{[BC05]}.

\subsubsection{Pure spinors}

Recall that for a vector space $V$ over $\mathbb{R}$ or $\mathbb{C}$
\cite[\S 1]{[AMB09]}, the pairing $\langle\cdot,\cdot\rangle$ between
vectors and covectors on $\mathbb{V}=V\oplus V^{\ast}$ is a non-degenerate
symmetric bilinear form of split signature. Taking the Clifford algebra
$\mathrm{Cl}(\mathbb{V})$ with product:
\[
xy+yx=\langle x,y\rangle1,\ \ \forall x,y\in\mathbb{V},
\]
the exterior algebra $\wedge V^{\ast}$ gives a spinor module for
$\mathrm{Cl}(\mathbb{V})$, where the isomorphism $\varrho:\mathrm{Cl}(\mathbb{V})\to\mathrm{End}(\wedge V^{\ast})$
is such that:
\begin{equation}
\varrho(v\oplus\alpha)\phi=\iota_{v}\phi+\alpha\wedge\phi,\ \ \forall v\oplus\alpha\in\mathbb{V},\phi\in\wedge V^{\ast}.\label{Eq_Action_TTM_Clif}
\end{equation}
An element $\phi\in\wedge V^{\ast}\smallsetminus\{0\}$ is called
a \textbf{pure spinor} if the subspace:
\[
N_{\phi}=\left\{ x\in\mathbb{V}\ \big|\ \varrho(x)\phi=0\right\} 
\]
is Lagrangian in $\mathbb{V}$. Furthermore, any Lagrangian subspace
$E\subset\mathbb{V}$ arises as a subspace $N_{\phi}$ for an appropriate
$\phi\in\wedge V^{\ast}\smallsetminus\{0\}$ (see \cite[III.1.9]{[ChevSpnrBk]}
or \cite[Eq.(6)]{[AMB09]} for a precise expression).

The constructions of the previous paragraph extend to the smooth category.
For a manifold $M$, the bundle $\wedge T^{\ast}M$ is a spinor module
for $\mathrm{Cl}(\mathbb{T}M)$, and any Lagrangian subbundle $E\subset\mathbb{T}M$
can be locally described by pure spinors in $\Omega(M)$. If $E\subset\mathbb{T}M$
is $\chi$-twisted Dirac structure on $M$ defined by a pure spinor
$\phi\in\Omega(M)$, then the latter satifies a differential equation
of the form:
\[
(d+\chi)\phi=\varrho(\sigma^{E})\phi,
\]
where $\sigma^{E}\in\Gamma(E^{\ast})$ is a unique section depending
on $\phi$ (see \cite[Prop.2.2]{[AMB09]} for more details).

Our main interest in pure spinors defining Dirac structures is their
relation to volume forms. As a consequence of \cite[Prop.1.15-(c)]{[AMB09]},
we have the following fact:
\begin{prop}
\label{Prop_DiracMorph_VolForm}Let $M$ be a manifold, let $(N,E,\chi)$
be a Dirac structure on a manifold $N$, and let $F\subseteq\mathbb{T}N$
be a Lagrangian complement to $E$ defined by a pure spinor $\psi_{F}\in\Omega(N)$.
If $(\Phi,\omega):(M,TM,0)\dashrightarrow(N,F,\chi)$ is a strong
Dirac morphism, then the backward image:
\[
(\Phi,\omega)^{!}F=\left\{ x\in\mathbb{T}M\ \big|\ \exists y\in F\text{ such that }x\sim_{(\Phi,\omega)}y\right\} 
\]
is a Lagrangian subbundle transverse to $TM$ in $\mathbb{T}M$, and
is defined by the pure spinor:
\[
e^{\omega}\Phi^{\ast}\psi_{F}\in\Omega(M).
\]
Furthermore, the top degree part $\left(e^{\omega}\Phi^{\ast}\psi_{F}\right)_{[\mathrm{top}]}\in\Omega^{[\mathrm{top}]}(M)$
is a volume form on $M$.
\end{prop}

In the context of a tq-Hamiltonian $G$-manifold $(M,\omega,\Phi)$,
this proposition states that the Dirac morphism $(\Phi,\omega):(M,TM,0)\dashrightarrow(G\kappa,E_{G}^{\kappa},\eta)$
and the Lagrangian splitting $\mathbb{T}G=E_{G}^{\kappa}\oplus F_{G}^{\kappa}$
give rise to a volume form $\left(e^{\omega}\Phi^{\ast}\psi_{G}^{\kappa}\right)_{[\mathrm{top}]}$
on $M$, with $\psi_{G}^{\kappa}\in\Omega(G)$ a pure spinor defining
the Lagrangian subbundle $F_{G}^{\kappa}$. The next subsection addresses
the construction and the properties of a distinguished $\psi_{G}^{\kappa}$,
as well as its links to a certain pure spinor defining the Gauss-Dirac
structure $\widehat{F}_{G}^{\kappa}$.

\subsection{Pure spinors on Lie groups\label{Subsec_Dirac_Str_Pure_Spinors}}

\subsubsection{The trivialization $G\times\mathrm{Cl}(\mathfrak{g})\protect\cong\wedge T^{\ast}G$}

The bilinear form $B$ on $\mathfrak{g}$ induces an isometry between
$(\mathfrak{g}\oplus\mathfrak{g}^{\ast},\langle\cdot,\cdot\rangle)$
and $(\mathfrak{d},B_{\mathfrak{d}})$ \cite[\S 4.1]{[AMB09]}. This
identification realizes $\mathrm{Cl}(\mathfrak{g})$ as a spinor module
for $\mathrm{Cl}(\mathfrak{d})$, where the isomorphism $\varrho^{\mathrm{Cl}}:\mathrm{Cl}(\mathfrak{d})\to\mathrm{End}(\mathrm{Cl}(\mathfrak{g}))$
is given by:
\begin{equation}
\varrho^{\mathrm{Cl}}(\xi\oplus\zeta)x=\xi\cdot x-(-1)^{|x||\zeta|}x\cdot\zeta,\ \ \forall\xi\oplus\zeta\in\mathfrak{d},x\in\mathrm{Cl}(\mathfrak{g}).\label{Eq_Action_Cl-g_Clif}
\end{equation}
Next, by fixing a lift $\tau:G\to\mathrm{Pin}(\mathfrak{g})$ of $\mathrm{Ad}:G\to\mathrm{O}(\mathfrak{g})$
such that for any orthonormal basis $\{v_{i}\}\subset\mathfrak{g}$:
\[
\tau(\xi)=\tfrac{d}{dt}\left(\tau(e^{t\xi})\right)|_{t=0}=-\sum_{i>j}B([\xi,v_{i}],v_{j})\in\mathrm{Cl}(\mathfrak{g}),\ \ \forall\xi\in\mathfrak{g},
\]
we obtain the following action of $G\times G$ on $\mathrm{Cl}(\mathfrak{g})$:
\begin{equation}
\mathcal{A}^{\mathrm{Cl}}(a,b)x=\tau(a)x\tau(b^{-1}),\ \ \forall(a,b)\in G\times G,x\in\mathrm{Cl}(\mathfrak{g}),\label{Eq_Action_Cl-g_Double}
\end{equation}
which extends the adjoint action of $G$ on $\mathrm{Cl}(\mathfrak{g})$
\cite[eq.(76)]{[AMB09]}.

Now fix a generator $\mathrm{vol}_{\mathfrak{g}}\in\wedge^{[\mathrm{top}]}\mathfrak{g}$
with dual $\mathrm{vol}_{\mathfrak{g}^{\ast}}\in\wedge^{[\mathrm{top}]}\mathfrak{g}^{\ast}$
such that $\iota(\mathrm{vol}_{\mathfrak{g}^{\ast}}^{\intercal})\mathrm{vol}_{\mathfrak{g}}=1$.
In light of the previous paragraph, define the trivialization $\mathcal{R}:G\times\mathrm{Cl}(\mathfrak{g})\to\wedge T^{\ast}G$
by the prescription:
\begin{equation}
\mathcal{R}(x)|_{g}=(q\circ\star)^{-1}\left(x\tau(g)\right),\ \ \forall(g,x)\in G\times\mathrm{Cl}(\mathfrak{g}),\label{Eq_Triv_Map_R}
\end{equation}
where $q:\wedge\mathfrak{g}\to\mathrm{Cl}(\mathfrak{g})$ denotes
the quantization map, and where $\star$ denotes the isomorphism:
\[
\star:\wedge\mathfrak{g}^{\ast}\longrightarrow\wedge\mathfrak{g},\ \phi\longmapsto\iota(\phi)\mathrm{vol}_{\mathfrak{g}}.
\]
The upshot of these constructions is that $\mathcal{R}$ gives a transparent
description of the pure spinors defining Lagrangian subbundles of
$\mathbb{T}G$. The properties that we need are the following \cite[Prop.4.2]{[AMB09]}:
\begin{prop}
\label{Prop_Properties_mathcal_R}Consider the trivializations $\mathsf{s}:G\times\mathfrak{d}\to\mathbb{T}G$
and $\mathcal{R}:G\times\mathrm{Cl}(\mathfrak{g})\to\wedge T^{\ast}G$
given by equations (\ref{Eq_Triv_map_TTG}) and (\ref{Eq_Triv_Map_R})
respectively. One has that:\begin{enumerate}

\item $\mathcal{R}$ intertwines the Clifford actions (\ref{Eq_Action_TTM_Clif})
and (\ref{Eq_Action_Cl-g_Clif}):
\begin{equation}
\mathcal{R}\left(\varrho^{\mathrm{Cl}}(\xi\oplus\zeta)\cdot x\right)=\varrho(\mathsf{s}(\xi\oplus\zeta))\cdot\mathcal{R}(x),\ \ \forall\xi\oplus\zeta\in\mathfrak{d},x\in\mathrm{Cl}(\mathfrak{g}).\label{Eq_R_Cliff}
\end{equation}

\item $\mathcal{R}$ intertwines the actions (\ref{Eq_Action_G_Double})
and (\ref{Eq_Action_Cl-g_Double}) of $G\times G$:
\begin{equation}
\left(\mathcal{A}\left(a^{-1},b^{-1}\right)^{\ast}\mathcal{R}(x)\right)_{g}=(-1)^{|a|(|g|+|x|)}\mathcal{R}\left(\mathcal{A}^{\mathrm{Cl}}(a,b)\cdot x\right)_{g},\ \ \forall(a,b)\in G\times G,x\in\mathrm{Cl}(\mathfrak{g}),\label{Eq_R_Equivariance}
\end{equation}
where $|g|:=|\tau(g)|$ for all $g\in G$.

\end{enumerate}
\end{prop}

\subsubsection{The Cartan-Dirac spinor and Liouville form}

For the remainder of this section, we assume that we have a fixed
generator $\mathrm{vol}_{\mathfrak{g}}\in\wedge^{[\mathrm{top}]}\mathfrak{g}$,
as well as a fixed automorphism $\kappa\in\mathrm{Aut}(G)$. Using
the same notation for the differential $\kappa\in\mathrm{Aut}(\mathfrak{g})$,
the latter can be viewed as an element of $\mathrm{O}(\mathfrak{g})$,
and we fix a lift $\tau(\kappa)\in\mathrm{Pin}(\mathfrak{g})$, such
that:
\[
\tau(\kappa)x\tau(\kappa^{-1})=(-1)^{|\tau(\kappa)||x|}\kappa(x),\ \ \forall x\in\mathrm{Cl}(\mathfrak{g}).
\]

As seen in the previous section, we have a Lagrangian splitting $\mathfrak{d}=\mathfrak{g}_{\Delta}^{\kappa}\oplus\mathfrak{g}_{\Delta^{-}}^{\kappa}$,
where the $\kappa$-twisted diagonal and anti-diagonal are given by:
\[
\mathfrak{g}_{\Delta}^{\kappa}=\left\{ \kappa(\xi)\oplus\xi\ \big|\ \xi\in\mathfrak{g}\right\} ,\ \ \mathfrak{g}_{\Delta^{-}}^{\kappa}=\left\{ \kappa(\xi)\oplus-\xi\ \big|\ \xi\in\mathfrak{g}\right\} .
\]
By direct computation, it is easily verified that these Lagrangians
are respectively defined by the following pure spinors in $\mathrm{Cl}(\mathfrak{g})$:
\begin{equation}
x_{\Delta}^{\kappa}:=\tau(\kappa),\ \ x_{\Delta^{-}}^{\kappa}:=q(\mathrm{vol}_{\mathfrak{g}})\tau(\kappa).\label{Eq_Spinors_Twisted_Diag}
\end{equation}
By equation (\ref{Eq_R_Cliff}) in Prop. \ref{Prop_Properties_mathcal_R},
the Lagrangian subbundles $E_{G}^{\kappa}=\mathsf{s}(G\times\mathfrak{g}_{\Delta}^{\kappa})$
and $F_{G}^{\kappa}=\mathsf{s}(G\times\mathfrak{g}_{\Delta^{-}}^{\kappa})$
are thus defined by the pure spinors:
\begin{equation}
(\phi_{G}^{\kappa})_{g}=\mathcal{R}\left(x_{\Delta}^{\kappa}\tau(\kappa^{-1}(g))\right),\ \ (\psi_{G}^{\kappa})_{g}=\mathcal{R}\left(x_{\Delta^{-}}^{\kappa}\tau(\kappa^{-1}(g))\right),\label{Eq_Cartan-Dirac_Spinor_Psi_kappa}
\end{equation}
in $\Omega(G)$. By equation (\ref{Eq_R_Equivariance}), these forms
behave as follows under $\kappa$-twisted conjugation:
\[
\left((\mathrm{Ad}_{a}^{\kappa})^{\ast}\phi_{G}^{\kappa}\right)|_{g}=(-1)^{|a|(|g|+|x_{\Delta}^{\kappa}|)}(\phi_{G}^{\kappa})_{g},\ \ \left((\mathrm{Ad}_{a}^{\kappa})^{\ast}\psi_{G}^{\kappa}\right)|_{g}=(-1)^{|a|(|g|+|x_{\Delta^{-}}^{\kappa}|)}(\psi_{G}^{\kappa})_{g},\ \ \forall a,g\in G,
\]
In the particular case that $G$ is connected, \textit{the pure spinors
$\phi_{G}^{\kappa},\psi_{G}^{\kappa}\in\Omega(G)$ are $\mathrm{Ad}_{G}^{\kappa}$-invariant}.

Proposition \ref{Prop_DiracMorph_VolForm} and the last paragraph
justify the following definition:
\begin{defn}
Let $(M,\omega,\Phi)$ be a $G\kappa$-valued tq-Hamiltonian manifold.
The associated \textbf{Liouville form} on $M$ is defined to be the
$G$-invariant volume form given by:
\[
\Lambda_{M}:=\left(e^{\omega}\Phi^{\ast}\psi_{G}^{\kappa}\right)_{[\mathrm{top}]}\in\Omega^{[\mathrm{top}]}(M)^{G}.
\]
\end{defn}

\begin{rem}
Although the trivialization $\mathcal{R}:G\times\mathrm{Cl}(\mathfrak{g})\to\wedge T^{\ast}G$
depends on the choice of the generator $\mathrm{vol}_{\mathfrak{g}}\in\wedge^{[\mathrm{top}]}\mathfrak{g}$,
the Liouville form $\Lambda_{M}$ is unaffected by this choice \cite[Rk.4.5-(a)]{[AMB09]}.
Indeed, replacing $\mathrm{vol}_{\mathfrak{g}}$ by $\lambda\mathrm{vol}_{\mathfrak{g}}$
with $\lambda\ne0$, the isomorphism $\star^{-1}:\wedge\mathfrak{g}\to\wedge\mathfrak{g}^{\ast}$
replaces $\mathcal{R}$ by $\lambda^{-1}\mathcal{R}$, but this does
not affect $\psi_{G}^{\kappa}=\mathcal{R}\big(q(\mathrm{vol}_{\mathfrak{g}})\tau(\kappa)\big)$.
\end{rem}

For the sake of completeness, we give explicit formulas for the pure
spinors $x_{\Delta}^{\kappa},x_{\Delta^{-}}^{\kappa}\in\mathrm{Cl}(\mathfrak{g})$.
Suppose for the remainder of this part that $G$ is a compact 1-connected
simple Lie group, and that $\kappa\in\mathrm{Out}(G)$ is induced
by a non-trivial Dynkin diagram automorphism. To formulate the next
proposition, we employ the following notation:\begin{itemize}

\item[-] For any subspace $\mathfrak{a}\subset\mathfrak{g}$, denote
by $\mathrm{vol}_{\mathfrak{a}}\in\wedge^{[\mathrm{top}]}\mathfrak{a}$
the generator induced by the orientation on $\mathfrak{g}$, with
dual form $\mathrm{vol}_{\mathfrak{a}^{\ast}}\in\wedge^{[\mathrm{top}]}\mathfrak{a}^{\ast}$
such that $\iota(\mathrm{vol}_{\mathfrak{a}}^{\intercal})\mathrm{vol}_{\mathfrak{a}^{\ast}}=1$.
For a subgroup $A\subseteq G$ with $\mathfrak{a}=\mathrm{Lie}A$,
denote by $d\mathrm{vol}_{A}\in\Omega^{[\mathrm{top]}}(A)$ the Riemannian
volume form such that $(d\mathrm{vol}_{A})|_{e}=\mathrm{vol}_{\mathfrak{a}^{\ast}}$.

\item[-] Recall that $T_{\kappa}\subset T$ is the image of $T\to T$,
$t\mapsto t\kappa(t^{-1})$, with Lie algebra $\mathfrak{t}_{\kappa}\subset\mathfrak{t}$.
We have that $|T^{\kappa}\cap T_{\kappa}|=2^{\dim\mathfrak{t}_{\kappa}}$
for $|\kappa|=2$, and that $|T^{\kappa}\cap T_{\kappa}|=3$ for $|\kappa|=3$.

\item[-] For $G$ of type $D_{4}$ and $|\kappa|=3$, fix an orthonormal
basis $\{a_{i},b_{i}\}_{i=1}^{7}\in(\mathfrak{g}^{\kappa})^{\perp}$
such that:
\[
\ker\left(\kappa-e^{\pm\mathsf{i}\frac{2\pi}{3}}\right)=\mathrm{Span}_{\mathbb{C}}\{-a_{i}\pm b_{i}\}_{i=1}^{7},\ \ \mathfrak{t}_{\kappa}=\mathrm{Span}_{\mathbb{R}}\{a_{1},b_{1}\}.
\]

\end{itemize} We can now state:
\begin{prop}
\label{Prop_Pure_Spinors_kappa_diagonals}Let $G$ be compact 1-connected
and simple, and let $\kappa\in\mathrm{Out}(G)$. A pure spinor $x_{\Delta}^{\kappa}\in\mathrm{Cl}(\mathfrak{g})$
defining the twisted diagonal $\mathfrak{g}_{\Delta}^{\kappa}\subset\mathfrak{d}$
is given by:
\begin{eqnarray*}
x_{\Delta}^{\kappa} & = & \begin{cases}
2^{\frac{1}{2}\dim(\mathfrak{g}^{\kappa})^{\perp}}q(\mathrm{vol}_{(\mathfrak{g}^{\kappa})^{\perp}}^{\intercal}), & \mbox{for }|\kappa|=2,\\
2^{-\frac{1}{2}\dim(\mathfrak{g}^{\kappa})^{\perp}}q\left(e^{-2\sqrt{3}\sum_{j}a_{j}\wedge b_{j}}\right), & \mbox{for }|\kappa|=3.
\end{cases}
\end{eqnarray*}
A pure spinor $x_{\Delta^{-}}^{\kappa}\in\mathrm{Cl}(\mathfrak{g})$
defining the twisted anti-diagonal $\mathfrak{g}_{\Delta^{-}}^{\kappa}\subset\mathfrak{d}$
is given by:

\[
x_{\Delta^{-}}^{\kappa}=\begin{cases}
2^{-\frac{1}{2}\dim(\mathfrak{g}^{\kappa})^{\perp}}q(\mathrm{vol}_{\mathfrak{g}^{\kappa}}), & \text{for }|\kappa|=2;\\
(\tfrac{\sqrt{3}}{4})^{\frac{1}{2}\dim(\mathfrak{g}^{\kappa})^{\perp}}q(e^{\frac{2}{\sqrt{3}}\sum_{j}a_{j}\wedge b_{j}}\mathrm{vol}_{\mathfrak{g}^{\kappa}}), & \text{for }|\kappa|=3.
\end{cases}
\]
\end{prop}

\noindent\emph{Outline of proof.} For $x_{\Delta}^{\kappa}=\tau(\kappa)$,
it suffices to directly check that the given expression yields a lift
of the differential $\kappa\in\mathrm{O}(\mathfrak{g})$ to $\mathrm{Pin}(\mathfrak{g})$,
i.e. that:
\[
\tau(\kappa)^{\intercal}\tau(\kappa)=1;\ \text{and}\ \tau(\kappa)\xi\tau(\kappa^{-1})=(-1)^{\dim(\mathfrak{g}^{\kappa})^{\perp}}\kappa(\xi),\ \ \forall\xi\in\mathfrak{g}.
\]
This is easily done with an orthonormal basis of $\mathfrak{g}$ when
$|\kappa|=2$, while for $|\kappa|=3$, the exponential is re-written
as a product:
\[
x_{\Delta}^{\kappa}=2^{-7}q(e^{-2\sqrt{3}\sum_{j}a_{j}\wedge b_{j}})=2^{-7}\prod_{i=1}^{7}(1-2\sqrt{3}a_{i}b_{i}).
\]
For the product $x_{\Delta^{-}}^{\kappa}=q(\mathrm{vol}_{\mathfrak{g}})\tau(\kappa)$,
one uses $q(\mathrm{vol}_{\mathfrak{g}})=q(\mathrm{vol}_{\mathfrak{g}^{\kappa}})q(\mathrm{vol}_{(\mathfrak{g}^{\kappa})^{\perp}})$,
as well as $q(\mathrm{vol}_{(\mathfrak{g}^{\kappa})^{\perp}})=\prod_{i}a_{i}b_{i}$
when $|\kappa|=3$.\qed
\begin{rem}
This proposition allows one to determine $\Lambda_{M}$ and $\psi_{G}^{\kappa}$
explicitly in certain cases. For instance, if $|\kappa|=2$, the approach
in the proof of \cite[Thm.3.1]{[AMW00]} leads to:
\[
(\psi_{G}^{\kappa})_{g}=\pm|T^{\kappa}\cap T_{\kappa}|^{-\frac{1}{2}}\cdot|\mathrm{det}_{\mathfrak{t}^{\perp}}^{\frac{1}{2}}\left(\tfrac{\mathrm{Ad}_{g}\kappa+1}{2}\right)|e^{-\frac{1}{2}B\left(\tfrac{\mathrm{Ad}_{g}\kappa-1}{\mathrm{Ad}_{g}\kappa+1}\theta_{g}^{L},\theta_{g}^{L}\right)}\wedge(d\mathrm{vol}_{T_{\kappa}})_{g},
\]
on the subset of $G$ where $(\mathrm{Ad}_{g}\kappa+1)$ is invertible.
One obtains a similar but more cumbersome expression when $|\kappa|=3$.
With $\kappa=1$, one recovers \cite[Prop.4.6]{[AMB09]}.
\end{rem}

\subsubsection{The Gauss-Dirac spinor }

Continuing with the notation above, let $G_{\mathbb{C}}$ be the complexification
of $G$, and let $\mathfrak{g}_{\mathbb{C}}=\mathfrak{n}_{+}\oplus\mathfrak{h}\oplus\mathfrak{n}_{-}$
be the triangular decomposition of its Lie algebra. The twisted Gauss-Dirac
structure is the image $\widehat{F}_{G}^{\kappa}=\mathsf{s}(G\times\mathfrak{s}^{\kappa})$,
where:
\[
\mathfrak{s}^{\kappa}=\left\{ (\xi_{+}+\kappa(\xi_{0}))\oplus(\xi_{-}-\xi_{0})\ \big|\ \xi_{\pm}\in\mathfrak{n}_{\pm},\xi_{0}\in\mathfrak{h}\right\} .
\]
After clarifying the relation between the subalgebra $\mathfrak{s}^{\kappa}\subset\mathfrak{d}$
and the twisted anti-diagonal $(\mathfrak{g}_{\mathbb{C}})_{\Delta^{-}}^{\kappa}$,
we describe the \textbf{Gauss-Dirac spinor} $\widehat{\psi}_{G_{\mathbb{C}}}^{\kappa}\in\Omega(G_{\mathbb{C}})$
defining $\widehat{F}_{G}^{\kappa}$, and discuss its properties relevant
to localization.

Let $\mathfrak{r}=\frac{1}{2}\sum_{\alpha\in\mathfrak{R}_{+}}e_{-\alpha}\wedge e_{\alpha}$
denote the \textbf{classical r-matrix} \cite[\S 3.6]{[AMB09]} , where
$\{e_{\alpha}\}_{\alpha\in\mathfrak{R}}\subset\mathfrak{n}_{+}\oplus\mathfrak{n}_{-}$
is a basis of root vectors satisfying $B(e_{\alpha},e_{\beta})=\delta_{\beta,-\alpha}$
and $e_{-\alpha}=\overline{e_{\alpha}}$. Letting $A^{-\mathfrak{r}}\in\mathrm{O}(\mathfrak{g}_{\mathbb{C}}\oplus\mathfrak{g}_{\mathbb{C}}^{\ast})$
denote the map $0\oplus\alpha\mapsto\iota_{\alpha}\mathfrak{r}\oplus\alpha$
\cite[\S 1.5]{[AMB09]}, the graph of $\mathfrak{r}\in\wedge^{2}\mathfrak{g}_{\mathbb{C}}$
in $\mathfrak{g}_{\mathbb{C}}\oplus\mathfrak{g}_{\mathbb{C}}^{\ast}$
is given by $\mathrm{Gr}_{\mathfrak{r}}=A^{-\mathfrak{r}}(0\oplus\mathfrak{g}_{\mathbb{C}}^{\ast})$.
Considering then the isometry:
\begin{equation}
\mathfrak{i}^{\kappa}:\mathfrak{g}_{\mathbb{C}}\oplus\mathfrak{g}_{\mathbb{C}}^{\ast}\longrightarrow\mathfrak{d}_{\mathbb{C}},\ \ \mathfrak{i}^{\kappa}(\xi\oplus\alpha)=\kappa\left(\xi+\tfrac{1}{2}B^{\sharp}\alpha\right)\oplus\left(\xi-\tfrac{1}{2}B^{\sharp}\alpha\right),\label{Eq_i_kappa_isometry}
\end{equation}
we have that \cite[Lem.3.16]{[AMB09]}:
\[
\mathfrak{s}^{\kappa}=\mathfrak{i}^{\kappa}(\mathrm{Gr}_{\mathfrak{r}}),\ (\mathfrak{g}_{\mathbb{C}})_{\Delta^{-}}^{\kappa}=\mathfrak{i}^{\kappa}(0\oplus\mathfrak{g}_{\mathbb{C}}^{\ast}),\ \text{and}\ (\mathfrak{g}_{\mathbb{C}})_{\Delta}^{\kappa}=\mathfrak{i}^{\kappa}(\mathfrak{g}_{\mathbb{C}}\oplus0).
\]
At the level of pure spinors defining Lagrangian subspaces, the present
setup leads to:
\begin{prop}
A pure spinor defining the Lagrangian subalgebra $\mathfrak{s}^{\kappa}\subset\mathfrak{d}_{\mathbb{C}}$
is given by:
\begin{equation}
x_{\mathfrak{s}^{\kappa}}:=\varrho^{\mathrm{Cl}}\left(e^{-(q\circ\mathfrak{i}^{\kappa})(\mathfrak{r})}\right)x_{\Delta^{-}}^{\kappa}\in\mathrm{Cl}(\mathfrak{g}_{\mathbb{C}}),\label{Eq_Spinor_Gauss-Dirac_Cl-g}
\end{equation}
where $x_{\Delta^{-}}^{\kappa}\in\mathrm{Cl}(\mathfrak{g}_{\mathbb{C}})$
is the pure spinor defining $(\mathfrak{g}_{\mathbb{C}})_{\Delta^{-}}^{\kappa}$.
For $\mathfrak{g}_{\mathbb{C}}$ simple and $\kappa\in\mathrm{Aut}(\mathfrak{g}_{\mathbb{C}})$
induced by a Dynkin diagram automorphism, one has that:
\begin{equation}
x_{\mathfrak{s}^{\kappa}}=\begin{cases}
2^{-\frac{1}{2}\dim\mathfrak{t}_{\kappa}}\left(\prod_{\alpha\in\mathfrak{R}_{+}}e_{\alpha}e_{-\alpha}\right)q(\mathrm{vol}_{\mathfrak{t}^{\kappa}}), & \text{for }|\kappa|=2;\\
\tfrac{\sqrt{3}}{4}\left(\prod_{\alpha\in\mathfrak{R}_{+}}e_{\alpha}e_{-\alpha}\right)q\left(e^{\frac{2}{\sqrt{3}}\mathrm{vol}_{\mathfrak{t}_{\kappa}}}\mathrm{vol}_{\mathfrak{t}^{\kappa}}\right), & \text{for }|\kappa|=3.
\end{cases}\label{Eq_Spinor_Gauss-Dirac_Cl-g_Dynkin}
\end{equation}
\end{prop}

\noindent\emph{Outline of proof.} Since $\mathfrak{g}_{\mathbb{C}}^{\ast}\subset\mathfrak{g}_{\mathbb{C}}\oplus\mathfrak{g}_{\mathbb{C}}^{\ast}$
is defined by $\mathrm{vol}_{\mathfrak{g}^{\ast}}\in\wedge^{[\mathrm{top}]}\mathfrak{g}_{\mathbb{C}}^{\ast}$,
and since $\mathrm{Gr}_{\mathfrak{r}}=A^{-\mathfrak{r}}(\mathfrak{g}_{\mathbb{C}}^{\ast})$,
the latter is defined by the pure spinor:
\[
\varrho(\widetilde{A}^{-\mathfrak{r}})\mathrm{vol}_{\mathfrak{g}^{\ast}}=\exp(-\iota(\mathfrak{r}))\mathrm{vol}_{\mathfrak{g}^{\ast}}\in\wedge\mathfrak{g}_{\mathbb{C}}^{\ast},
\]
where $\widetilde{A}^{-\mathfrak{r}}\in\mathrm{Pin}(\mathfrak{g}\oplus\mathfrak{g}^{\ast})_{\mathbb{C}}$
is the lift of $A^{-\mathfrak{r}}$ in the representation $\varrho:\mathrm{Cl}(\mathfrak{g}\oplus\mathfrak{g}^{\ast})_{\mathbb{C}}\to\mathrm{End}(\wedge\mathfrak{g}_{\mathbb{C}}^{\ast})$.
Next, using the isometry $\mathfrak{i}^{\kappa}$ of equation (\ref{Eq_i_kappa_isometry})
and the quantization maps, one constructs the unique isomorphism of
spinor modules $R^{\kappa}:\wedge\mathfrak{g}_{\mathbb{C}}^{\ast}\to\mathrm{Cl}(\mathfrak{g}_{\mathbb{C}})$
mapping $\mathrm{vol}_{\mathfrak{g}^{\ast}}$ to $x_{\Delta^{-}}^{\kappa}$,
and intertwining the Clifford actions of $\mathrm{Cl}(\mathfrak{g}\oplus\mathfrak{g}^{\ast})_{\mathbb{C}}\cong\mathrm{Cl}(\mathfrak{d}_{\mathbb{C}})$
via:
\[
R^{\kappa}\left(\varrho(\xi\oplus\alpha)\phi\right)=\varrho^{\mathrm{Cl}}\left(\mathfrak{i^{\kappa}}(\xi\oplus\alpha)\right)R^{\kappa}(\phi),\ \ \forall\phi\in\wedge\mathfrak{g}_{\mathbb{C}}^{\ast},\xi\oplus\alpha\in(\mathfrak{g}\oplus\mathfrak{g}^{\ast})_{\mathbb{C}}.
\]
Under the isomorphism $R^{\kappa}$, the counterpart of the lift $\widetilde{A}^{-\mathfrak{r}}$
is the element $e^{-(q\circ\mathfrak{i}^{\kappa})(\mathfrak{r})}\in\mathrm{Cl}(\mathfrak{d}_{\mathbb{C}})$,
which yields equation (\ref{Eq_Spinor_Gauss-Dirac_Cl-g}).

In the case of $\mathfrak{g}_{\mathbb{C}}$ simple and $\kappa\in\mathrm{Aut}(\mathfrak{g}_{\mathbb{C}})$
induced by a diagram automorphism, equation (\ref{Eq_Spinor_Gauss-Dirac_Cl-g_Dynkin})
follows from a direct computation. One uses the expressions of $x_{\Delta^{-}}^{\kappa}$
in Proposition \ref{Prop_Pure_Spinors_kappa_diagonals}, along with
the identity:
\[
\varrho^{\mathrm{Cl}}\left(\mathfrak{i^{\kappa}}(\xi\oplus0)\right)\circ q=q\circ\left(\iota\left[B^{\flat}(\tfrac{\kappa+1}{2}\xi)\right]+(\kappa-1)\xi\right),\ \ \forall\xi\in\mathfrak{g}_{\mathbb{C}},
\]
to compute the product with the exponential $\varrho^{\mathrm{Cl}}\left(e^{-(q\circ\mathfrak{i}^{\kappa})(\mathfrak{r})}\right)=\exp\varrho^{\mathrm{Cl}}\left(-(q\circ\mathfrak{i}^{\kappa})(\mathfrak{r})\right)$.\qed

\hfill
\begin{rem}
We omit the detailed construction of $R^{\kappa}:\wedge\mathfrak{g}^{\ast}\to\mathrm{Cl}(\mathfrak{g})$
above, since it would require a disgression that is not used elsewhere,
with additional notation. The idea is as follows. Using $B_{\mathfrak{d}}$
to identify $\mathfrak{g}_{\Delta^{-}}^{\kappa}$ with $(\mathfrak{g}_{\Delta}^{\kappa})^{\ast}$
in the splitting $\mathfrak{d}=\mathfrak{g}_{\Delta}^{\kappa}\oplus\mathfrak{g}_{\Delta^{-}}^{\kappa}$,
one constructs a first isomorphism of spinor modules $R_{1}^{\kappa}:\wedge\mathfrak{g}_{\Delta^{-}}^{\kappa}\to\mathrm{Cl}(\mathfrak{g})$.
On the other hand, the isometry of equation (\ref{Eq_i_kappa_isometry})
extends to isomorphisms $\mathfrak{i}^{\kappa}:\wedge(\mathfrak{g}\oplus\mathfrak{g}^{\ast})\to\wedge\mathfrak{d}$
and $\mathfrak{i}^{\kappa}:\wedge\mathfrak{g}^{\ast}\to\wedge\mathfrak{g}_{\Delta^{-}}^{\kappa}$,
and using the quantization maps for $\mathrm{Cl}(\mathfrak{g}\oplus\mathfrak{g}^{\ast})$
and $\mathrm{Cl}(\mathfrak{d})$, one constructs a second isomorphism
of spinor modules $R_{2}^{\kappa}:\wedge\mathfrak{g}^{\ast}\to\wedge\mathfrak{g}_{\Delta^{-}}^{\kappa}$
such that $R^{\kappa}=R_{1}^{\kappa}\circ R_{2}^{\kappa}$ gives the
desired properties.
\end{rem}

Turning to Lagrangian subbundles of $\mathbb{T}G_{\mathbb{C}}$, define
the trivialization $\mathsf{e}^{\kappa}:G\times\mathfrak{g}_{\mathbb{C}}\longrightarrow E_{G}^{\kappa}$
by:
\begin{equation}
\mathsf{e}^{\kappa}(\xi)|_{g}=\mathsf{s}\left(\mathfrak{i}^{\kappa}(\xi\oplus0)\right)|_{g},\ \ \forall(g,\xi)\in G\times\mathfrak{g}_{\mathbb{C}}.\label{Eq_e_kappa_trivialization}
\end{equation}
The twisted Gauss-Dirac structure $\widehat{F}_{G_{\mathbb{C}}}^{\kappa}$
arises as \textit{the image of the Lagrangian complement} $F_{G_{\mathbb{C}}}^{\kappa}$
\textit{under the orthogonal transformation} $A^{-\mathsf{e}^{\kappa}(\mathfrak{r})}\in\Gamma(\mathrm{O}(\mathbb{T}G_{\mathbb{C}}))$
\cite[Cor.3.17]{[AMB09]}. By Proposition \ref{Prop_Properties_mathcal_R}
and equation (\ref{Eq_Spinor_Gauss-Dirac_Cl-g}), the twisted Gauss-Dirac
spinor is hence given by:
\begin{equation}
(\widehat{\psi}_{G_{\mathbb{C}}}^{\kappa})_{g}:=\mathcal{R}\left(x_{\mathfrak{s}^{\kappa}}\tau\left(\kappa^{-1}(g)\right)\right),\ \ \forall g\in G_{\mathbb{C}}.\label{Eq_Gauss-Dirac_Spinor}
\end{equation}
Its main properties are as follows:
\begin{prop}
\label{Prop_TTG_pure_spinors_Gauss}Suppose $G$ is compact 1-connected
and simple with complexification $G_{\mathbb{C}}$, and let $\kappa\in\mathrm{Out}(G)$
be non-trivial. The twisted Gauss-Dirac spinor $\widehat{\psi}_{G_{\mathbb{C}}}^{\kappa}\in\Omega(G_{\mathbb{C}})$
satisfies the following properties:
\begin{enumerate}
\item \textbf{Invariance properties:} For all $a_{\pm}\in\exp(\mathfrak{n}_{\pm})$
and $t\in T_{\mathbb{C}}$, one has that: 
\[
\begin{array}{ccccc}
R_{a_{+}}^{\ast}\widehat{\psi}_{G_{\mathbb{C}}}^{\kappa} & = & L_{a_{-}}^{\ast}\widehat{\psi}_{G_{\mathbb{C}}}^{\kappa} & = & \widehat{\psi}_{G_{\mathbb{C}}}^{\kappa},\\
R_{t}^{\ast}\widehat{\psi}_{G_{\mathbb{C}}}^{\kappa} & = & L_{t^{-1}}^{\ast}\widehat{\psi}_{G_{\mathbb{C}}}^{\kappa} & = & t^{\rho}\widehat{\psi}_{G_{\mathbb{C}}}^{\kappa}.
\end{array}
\]
\item \textbf{Transformation by r-matrix:} The pure spinors $\widehat{\psi}_{G_{\mathbb{C}}}^{\kappa}$
and $\psi_{G_{\mathbb{C}}}^{\kappa}$ are related by the equation:
\[
\widehat{\psi}_{G_{\mathbb{C}}}^{\kappa}=\varrho\left(\exp\left(-\mathsf{e^{\kappa}}(\mathfrak{r})\right)\right)\cdot\psi_{G_{\mathbb{C}}}^{\kappa}.
\]
\item \textbf{Differential equations:} For any $\kappa$-invariant dominant
weight $\lambda\in(\Lambda_{+}^{\ast})^{\kappa}$, the spinor $\Delta_{\lambda}\widehat{\psi}_{G_{\mathbb{C}}}^{\kappa}$
satisfies the differential equation: 
\[
(d+\eta)\Delta_{\lambda}\widehat{\psi}_{G_{\mathbb{C}}}^{\kappa}=\varrho\big(\mathsf{e}^{\kappa}\big(2\pi iB^{\sharp}(\lambda+\rho)\big)\big)\cdot\Delta_{\lambda}\widehat{\psi}_{G_{\mathbb{C}}}^{\kappa},
\]
where $\Delta_{\lambda}$ is the spherical harmonic of equation (\ref{Eq_Spherical_Harmonic_lambda}),
and $\rho\in(\mathfrak{t}^{\kappa})^{\ast}$ the half-sum of positive
roots of $G$.
\end{enumerate}
\end{prop}

\noindent\emph{Outline of proof.} The invariance properties follow
from Proposition \ref{Prop_Properties_mathcal_R}-(3), while the equation
in (2) follows from Proposition \ref{Prop_Properties_mathcal_R}-(1)
and equation (\ref{Eq_Spinor_Gauss-Dirac_Cl-g}). The proof for the
differential equation in (3) is the same as the proof of \cite[Prop.4.18]{[AMB09]}.\qed

\hfill

For the purposes of Duistermaat-Heckman localization, it is useful
in practice to have explicit expressions for $(\widehat{\psi}_{G_{\mathbb{C}}}^{\kappa})_{t}$
with $t\in T_{\mathbb{C}}$. Using the basis $\{e_{\alpha}\}_{\alpha\in\mathfrak{R}}\subset\mathfrak{n}_{+}\oplus\mathfrak{n}_{-}$
introduced above, let $(e^{\pm\alpha})^{L}:=B(\theta^{L},e_{\mp\alpha})$
for all $\alpha\in\mathfrak{R}_{+}$, and define the following 2-form
on $G_{\mathbb{C}}$:
\begin{equation}
\varpi:=\tfrac{1}{4}\sum_{\alpha\in\mathfrak{R}_{+}}(e^{\alpha})^{L}\wedge(e^{-\alpha})^{L}\in\Omega^{2}(G_{\mathbb{C}}).\label{Eq_varpi_Gauss_cell}
\end{equation}
Combining equation (\ref{Eq_Spinor_Gauss-Dirac_Cl-g_Dynkin}) with
Proposition \ref{Prop_TTG_pure_spinors_Gauss}-(1), and using the
equation for $\tau(t)$ given in the proof of \cite[Prop.4.1]{[AMW00]},
a direct computation of $\mathcal{R}(x_{\mathfrak{s}^{\kappa}}\tau(t))=(\star^{-1}\circ q^{-1})\left(x_{\mathfrak{s}^{\kappa}}\tau(t)\right)$
yields:
\begin{prop}
\label{Prop_Expressions_Gauss-Dirac_Spinor}With the assumptions of
Proposition \ref{Prop_TTG_pure_spinors_Gauss}, we have for all $t\in T_{\mathbb{C}}$
that:
\[
(\widehat{\psi}_{G_{\mathbb{C}}}^{\kappa})_{t}=\begin{cases}
(-1)^{\frac{1}{2}\dim(\mathfrak{t}^{\perp})^{\kappa}}|T^{\kappa}\cap T_{\kappa}|^{-\frac{1}{2}}\cdot t^{\rho}\left(\exp(\varpi)\wedge d\mathrm{vol}_{T_{\kappa}}\right)_{t}, & \text{for }|\kappa|=2;\\
\tfrac{1}{2}t^{\rho}\exp\left(\varpi+\tfrac{\sqrt{3}}{2}d\mathrm{vol}_{T_{\kappa}}\right)_{t}, & \text{for }|\kappa|=3.
\end{cases}
\]
\end{prop}

\subsection{Dirac geometry of fusion\label{Subsec_Dirac_Fusion}}

Given a Lie group $G$, let $\theta^{L,i}$ (resp. $\theta^{R,i}$)
denote the pullback of $\theta^{L}$ (resp. $\theta^{R}$) to the
$i$th component of $G\times G$, let $\mathrm{Mult}:G\times G\to G$
denote the multiplication map, and let $\mathrm{Inv}:G\to G$ denote
inversion. Define the following forms on $G\times G$:
\[
\varsigma:=-\tfrac{1}{2}B(\theta^{L,1},\theta^{R.2})\in\Omega^{2}(G\times G),\ \eta_{G\times G}=\eta^{1}+\eta^{2}\in\Omega^{3}(G\times G),
\]
where $\eta_{G\times G}$ denotes the Cartan 3-form with $\eta^{i}=\frac{1}{12}B([\theta^{L,i},\theta^{L,i}],\theta^{L,i})$.

In the untwisted theory of q-Hamiltonian manifolds, the properties
of the fusion operations are immediate consequences of the fact that:
\[
(\mathrm{Mult},\varsigma):\left(G\times G,E_{G}^{1}\oplus E_{G}^{2},\eta_{G\times G}\right)\dashrightarrow(G,E_{G},\eta)
\]
is a strong Dirac morphism. The purpose of the present section is
to formulate the main results of \cite[\S\S 3.4,4.3]{[AMB09]} in
the context of twisted conjugation. Conceptually, the quickest way
of doing so is to apply the theory developed in \cite[\S\S 3.4,4.3]{[AMB09]}
to the disconnected group $K=G\rtimes\Gamma$, where $\Gamma\subset\mathrm{Aut}(G)$
is a finitely generated subgroup. From this standpoint, the right
multiplication map by $\kappa^{-1}\in\Gamma$ identifies the component
$G\kappa\subset K$ with the group $G$ equipped with the action $\mathrm{Ad}_{G}^{\kappa}$.
In particular, we get explicit equivariant isomorphisms $E_{G}^{\kappa}\cong E_{K}|_{G\kappa}$
and $F_{G}^{\kappa}\cong F_{K}|_{G\kappa}$ which we implicitly use
below.

Given $\kappa_{1},\kappa_{2},\kappa\in\mathrm{Aut}(G)$, let $G\kappa$
denote the group $G$ equipped with the action $\mathrm{Ad}_{G}^{\kappa}$,
and consider the $G$-equivariant maps:
\begin{equation}
\mathrm{Mult}^{\kappa_{1}}:G\kappa_{1}\times G\kappa_{2}\to G\kappa_{1}\kappa_{2},\ \ \mathrm{Mult}^{\kappa_{1}}(a,b)=\left(\mathrm{Mult}\circ(1\times\kappa_{1})\right)(a,b);\label{Eq_Mult_kappa}
\end{equation}
\begin{equation}
\mathrm{Inv}^{\kappa}:G\kappa\longrightarrow G\kappa^{-1},\ \ \mathrm{Inv}^{\kappa}(a)=\left(\kappa^{-1}\circ\mathrm{Inv}\right)(a);\label{Eq_Inv_kappa}
\end{equation}
as well as the 2-form:
\begin{equation}
\varsigma^{\kappa}:=(1\times\kappa)^{\ast}\varsigma\in\Omega^{2}(G\times G).\label{Eq_varsigma_kappa}
\end{equation}
Modifying the proof \cite[Thm.3.9]{[AMB09]} to take into account
the effect of twisting automorphisms, it is easy to establish:
\begin{prop}
\label{Prop_Twisted_Mult_Inv}Let $G$ be a Lie group with $\kappa,\kappa_{1},\kappa_{2}\in\mathrm{Aut}(G)$.
With the notation of this section, the multiplication map extends
to a Dirac morphism:
\[
(\mathrm{Mult}^{\kappa_{1}},\varsigma^{\kappa_{1}}):(G\kappa_{1}\times G\kappa_{2},E_{G}^{\kappa_{1},1}\oplus E_{G}^{\kappa_{2},2},\eta_{G\times G})\dasharrow(G\kappa_{1}\kappa_{2},E_{G}^{\kappa_{1}\kappa_{2}},\eta),
\]
where in terms of the map in equation (\ref{Eq_e_kappa_trivialization}):
\[
E_{G}^{\kappa_{1},1}\oplus E_{G}^{\kappa_{2},2}=\left\{ \mathsf{e}^{\kappa_{1},1}(\xi)+\mathsf{e}^{\kappa_{2},2}(\zeta)\ \big|\ \xi,\zeta\in\mathfrak{g}\right\} .
\]
Similarly, the inversion map extends to a Dirac morphism:
\[
(\mathrm{Inv}^{\kappa},0):(G\kappa,E_{G}^{\kappa},\eta)\dasharrow(G\kappa^{-1},(E_{G}^{\kappa^{-1}})^{\intercal},-\eta),
\]
where $(E_{G}^{\kappa^{-1}})^{\intercal}=\left\{ X\oplus-\alpha\ \big|\ X\oplus\alpha\in E_{G}^{\kappa^{-1}}\right\} .$
\end{prop}

\begin{rem}
\label{Rk_tq-Ham_Fusion_Inversion}Recall that the composition of
Dirac morphisms $(\Phi_{i},\omega_{i})$ ($i=1,2$) is given by:
\[
(\Phi_{2},\omega_{2})\circ(\Phi_{1},\omega_{1})=(\Phi_{2}\circ\Phi_{1},\omega_{1}+\Phi_{1}^{\ast}\omega_{2}).
\]
\begin{enumerate}

\item With the notation of Proposition \ref{Prop_Def_Internal_Fusion},
let $(M,\omega,\Phi)$ be a $G\kappa_{1}\times G\kappa_{2}\times H\tau$-valued
tq-Hamiltonian manifold. Proposition \ref{Prop_Def_Internal_Fusion}
follows from the fact that:
\[
(\Phi_{\mathrm{fus}},\omega_{\mathrm{fus}})=(\mathrm{Mult}^{\kappa_{1}},\varsigma^{\kappa_{1}})\circ(\Phi,\omega)
\]
is a strong Dirac morphism by the previous proposition. Proposition
\ref{Prop_Def_Fusion} is obtained as the special case where $H=\{e\}$
and $(M,\omega,\Phi)=(M_{1}\times M_{2},\omega_{1}+\omega_{2},\Phi_{1}\times\Phi_{2})$,
where $(M_{i},\omega_{i},\Phi_{i})$ are $G\kappa_{i}$-valued tq-Hamiltonian
manifolds.

\item Similarly, the last proposition implies Proposition \ref{Prop_Inversion}
since the composition:
\[
(\Phi^{-},-\omega)=(\mathrm{Inv}^{\kappa},0)\circ(\Phi,\omega)
\]
gives a strong Dirac morphism for any $G\kappa$-valued tq-Hamiltonian
manifold $(M,\omega,\Phi)$.

\end{enumerate}
\end{rem}

The next topic we address is the Liouville form of a fusion product.
Regarding the pure spinor $\psi_{G}^{\kappa}\in\Omega(G)$ defining
the Lagrangian complement $F_{G}^{\kappa}\subset\mathbb{T}G$, we
have the following generalization of \cite[Thm.4.9]{[AMB09]}:
\begin{prop}
\label{Prop_Pullback_Mult_Pure_Spinor}Let $G$ be a Lie group, and
$\kappa_{1},\kappa_{2}\in\mathrm{Aut}(G)$ two automorphisms. The
pullback of $\psi_{G}^{\kappa_{1}\kappa_{2}}\in\Omega(G)$ under the
map $\mathrm{Mult}^{\kappa_{1}}:G\times G\rightarrow G$ satisfies
the equation:
\begin{equation}
\exp(\varsigma^{\kappa_{1}})(\mathrm{Mult}^{\kappa_{1}})^{\ast}(\psi_{G}^{\kappa_{1}\kappa_{2}})=\varrho\left(\exp\left(-\hat{\gamma}\right)\right)\cdot\left(\psi_{G}^{\kappa_{1},1}\otimes\psi_{G}^{\kappa_{2},2}\right),\label{Eq_Pullback_Mult_Pure_Spinor}
\end{equation}
where for any orthonormal basis $\{v_{i}\}\subset\mathfrak{g}$, the
bivector $\hat{\gamma}\in\Gamma\left(\wedge^{2}(E_{G}^{\kappa_{1},1}\oplus E_{G}^{\kappa_{2},2})\right)$
is given by:
\[
\hat{\gamma}=\tfrac{1}{2}\sum_{i}\mathsf{e}^{\kappa_{1},1}(v_{i})\wedge\mathsf{e}^{\kappa_{2},2}(v_{i}).
\]
\end{prop}

\noindent\emph{Outline of proof.} With the notation of the proposition,
let $K=G\rtimes\langle\kappa_{1},\kappa_{2}\rangle$. Applying \cite[Thm.4.9]{[AMB09]}
to the disconnected group $K$ yields:
\begin{equation}
e^{\varsigma_{K}}\mathrm{Mult}_{K\times K}^{\ast}\psi_{K}=\varrho\Big(\exp(-\tfrac{1}{2}\sum_{i}\mathsf{e}_{K}^{1}(v_{i})\wedge\mathsf{e}_{K}^{2}(v_{i}))\Big)\cdot\left(\psi_{K}^{1}\otimes\psi_{K}^{2}\right),\label{Eq_ABM09_Thm4.9}
\end{equation}
where $\psi_{K}=\mathcal{R}\left(q(\mathrm{vol}_{\mathfrak{g}})\right)$
defines the Lagrangian complement of the untwisted Cartan-Dirac structure
$E_{K}\subset\mathbb{T}K$, and $\mathsf{e}_{K}:\mathfrak{g}\to E_{K}$
is the untwisted analogue to the trivialization in equation (\ref{Eq_e_kappa_trivialization}).
The equation in the statement is then obtained by first restricting
equation (\ref{Eq_ABM09_Thm4.9}) to the outer component $G\kappa_{1}\times G\kappa_{2}\subset K\times K$,
and then identifying $G\kappa_{1}\times G\kappa_{2}$ with $G\times G$
using the right multiplication maps by $\kappa_{i}^{-1}$. \qed

\hfill

We can now state:
\begin{cor}
\label{Cor_Fusion_prod_Liouville}For $i=1,2$, let $(M_{i},\omega_{i},\Phi_{i})$
be $G\kappa_{i}$-valued tq-Hamiltonian $G$-spaces. The Liouville
form of the fusion product $(M_{1}\circledast M_{2},\omega_{\mathrm{fus}},\Phi_{\mathrm{fus}})$
is given by:
\[
\Lambda_{M_{1}\circledast M_{2}}=\Lambda_{M_{1}}\otimes\Lambda_{M_{2}}\in\Omega^{[\mathrm{top}]}(M_{1}\times M_{2})^{G}.
\]
\end{cor}

\begin{proof}
Similarly to \cite[Prop.5.15]{[AMB09]}, the previous proposition
gives the computation:
\[
\begin{array}{ccc}
e^{\omega_{\mathrm{fus}}}\Phi_{\mathrm{fus}}^{\ast}\psi_{G}^{\kappa_{1}\kappa_{2}} & = & e^{\omega_{1}+\omega_{2}}(\Phi_{1}\times\Phi_{2})^{\ast}\left(e^{\varsigma^{\kappa_{1}}}(\mathrm{Mult}^{\kappa_{1}})^{\ast}(\psi_{G}^{\kappa_{1}\kappa_{2}})\right)\\
 & = & e^{\omega_{1}+\omega_{2}}(\Phi_{1}\times\Phi_{2})^{\ast}\left(\varrho(e^{-\hat{\gamma}})\cdot(\psi_{G}^{\kappa_{1},1}\otimes\psi_{G}^{\kappa_{2},2})\right)\\
 & = & \exp\left(-\iota\left(\tfrac{1}{2}\sum_{i}(v_{i})_{M_{1}}\wedge(v_{i})_{M_{2}}\right)\right)\left(e^{\omega_{1}}\Phi_{1}^{\ast}\psi_{G}^{\kappa_{1},1}\otimes e^{\omega_{2}}\Phi_{2}^{\ast}\psi_{G}^{\kappa_{2},2}\right).
\end{array}
\]
The claim follows by taking the top degree part on both sides, since
in the RHS, the term $\Lambda_{M_{1}}\otimes\Lambda_{M_{2}}$ is unaffected
by the exponential of the interior product.
\end{proof}
\bibliographystyle{amsplain}
\bibliography{TMDH}

\end{document}